\documentclass[11pt,reqno]{amsart}

\usepackage[colorlinks,linkcolor={blue},
citecolor={blue},urlcolor={blue}]{hyperref}

\usepackage{mathrsfs}
\usepackage{bbm}
\usepackage{amsfonts}
\usepackage{}
\usepackage{pifont}
\usepackage[shortlabels]{enumitem}
\usepackage{empheq}
\usepackage{amsfonts}
\usepackage{amsmath, amsrefs, amsthm, amssymb}
\usepackage[active]{srcltx}		
\usepackage{pdfsync}					
\usepackage{graphicx}
\usepackage [latin1]{inputenc}
\usepackage{bbm}
\usepackage{amsfonts}
\usepackage{amsthm}
\usepackage{graphicx}
\usepackage{framed}
\usepackage{multicol,multienum}
\usepackage{graphicx} 
\usepackage{booktabs} 
\usepackage{mathrsfs}
\usepackage{tcolorbox}
\usepackage{color}
\usepackage{xcolor}

\newtheorem{theorem}{Theorem}[section]

\newtheorem{lemma}[theorem]{Lemma}

\theoremstyle{definition}

\newtheorem{remark}[theorem]{Remark}

\numberwithin{equation}{section}

\begin{document}

\title [Keller-Segel-Fluid model in $\mathbb{R}^2$]{On the Keller-Segel models interacting with a stochastically forced incompressible viscous flow in $\mathbb{R}^2$} 

\author{Lei Zhang}
\address{School of Mathematics and Statistics, Hubei Key Laboratory of Engineering Modeling  and Scientific Computing, Huazhong University of Science and Technology,  Wuhan 430074, Hubei, P.R. China.}
\email{lzhang890701@163.com; lei\_zhang@hust.edu.cn (L. Zhang)}

\author{Bin Liu}
\address{School of Mathematics and Statistics, Hubei Key Laboratory of Engineering Modeling  and Scientific Computing, Huazhong University of Science and Technology,  Wuhan 430074, Hubei, P.R. China.}
\email{binliu@mail.hust.edu.cn (B. Liu)}


\date{\today}

\begin{abstract}
This paper considers the Keller-Segel model coupled to stochastic Navier-Stokes equations  (KS-SNS, for short), which describes the dynamics of oxygen and bacteria densities evolving within a stochastically forced 2D incompressible viscous flow.  Our main goal is to investigate the existence and uniqueness of global solutions (strong in the probabilistic sense and weak in the PDE sense) to the KS-SNS system. A novel approximate KS-SNS system with proper regularization and cut-off operators in $H^s(\mathbb{R}^2)$ is introduced, and the existence of approximate solution is proved by some a priori uniform bounds and a careful analysis on the approximation scheme. Under appropriate assumptions, two types of stochastic entropy-energy inequalities that seem to be new in their forms are derived, which together with the Prohorov theorem and Jakubowski-Skorokhod theorem enables us to show that the sequence of approximate solutions converges to a global martingale weak solution. In addition, when $\chi(\cdot)\equiv \textrm{const.}>0$, we prove that the solution is pathwise unique, and hence by the Yamada-Wantanabe theorem that the KS-SNS system admits a unique global pathwise weak solution.
\end{abstract}

\maketitle
\tableofcontents
\section{Introduction}\label{sec1}
\subsection{Statement of the problem}

Chemotaxis refers to the directional movement of cells, such as bacteria, in response to chemical signals. A notable example is that bacteria often swim towards areas with higher concentrations of oxygen for survival.  One of the most renowned models in chemotaxis is the Keller-Segel (KS) model, which was developed by Keller and Segel \cite{keller1970,keller1971}. It has since become one of the most extensively studied models in mathematical biology. The interplay between cells and the surrounding fluid, where chemical substances are consumed, has been acknowledged in \cite{fujikawa1989fractal,dombrowski2004self,tuval2005}. These studies confirm that the density of bacteria and chemoattractants change with the motion of fluid. Consequently, the velocity field of fluid is influenced by both moving bacteria and external body forces. To describe such a coupled biological phenomena, Tuval et al. \cite{tuval2005} introduced the deterministic Keller-Segel-Navier-Stokes (KS-NS) system. The main goal of this paper is to study the Cauchy problem for the Keller-Segel model of consumption type coupled with an incompressible fluid described by the stochastic Navier-Stokes equations (KS-SNS) (cf.
\cites{tuval2005,zhai20202d}):
\begin{equation}\label{KS-SNS}
\left\{
\begin{aligned}
&\mathrm{d}  n +u\cdot \nabla n  \mathrm{d} t = d_1\Delta n  \mathrm{d} t-  \textrm{div} \left(n\chi(c)\nabla c\right) \mathrm{d} t , \\
&\mathrm{d}  c+ u\cdot \nabla c   \mathrm{d} t =d_2 \Delta c \mathrm{d} t-n \kappa(c) \mathrm{d} t ,\\
&\mathrm{d}  u+ \left((u\cdot \nabla) u+\nabla P\right) \mathrm{d} t =  d_3 \Delta u \mathrm{d} t + n\nabla \phi \mathrm{d} t+ f(t,u) \mathrm{d}  W,\\
&  \textrm{div}   u  =0 , \\
& n|_{t=0}=n_0,~c|_{t=0}=c_0,~ u|_{t=0}=u_0,
\end{aligned}
\right.
\end{equation}
for all  $(t,x)\in \mathbb{R}^{+}\times\mathbb{R}^2$. The unknowns are the scalar functions $n=n(t, x)$, $c=c(t, x)$, $ P=P(t, x): \mathbb{R}^{+} \times \mathbb{R}^2 \rightarrow \mathbb{R}^{+}$ and the two-dimensional vector field $u=u(t, x): \mathbb{R}^{+} \times \mathbb{R}^2 \rightarrow \mathbb{R}^2$, which represent the density of cells, the chemical sensitivity, the pressure and the velocity of fluid, respectively. The positive constants $d_i$, $i=1,2,3$, are the diffusion coefficients for cells, substrate and fluid. The function $\chi(c)$ denotes the chemotactic sensitivity and $\kappa(c)$ describes the consumption rate of the substrate. The term  $n\nabla \phi$ characterizes the external agency  exerted by bacteria on fluid via a time-independent gravitational potential $\phi=\phi(x)$. The triple $(n_0,c_0,u_0)$ represents the initial data satisfying proper regularity conditions which will be described later.

The deterministic Keller-Segel-Navier-Stokes system (by taking $f(t,u)\equiv 0$ in \eqref{KS-SNS}) was initially proposed by Tuval, et al. \cite{tuval2005} (see also \cites{fujikawa1989fractal,dombrowski2004self}) to describe the interaction between the bacterial populations of consumption type and the surrounding fluid, which reveals the complex facets of the spatio-temporal behavior in colonies of the aerobic species \emph{Bacillus subtilis} when suspended in sessile water drops. The density of bacteria and the evolution of chemical substrates are changing over time corresponding to the flow of liquid environment, and conversely, the dynamic behavior of fluid is affected by certain external body force, such as the gravity of bacteria, centrifugal, electric or magnetic forces and uncertain external force surrounding the fluid \cites{duan2010global,zhai20202d}.

In recent years, in order to get a better understanding of its important applications in the biomathematics \cites{tuval2005,hillen2009user,arumugam2021keller}, the mathematical analysis of KS-NS system such as well-posedness, asymptotic behavior et al. for the KS-NS system of consumption type have been extensively studied in both bounded and unbounded domains from the PDEs' point of view. Among others we would like to mention the following incomplete references which are related to the present work. For the results concerning the KS-NS system in bounded domains, see for example Lorz \cite{lorz2010coupled}, Winkler \cites{winkler2012global,winkler2012global,winkler2016global}, Black and Winkler \cite{black2022global}, Ding and Lankeit \cite{ding2022generalized} and so on. Recently, Winkler \cite{winkler2017far} proved that after some relaxation time, the global weak solutions constructed in \cite{winkler2016global} enjoys further regularity properties and thereby complies with the concept of eventual energy solutions. Moreover, in \cite{winkler2022does}, the possibility for singularities to  weak energy solutions occur on small time-scales was shown to arise only on a sets of measure zero.

Note that the aforementioned work are focused on bounded domains, while when people investigate the dynamic behavior of bacteria and chemicals evolving within large scale spacial regions, such as the lakes and oceans, it will be beneficial to approximately suppose that the domain that the bacteria living in  is unbounded, for example the whole plane $\mathbb{R}^2$ or the physical space $\mathbb{R}^3$. In fact, there already have some interesting results along this direction, and we refer to the works by Duan, Lorz and Markowich \cite{duan2010global}, Liu  and  Lorz \cite{liu2011coupled}, Chae, Kang  and  Lee \cite{chae2014global}, Duan, Li and Xiang \cite{duan2017global}, Kang  and  Kim \cite{kang2017existence}, Zhang and Zheng \cites{zhang2014global,zhang2021global}, Diebou Yomgne \cite{diebou2021well} and the references cited therein. A more recent advance in $\mathbb{R}^d $ or $\mathbb{T}^d $ comes from Jeong and Kang \cite{jeong2022well}, where the authors investigated the local well-posedness and blow-up phenomena in Sobolev spaces for both partially inviscid and fully inviscid KS-NS system by performing a new weighted Gagliardo-Nirenberg-Sobolev inequality.

Physically speaking, incorporating stochastic effects is crucial in creating mathematical models for complex phenomena in science that involve uncertainty. For instance, the evolution of viscous fluids is usually not only affected by the external force $n\nabla \phi$ caused by bacteria, but also by random sources from the environment. The presence of randomness can significantly impact the overall evolution of the viscous fluid \cite{flandoli2008introduction,breit2018stochastically}. As a matter of fact, numerous studies have been conducted on the stochastic Navier-Stokes equations, as evidenced in \cite{bensoussan1995stochastic,flandoli1995martingale,brzezniak20132d,breit2018local,
chen2019martingale,hofmanova2019non,chen2022sharp} and their cited references. Due to the widespread applications of random fluctuations in hydrodynamics, developing a stochastic theory for the KS-NS system coupled with perturbed momentum equations by random forces is essential. In this paper, we assume that the viscous flow described by Navier-Stokes equations are inevitably affected, besides the external force $n\nabla \phi$ stemming from the bacteria, also by some random factors in surrounding environment. Therefore, it will be more felicitously assuming that the incompressible viscous flow is perturbed by a stochastic external force, especially in the form of $f(t,u)\mathrm{d}  W$, in \eqref{KS-SNS}$_3$. More precisely, we denote by $\{W(t)\}_{t\geq 0}$ an $U$-valued cylindrical Wiener process defined on a fixed stochastic basis $(\Omega,\mathcal {F},\{\mathcal {F}_t\}_{t>0},\mathbb{P})$ with complete right-continuous filtration, and $W(t)$ is given by the formal expansion
\begin{equation}\label{1..2}
\begin{split}
W(t )=\sum_{j\geq 1}W^j(t )e_j,~~\mathbb{P}\textrm{-a.s.},
 \end{split}
\end{equation}
where $\{W^j(t )\}_{k\geq1}$ is a family of mutually independent $\mathcal {F}_t$-adapted real-valued standard Wiener processes, and $\{e_j\}_{j\geq1}$ is a complete orthonormal basis in the separable Hilbert space $U$. To make sense of the series \eqref{1..2}, one can consider a larger auxiliary space $U_0$ via
$
 U_0:=\{v=\sum_{j\geq 1}\alpha_j e_j;~\sum_{j\geq 1} \alpha_j^2/j^2  <\infty\} \supset U,
$
which is endowed with the norm $\|v\|_{U_0}^2=\sum_{k\geq 1} \alpha_j^2/j^2$, for any $v=\sum_{j\geq 1} \alpha_j e_j$. Note that the embedding from $U$ into $ U_0$ is Hilbert-Schmidt \cite{da2014stochastic}, and the trajectories of $W $ are in $\mathcal {C}([0,T];U_0)$, $\mathbb{P}$-a.s.


To the best   of our knowledge, unlike its deterministic counterpart (i.e., $f\equiv 0$ in \eqref{KS-SNS}), the literature on the stochastically perturbed KS-NS system is still relatively scarce. Indeed, the first result concerning the KS-SNS system \eqref{KS-SNS} was obtained recently by Zhai and Zhang in \cite{zhai20202d}, which deals with the global solvability of mild or weak solutions to \eqref{KS-SNS} in a bounded convex domain $\mathscr{O}\subset \mathbb{R}^2$. Later, the Zhang and Liu proved the existence of a global martingale weak solution in dimension three when the system \eqref{KS-SNS} is perturbed by a more general L\'{e}vy processes \cite{zhang2022global}. Very recently, Hausenblas et al. \cite{hausenblas2024existence} studied the initial-boundary value problem for the two-dimensional SCNS system with an additional random noise imposed on the $n$-equation.

The \textsf{main contribution} of  present work is further to investigate the global solvability of \eqref{KS-SNS} under different relaxed assumptions in two-dimensional plane. Compared with the existed results, the novelty in our work are three folds:

$\bullet$ In this paper, we initiates the study of global well-posedness of \eqref{KS-SNS} in the context of unbounded domain  $\mathbb{R}^2$, which of course causes new difficulties and hence novel ideas have to be introduced. We believe that the framework used in the proof is applicable for studying the other generalized systems in unbounded domains.

$\bullet$ Our first main result shows that the cell-density $n_0$ lives in a larger class of functions and so removes the restriction of the $L^2$-integrability of the quantity $n_0 \langle x\rangle  $ in \cite{zhai20202d}. The proof is essentially based on an entropy-energy inequality established in Lemma \ref{unf}, which even improves the one for deterministic counterpart.

$\bullet$ In the second main result, although the condition $n_0 \langle x\rangle \in L^2(\mathbb{R}^2)$ is retained,  we  prove that the large enough diffusion coefficients $d_i$, $i=1,2,3$, are sufficient to guarantee the existence of global solutions. The proof also strongly depends on a newly derived entropy-energy inequality.

\subsection{Main results} \label{sec1.2}

Let us first give some notations that will be frequently used in the following argument. The inequality $A\lesssim_{a,b,...}B$ means that there exists a positive constant $C$ depending only on $a,b,...$ such that $A\leq C B$, while $A\asymp_{a,b,...}B$ indicates that there two positive constants $c\leq C$ depending only on $a,b,...$ such that $cB\leq A\leq C B$. We use $\nabla \wedge u$ to denote the vorticity of a two-dimensional vector field $u=(u_1,u_2)$, i.e., $\nabla \wedge u= \partial_{x_1}u_2-\partial_{x_2}u_1$.

To give the statement of the first main result, let us make the following assumptions:
\begin{itemize}
\item [(\textbf{\textsf{A$_1$}})] The initial data $(n_0,c_0,u_0)$ satisfies

\begin{itemize}
\item [1)] $n_0 \in L^1(\mathbb{R}^2)\cap L^2(\mathbb{R}^2)$, $n_0>0$,

\item [2)] $c_0 \in L^1(\mathbb{R}^2)\cap L^\infty(\mathbb{R}^2)$, $|\nabla \sqrt{c}|\in L^2(\mathbb{R}^2)$, $c_0>0$,

\item [3)] $u_0 \in L^2(\mathbb{R}^2)$, $ \textrm{div}  u_0 =0$.
\end{itemize}

\item [{(\textbf{\textsf{A$_2$}})}] For the parameters $\phi,\kappa$ and $\chi$ in \eqref{KS-SNS}, we assume that

\begin{itemize}
\item [1)] $\phi\in W^{2,\infty}(\mathbb{R}^2;\mathbb{R})$,

\item [2)] $\chi :[0,\infty)\mapsto (0,\infty)$ is a $\mathcal {C}^2$ function, $\kappa\in \mathcal {C}^2([0,\infty))$, $\kappa(0)=0$,  and $\kappa>0$ on $(0,\infty)$,

\item [3)] $\left(\kappa/\chi\right)'>0$, $\left(\kappa/\chi\right)''\leq 0$,  and $ \left(\kappa \chi\right)'\geq0$ on $[0,\infty)$.
\end{itemize}

\item [{(\textbf{\textsf{A$_3$}})}] For the stochastic term, we suppose

\begin{itemize}
\item [1)] For all $s\in \mathbb{R}$, there is  a positive constant $\varrho>0$ such that
\begin{equation*}
\begin{split}
  \|f (t,u)\|_{L_2(U;H^s)}^2&\leq \varrho (1+\|u\|_{H^s}^2),\quad\forall u \in H^s(\mathbb{R}^2).
 \end{split}
\end{equation*}
\item [2)] For each $M>0$, there is a $C_M>0$ such that
\begin{equation*}
\begin{split}
\|f (t,u)-f (t,v)\|_{L_2(U;H^s)}  &\leq C_M  \|u-v\|_{H^s} ,\\
\|\nabla \wedge f (t,u)-\nabla \wedge f (t,v)\|_{L_2(U;H^{s-1})}  &\leq C_M \|u-v\|_{H^s},
 \end{split}
\end{equation*}
for any $u,v \in B_M:=\{u \in H^s(\mathbb{R}^2);~\|u\|_{H^s}\leq M\}$, where $L_2(U;H^s(\mathbb{R}^2))$ denotes the space of all Hilbert-Schmidt operators from $U$ to $H^s(\mathbb{R}^2)$ equipped with the norm
$
\|f(t,u)\|_{L_2(U;H^s)}^2=\sum_{j\geq 1}\|f_j(t,u)\|_{H^s}^2$, where $f_j(t,u)= f(t,u)e_j$.
 \end{itemize}
\end{itemize}

\begin{remark}
Concerning the chemotactic sensitivity and the consumption rate of substrate in (\textsf{A$_2$}), it is of interest to consider $\chi\equiv\textrm{const.}>0 $ (which is required for uniqueness result in our results), and $\kappa(s)=s$ for all $s\geq0$,
which corresponds to the prototypical chemotaxis model (cf. \cite{hillen2009user,arumugam2021keller}).
One of the candidates in (\textsf{A$_3$}) is given by the linear multiplicative noise $f(t,u)=cu$, for some $c\in \mathbb{R}\setminus\{0\}$, where $c$ describes the strength of the noise.
\end{remark}

Our first main result can now be stated by the following result. Without loss of generality, we shall assume that $d_1=d_2=d_3\equiv 1$ in Theorem \ref{th1}.

\begin{theorem}\label{th1}
Under the hypothesises (\textsf{A$_1$})-(\textsf{A$_3$}), the following statements hold:
\begin{itemize}
\item [(\textbf{\textsf{a$_1$}})] (\textsf{Global martingale weak solution}) For any $T>0$, there exists a stochastic basis $(\Omega,\mathcal {F},\{\mathcal {F}_t\}_{t>0},\mathbb{P})$ with complete right continuous filtration, on which defined a $\mathcal {F}_t$-cylindrical Wiener process $\{W(t)\}_{t\geq 0} $ and a progressively measurable process $(n,c,u)$, such that $\mathbb{P}$-a.s.
\begin{equation*}
\begin{split}
&n \in  L^\infty(0,T;L^1(\mathbb{R}^2)\cap L^2(\mathbb{R}^2))\bigcap L^2(0,T;H^1(\mathbb{R}^2)),\\
&c \in  L^\infty(0,T;L^1(\mathbb{R}^2)\cap L^\infty(\mathbb{R}^2))\bigcap L^\infty(0,T;H^1(\mathbb{R}^2))\bigcap L^2(0,T;H^2(\mathbb{R}^2)),\\
&u \in  L^\infty(0,T; L^2(\mathbb{R}^2))\bigcap L^2(0,T;H^1(\mathbb{R}^2)),
 \end{split}
\end{equation*}
and for any $t \in (0,T]$, the following relationships
\begin{equation}\label{1.2}
\begin{split}
 \int_{\mathbb{R}^2}n(t) \varphi\mathrm{d}x&= \int_{\mathbb{R}^2}n(0) \varphi\mathrm{d}x+\int_0^t \int_{\mathbb{R}^2} (u n-d_1\nabla n+n\chi(c)\nabla c)\cdot\nabla\varphi\mathrm{d}x\mathrm{d}r,\\
 \int_{\mathbb{R}^2}c(t) \varphi\mathrm{d}x&= \int_{\mathbb{R}^2}c(0) \varphi\mathrm{d}x +\int_0^t \int_{\mathbb{R}^2} (u c-d_2\nabla c-n \kappa(c))\cdot\nabla\varphi\mathrm{d}x\mathrm{d}r,\\
 \int_{\mathbb{R}^2}u(t)\cdot\psi\mathrm{d}x&=\int_{\mathbb{R}^2}u(0)\cdot\psi\mathrm{d}x  +\int_0^t \int_{\mathbb{R}^2} (u \otimes u- d_3\nabla u ) : \nabla\psi\mathrm{d}x\mathrm{d}r\\
 &+\int_0^t \int_{\mathbb{R}^2} n \nabla\phi\cdot \psi\mathrm{d}x\mathrm{d}r+ \sum_{j\geq 1}\int_0^t \int_{\mathbb{R}^2} f(r,u)\cdot \psi \mathrm{d}x \mathrm{d}  W^j(r)
 \end{split}
\end{equation}
hold $\mathbb{P}$-a.s., for any $\varphi\in \mathcal {C}^\infty_0(\mathbb{R}^2;\mathbb{R})$ and $\psi\in \mathcal {C}^\infty_{0 }(\mathbb{R}^2;\mathbb{R}^2)$ with $ \textrm{div} \ \psi =0$.

\item [(\textbf{\textsf{a$_2$}})] (\textsf{Pathwise uniqueness}) In addition, if the chemotactic sensitivity  $\chi(\cdot)\equiv const.> 0$ is a real number, then under the given stochastic basis $(\Omega,\mathcal {F},\{\mathcal {F}_t\}_{t>0},\mathbb{P})$, the weak solution $(n,c,u)$ to the KS-SNS system \eqref{KS-SNS} exists uniquely. As a result, system \eqref{KS-SNS} admits a global pathwise weak solution.
\end{itemize}
\end{theorem}

\begin{remark} A few remarks on Theorem \ref{th1} are in order.
\begin{itemize}
\item [$\bullet$] In the above theorem, it should be noted that martingale solution means that the solution is defined on a filtered completed probability space which is not assumed in advance,  while the weak solution means that the solution satisfies the system in  the framework of PDEs, where the associated spacial derivatives are described in the sense of distribution.

\item [$\bullet$] By introducing reasonable boundary conditions within a smooth bounded domain $D\subset \mathbb{R}^2$, such as $ \frac{\partial n}{\partial \textbf{n}} =\frac{\partial c}{\partial \textbf{n}}=0$ and $u =0$,  on $\mathbb{R}^+\times\partial D$,  where $\textbf{n}$ stands for the inward normal on the boundary, it is not hard to verify that the proof in present work still remains true for the KS-SNS in smooth bounded domains, which indicates that Theorem \ref{th1} covers the results in \cites{liu2011coupled,winkler2012global,jiang2015global}.

\item [$\bullet$]  Theorem \ref{th1} provides a new global well-posedness of system \eqref{KS-SNS} without assuming that $n_0|x| \in L^1(\mathbb{R}^2)$ for the density of bacteria. Roughly speaking, $n_0|x| \in L^1(\mathbb{R}^2)$ requires that
\begin{equation}
\begin{split}
n_0(x)\sim \frac{1}{|x|^{3+\delta}}~~ \textrm{as}~~ |x|\rightarrow\infty,~~ \textrm{for some} ~~\delta>0,
 \end{split}
\end{equation}
which excludes the class of initial data $n_0$ decaying slower than $1/|x|^{3+\delta}$ as $|x|\rightarrow\infty$, such as $n_0(x)= 1/(1+|x|)^\gamma$ ($2<\gamma \leq 3$), it is clear that $n_0\in L^1 (\mathbb{R}^2)$, but $n_0|x| \not\in L^1 (\mathbb{R}^2)$. Thereby, Theorem \ref{th1} improves the recent global well-posedness result for the KS-SNS system \eqref{KS-SNS} in \cite{zhai20202d}, and also the global result for deterministic counterparts in Theorem 1.4 (2D version) of \cite{chae2013existence}, Theorem 1.1 of \cite{zhang2014global}, Theorem 2.1 of \cite{liu2011coupled} and Theorem 3.1 of \cite{duan2010global}.

\item [$\bullet$] \textsf{Difficulties and strategies} in proving Theorem \ref{th1}:
\begin{itemize}
\item [{1)}] In order to overcome the difficulty caused by the low-regularity terms such as $ \textrm{div}  (n \nabla c  )$ in the $n$-equation, we first regularize the nonlinear terms by using standard mollifiers. This regularization is important in establishing the well-posedness of the approximate solutions, especially for the uniqueness of approximations and the entropy-energy inequality. Indeed, such a technique has been used in the existed works for deterministic CNS systems (cf. \cite{duan2010global,liu2011coupled,chae2014global}). However, different from the deterministic cases and the existed works for stochastic counterparts, it is not an obvious work to ensure the existence of global solutions to the approximate system \eqref{Mod-1}. To overcome this difficulty, we use the regularization method in \cite{majda2002vorticity} to regularize the system \eqref{Mod-1} with Friedrichs projectors, resulting in a nonlinear SDE in Hilbert space. Furthermore, we apply the efficient truncation technique (cf. \cite{rockner2014local,breit2018stochastically,du2020local}) to add a smooth cut-off function (with the parameter $R$) in front of the nonlinear terms, which help us to construct global approximate solutions on any interval $[0,T]$. For this three-layer approximate system, one can derive a priori uniform estimates, and the existence and uniqueness of approximations is ensured by the classical Leha-Ritter Theorem.

\item [{2)}] Due to the ``nonlocal peoperty'' of the cut-off operators and the loss of compactness from $H^s(\mathbb{R}^2)$ into $ H^t(\mathbb{R}^2)$ for $s>t$, one can not prove the property $\theta_R(\|\textbf{u}^{k,R,\epsilon}\|_{W^{1,\infty}})
\rightarrow\theta_R(\|\textbf{u}^{R,\epsilon}\|_{W^{1,\infty}})$ as $k\rightarrow\infty$ from the uniform bounds in Lemma \ref{lem2} directly.  In this work we shall overcome this difficulty by adopting ideas from \cite{li2021stochastic}, where the key step is to demonstrate the convergence of $\textbf{u}^{k,R,\epsilon}$ in $\mathcal {C}([0, T], \textbf{H}^{s-3}(\mathbb{R}^2))$ through a careful analysis of the approximation scheme, and then elevate the spatial regularity of the solution in $\textbf{H}^{s}(\mathbb{R}^2)$. Then, based on a uniqueness result for system and the stopping techniques, one can show that the system \eqref{3.1} admits a global approximate solution, see Lemmas \ref{lem7}-\ref{lem8} for details. We also remark that the aforementioned difficulty will not occur neither in the deterministic KS-NS system nor the KS-SNS system in bounded domains.

\item [{3)}] We prove the convergence of the approximate solution with parameter $\epsilon>0$ by using a stochastic compact method. As usual, we are inspired to establish some uniform a priori bounds for the approximation solutions. However, very different with the decoupled Navier-Stokes equations, the deterministic CNS system or the stochastic CNS system in bounded domain,  the existed energy estimates are not sufficient to achieve this goal. By taking advantage of the special structure of the system, we are able to derive a stochastic version of the entropy-energy functional inequality (cf. Lemma \ref{lem10}), which improves the the stochastic version in \cite{zhai20202d} and even the deterministic result in \cite{zhang2014global}. Based on this type of uniform estimate, we can prove the tightness of the approximate solutions, and then use the stochastic compactness method to take the limit as $\epsilon \rightarrow 0$ to obtain the global martingale solutions to \eqref{KS-SNS}. We remark that the entropy-type estimates is a powerful tool in dealing with global well-posedness problems, which has been widely applied in the study of global solvability of deterministic chemotaxis-fluid systems, such as  \cite{duan2010global,liu2011coupled,winkler2012global,chae2014global} and the references therein.
    %
    \end{itemize}
\end{itemize}
\end{remark}

 To state the second main result, we need the following hypothesises.

\begin{itemize}
\item [{(\textbf{\textsf{B$_1$}})}] The initial data $(n_0,c_0,u_0)$ satisfies
\begin{itemize}
\item [1)] $n_0 (| \ln n_0 |+\langle x\rangle) \in L^1(\mathbb{R}^2)$, $\langle x\rangle= (1+|x|^2)^{1/2}$; $n_0>0$,
 \item [2)] $c_0 \in L^\infty(\mathbb{R}^2)\cap L^1(\mathbb{R}^2)\cap H^1(\mathbb{R}^2)$, $c_0>0$,
 \item [3)] $u_0 \in L^2(\mathbb{R}^2)$, $ \textrm{div}  u_0 =0$.
\end{itemize}

\item [{(\textbf{\textsf{B$_2$}})}]  The diffusion coefficients $d_1$, $d_2$ and $d_3$ satisfy
\begin{equation}\label{1.4}
\begin{split}
\frac{2\Lambda^2\|n_0 \|_{L^1}}{d_1 d_2} \left(M_\kappa^2 +  \Lambda^4 \|  c_0\|_{L^\infty}^2\left(\frac{M_\chi^4 }{8 d_1^2  } +\frac{4}{ d_3}\right)\right)\leq 1,
 \end{split}
\end{equation}
where $$ M_\kappa:=   \max_{t\in[0,\|c_0\|_{L^\infty}]} |\kappa (t)|, ~~~M_\chi:=   \max_{t\in[0,\|c_0\|_{L^\infty}]} |\chi (t)|,$$ and $\Lambda$ is the general constant such that the Gagliardo-Nirenberg inequality $\|n^\epsilon \|_{L^2} \leq \Lambda \|\sqrt{n^\epsilon} \|_{L^2} \|\nabla\sqrt{n^\epsilon}\|_{L^2} $ holds.

\item [{(\textbf{\textsf{B$_3$}})}]  The potential $\phi\in W^{1,\infty}(\mathbb{R}^2 )$; For each $M>0$, there exists a $C_M>0$ such that
$$
|\chi (t)|\leq C_M,~~|\kappa (t)|\leq C_M,~~ \forall t\in [0,M].
$$
Moreover, $\kappa(\cdot)$ is a non-negative function which is continuous at $t=0$ with $\kappa (0)=0$.
\end{itemize}

Our second main result can be stated by the following theorem.

\begin{theorem}\label{th2}
Under the assumptions {(\textsf{B$_1$})}-{(\textsf{B$_3$})}, the following statements hold:
\begin{itemize}
\item [{(\textbf{\textsf{b$_1$}})}] The KS-SNS system \eqref{KS-SNS} has at least one   global martingale weak solution $(n,c,u,W)$. More precisely, there exists a stochastic basis $(\Omega,\mathcal {F},\{\mathcal {F}_t\}_{t>0},\mathbb{P})$, on which defined a $\mathcal {F}_t$-cylindrical Wiener process $\{W(t)\}_{t\geq 0} $ and a progressively measurable triple $(n,c,u)$, such that $\mathbb{P}$-a.s.
\begin{equation*}
\begin{split}
&n (|\ln n|+| x|)\in  L^\infty(0,T;L^1(\mathbb{R}^2) ),~ \sqrt{n} \in L^2(0,T;H^1(\mathbb{R}^2)),\\
&c \in  L^\infty(0,T;L^\infty(\mathbb{R}^2)\cap L^1(\mathbb{R}^2)\cap H^1(\mathbb{R}^2))\bigcap L^2(0,T;H^2(\mathbb{R}^2)),\\
&u \in  L^\infty(0,T; L^2(\mathbb{R}^2))\bigcap L^2(0,T;H^1(\mathbb{R}^2)),
 \end{split}
\end{equation*}
and the quadruple $(n,c,u,W)$ satisfies the relationships in \eqref{1.2} $\mathbb{P}$-a.s.

\item [{(\textbf{\textsf{b$_2$}})}] Moreover, if $\chi(\cdot)\equiv \textrm{const.}$ is a positive constant, then the KS-SNS system \eqref{KS-SNS} admits a unique global pathwise weak solution under the stochastic basis $(\Omega,\mathcal {F},\{\mathcal {F}_t\}_{t>0},\mathbb{P})$.
\end{itemize}
\end{theorem}

\begin{remark} For Theorem \ref{th2}, we make the following remarks:
\begin{itemize}
\item [$\bullet$] The strategy of the proof for Theorem \ref{th2} is similar to Theorem \ref{th1}, which depends on the same approximate scheme introduced in Subsection \ref{subsec2.1}, and the proof can also be extended to the KS-SNS system in bounded domain $D\subset \mathbb{R}^2$ with suitable boundary conditions. The key point is to derive a new stochastic analogue of entropy-energy functional inequality (cf. Lemma \ref{lem4.1}) by making use of the fine structure of the equations.

\item [$\bullet$] The main difference between Theorem \ref{th2} and the existing results (cf. \cites{zhai20202d,zhang2022global}) lies in the fact that only locally continuity conditions (cf. (\textsf{B$_3$})) for the functional $\chi$ and $\kappa$ are needed without any structural and monotonicity conditions (cf. {(\textsf{A$_2$})} in Theorem \ref{th1} or the assumptions (A)(b)  in \cite{zhai20202d}), while the cost is to keep the valid of the condition $n_0|x| \in L^1(\mathbb{R}^2)$  and add some boundedness condition \eqref{1.4} for diffucsion coefficients. Moreover, the assumption {(\textsf{B$_2$})} infers that, for any fixed viscosity $d_3>0$, the properly chosen coefficients $d_1,d_2>0$ are sufficient to rule out the singularity of solutions.
\end{itemize}
\end{remark}

\subsection{Structure of the paper}
In Section \ref{subsec2.1}, we introduce the approximation procedure for \eqref{KS-SNS} by felicitously regularizing the system and defining cut-off operators (cf. the systems \eqref{Mod-1} and \eqref{Mod-2}), from which we construct a family of approximations $\textbf{y}^{k,R,\epsilon} \in \mathcal {C}([0,T];\textbf{H}^s(\mathbb{R}^2))$. Subsection \ref{subsec2.2} is devoted to an uniform estimate with respect to $k\geq1$, and we prove in Subsection 2.3 that the approximations converges (up to a subsequence) to a unique pathwise solution $\textbf{y}^{\epsilon} $ of \eqref{Mod-1}. In Section \ref{sec3}, we first establish an entropy-nergy functional inequality and several energy estimates on the events $\Omega_R^\epsilon$ and $\widetilde{\Omega}_N^\epsilon$ large enough. Then by using the stochastic compactness method and Yamada-Watanabe Theorem, we take the limit as $\epsilon\rightarrow0$ to construct a unique global pathwise solution $(n,c,u)$ to system \eqref{KS-SNS}. In Section \ref{sec4}, a new entropy-energy inequality under the key conditions {(\textsf{B$_2$})} on diffusion coefficients is derived. The method of the proof for Theorem \ref{th2} is similar to Theorem \ref{th1}, we just formulate the sketch of proof.

\section{Solvability of the modified systems}

\subsection{Approximation scheme}\label{subsec2.1}
We will construct the approximation solutions as follows. Let $\rho^\epsilon $ be a standard Friedrichs mollifier (cf. \cite{majda2002vorticity}), then we consider the following first modified system:
\begin{equation}\label{Mod-1}
\left\{
\begin{aligned}
&\mathrm{d}  n^\epsilon +u^\epsilon\cdot \nabla n^\epsilon  \mathrm{d} t = \Delta n ^\epsilon \mathrm{d} t-  \textrm{div} \left(n^\epsilon[(\chi(c^\epsilon)\nabla c^\epsilon)*\rho^\epsilon]\right) \mathrm{d} t , \\
&\mathrm{d}  c^\epsilon+ u^\epsilon\cdot \nabla c^\epsilon   \mathrm{d} t = \Delta c^\epsilon \mathrm{d} t-\kappa(c^\epsilon)(n^\epsilon*\rho^\epsilon)  \mathrm{d} t ,\\
&\mathrm{d} u^\epsilon+   \textbf{P} (u^\epsilon\cdot \nabla) u^\epsilon  \mathrm{d} t =   A u^\epsilon \mathrm{d} t +  \textbf{P}[ (n^\epsilon\nabla \phi)*\rho^\epsilon] \mathrm{d} t+  \textbf{P} f(t,u^\epsilon) \mathrm{d} W  ,\\
&  \textrm{div}   u ^\epsilon =0 , \\
&n^\epsilon|_{t=0}=n_0^\epsilon,~c|_{t=0}=c_0^\epsilon,~ u|_{t=0}=u_0^\epsilon,
\end{aligned}
\right.
\end{equation}
 with smooth initial conditions
$$
n_0^\epsilon= n_0*\rho^\epsilon, ~~c_0^\epsilon= c_0*\rho^\epsilon,~~u_0^\epsilon= u_0*\rho^\epsilon.
$$
Here $A= - \textbf{P}  \Delta$ is the Stokes operator, where $ \textbf{P} $ denotes the
Helmholtz projection from $L^2(\mathbb{R}^2)$ into $L^2_ \textrm{div} (\mathbb{R}^2):=\{u \in L^2(\mathbb{R}^2);~  \textrm{div} ~ u=0\}$.

The construction of solutions $(n^\epsilon,c^\epsilon,u^\epsilon)$ to \eqref{Mod-1} seems to be impossible by the existed methods for deterministic KS-NS system in \cites{duan2010global,liu2011coupled,chae2014global,zhang2014global}) and the stochastic counterparts in \cites{zhai20202d,zhang2022global}, which forces us to introduce a further regularized system in the following manner.

For any $k\in\mathbb{N}$, let $J_k$ be the frequency truncation operators defined by
$$
\widehat{J_k f} (\xi)= \textbf{1}_{\mathcal {C}_k}(\xi) \widehat{f}(\xi),~~~\mathcal {C}_k=\{\xi\in \mathbb{R}^2;~ k^{-1}\leq |\xi|\leq k\},
$$
where $\widehat{f}$ denotes the Fourier transformation of $f$. For any $R>0$, choose a smooth cut-off function $\theta_R : [0,\infty)\rightarrow [0,1]$ such that
$\theta_R (x)\equiv 1$ if $0\leq x \leq R$ and $\theta_R (x)\equiv 0$ if $x > 2R$. Then the second modified KS-SNS system is given by
\begin{equation}\label{Mod-2}
\left\{
\begin{aligned}
& \mathrm{d}  \textbf{y}^{k,R,\epsilon}= F ^{k,R } (\textbf{y}^{k,R,\epsilon}  )  \mathrm{d} t
+G^{ R} (\textbf{y}^{k,R,\epsilon}  )\mathrm{d}  \mathcal {W}(t),\\
& \textbf{y}^{k,R,\epsilon}(0) =\textbf{y}^{k,R,\epsilon}_0= (J_kn_0^\epsilon,J_kc_0^\epsilon,J_ku_0^\epsilon),
\end{aligned}
\right.
\end{equation}
where $\textbf{y}^{k,R,\epsilon} =(n^{k,R,\epsilon} , c^{k,R,\epsilon},
u^{k,R,\epsilon}): \Omega\times \mathbb{R}^+ \times \mathbb{R}^2\mapsto \mathbb{R}^4$ is a four-dimensional random field. The drift term and diffusion term in \eqref{Mod-2} are defined  by
$$
 F^{k,R}(\cdot)= \begin{pmatrix}
F_1^{k,R} (\cdot)\\
  F_2^{k,R} (\cdot)\\
F_3^{k,R}(\cdot)
\end{pmatrix}, \quad G^{ R}(\cdot) \mathrm{d}  \mathcal {W}=\begin{pmatrix}
0 & 0 & 0\\
  0& 0 & 0\\
0 & 0&  \theta_R\left(\|\cdot\|_{W^{1,\infty}}\right)    \textbf{P}  f(t,\cdot)
\end{pmatrix}\begin{pmatrix}
0 \\
  0 \\
\mathrm{d}  W
\end{pmatrix}.
$$
Given any $\textbf{y}=(n,c,u)$, the functionals $F_i^{k,R}(\cdot)$, $i=1,2,3$ are given by
\begin{equation*}
\begin{split}
 F_1^{k,R}(\textbf{y} ) =& \Delta J_k^2 n -\theta_R(\|(u,n)\|_{W^{1,\infty}})J_k(J_k u\cdot \nabla J_k n)\\
 &- \theta_R(\|(n,c)\|_{W^{1,\infty}})    \textrm{div}  J_k(J_kn[\chi(J_kc)\nabla J_kc]*\rho^{\epsilon}),\\
F_2^{k,R}(\textbf{y} )= &\Delta J_k^2 c-\theta_R(\|(u,c)\|_{W^{1,\infty}})J_k(J_k u\cdot \nabla J_k c) -\theta_R(\|(n,c)\|_{W^{1,\infty}})J_k([J_kn \kappa(J_kc)]*\rho^{\epsilon}),\\
F_3^{k,R}(\textbf{y} )= &\Delta  J_k^2u-\theta_R(\|u\|_{W^{1,\infty}}) \textbf{P} J_k(J_k u\cdot \nabla J_k u) + \textbf{P} J_k((J_kn\nabla \phi)*\rho^{\epsilon}).
\end{split}
\end{equation*}

In order to solve the approximate system \eqref{Mod-2}, for any $s\in\mathbb{R} $, we define the  space
$$
 \textbf{H}^s(\mathbb{R}^2):=\{(n,c,u)\in H^s(\mathbb{R}^2)\times H^s(\mathbb{R}^2) \times(H^s(\mathbb{R}^2))^2;~  \textrm{div}  u=0\},
$$
which is equipped with the norm $\|(n,c,u)\|_{\textbf{H}^s}= \|n\|_{H^s}+\|c\|_{H^s}+\|u\|_{H^s}$. Especially if $s=0$, we set $\textbf{H}^0 (\mathbb{R}^2)=\textbf{L}^2 (\mathbb{R}^2)$. For convenience, we also denote $\| (f,g)\|_{W^{s,p}}=\|f\|_{W^{s,p}}+\|g\|_{W^{s,p}}$.

\begin{lemma}\label{lem1} Given a stochastic basis $(\Omega,\mathcal {F},\{\mathcal {F}_t\}_{t>0},\mathbb{P})$. Let $s> 5$, $\epsilon>0$, $R> 1$ and $k\in \mathbb{N}$, then
under the assumptions {(\textsf{A$_1$})}-{(\textsf{A$_3$})}, the system \eqref{Mod-2} has a unique solution
$$
(n^{k,R,\epsilon} , c^{k,R,\epsilon},
u^{k,R,\epsilon})\in \mathcal {C}([0,T];\textbf{H}^s(\mathbb{R}^2)),~~ \mathbb{P}\mathrm-{a.s.}
$$
\end{lemma}

\begin{proof}[\emph{\textbf{Proof}}]
Since supp $\widehat{J_k f} \subset\{\xi\in \mathbb{R}^2;~1/k\leq |\xi |\leq k\}$ for all $k\geq1$, it follows from the Bernstein inequality (cf. \cite[Lemma 2.1 and Lemma 2.2]{bahouri2011fourier}, which will be applied frequently in the following estimations without further declaration) that
\begin{equation}\label{2.1}
\begin{split}
\|\nabla^l J_k f\|_{H^s} \asymp _{k,l} \|J_k f\|_{H^s},\quad \textrm{for any}~ k\geq 1,~l\geq0.
\end{split}
\end{equation}
For any fixed $k,R$ and $\epsilon$, we deduce from \eqref{2.1} that
\begin{equation}\label{2.2}
\begin{split}
 \|F^{k,R}(\textbf{y}^{k,R,\epsilon}) \|_{\textbf{H}^s} \lesssim_k  \|\textbf{y}^{k,R, \epsilon}\|_{\textbf{H}^s}+1 ,\quad\|G^{R}(t,\textbf{y}^{k,R,\epsilon}) \|_{H^s} \lesssim \|\textbf{y}^{k,R,\epsilon}\|_{H^s}+1.
\end{split}
\end{equation}
Moreover, by \eqref{2.1} and the Moser-type estimate (cf. \cite{miao2012littlewood}), one can verify that both $F^{k,R}(\cdot):\textbf{H}^s(\mathbb{R}^2)\mapsto \textbf{H}^s(\mathbb{R}^2) $  and $G^{k,R}(\cdot): \textbf{H}^s(\mathbb{R}^2)\mapsto L_2(U;\textbf{H}^s(\mathbb{R}^2))$ are locally Lipschitz continuous functionals, that is, we have for a given $M>0$
\begin{equation}\label{2.3}
\begin{split}
\|F^{k,R}(\textbf{y}_1^{k,R,\epsilon}) -F^{k,R}(\textbf{y}_2^{k,R,\epsilon}) \|_{\textbf{H}^s} &\lesssim_{M,R,k,\phi,\epsilon}  \|\textbf{y}_1^{k,R,\epsilon}-\textbf{y}_2^{k,R,\epsilon}\|_{\textbf{H}^s} ,\\
 \|G^{R}(\textbf{y}_1^{k,R,\epsilon})-G^{R}(\textbf{y}_2^{k,R,\epsilon}) \|_{\textbf{H}^s} &\lesssim_{R,\epsilon}  \|\textbf{y}_1^{k,R,\epsilon}-\textbf{y}_2^{k,R,\epsilon}\|_{\textbf{H}^s},
\end{split}
\end{equation}
where $\|\textbf{y}_1^{k,R,\epsilon}\|_{\textbf{H}^s}\leq M$ and $\|\textbf{y}_2^{k,R,\epsilon}\|_{\textbf{H}^s}\leq M$. The verification of \eqref{2.3} is straightforward and based on the Bernstein inequality and the Moser inequality, we omit the details here.

According to conditions \eqref{2.2} and \eqref{2.3}, one can conclude from the well-known theory for SDEs in Hilbert spaces (cf. Theorms 5.1.1 and 5.1.2 in \cite{kallianpur1995stochastic}; Theorem 4.2.4 in \cite{prevot2007concise}) and the cancelation property (cf. \cites{glatt2014local,majda2002vorticity}) that the system \eqref{Mod-2} admits a unique global smooth solution $\textbf{y} ^{k,R,\epsilon}$ in $\mathcal {C}([0,T];\textbf{H}^s(\mathbb{R}^2))$, $\mathbb{P}$-a.s. The proof of Lemma \ref{lem1} is completed.
\end{proof}

In the following sections, we will take the limits as $k\rightarrow\infty$, $R\rightarrow\infty$  and $\epsilon\rightarrow 0$ in suitable sense successively to prove that the limit process $(n,c,u)$ is actually a unique global pathwise solutions to the KS-SNS system \eqref{KS-SNS}.

\subsection{Uniform bounds in Sobolev spaces}\label{subsec2.2}
As a first step, we will take the limit as $k\rightarrow\infty$  in \eqref{Mod-2} to construct a global solution to system \eqref{Mod-1} with cutoffs. Let us start with the following a priori estimates uniformly in $k\in \mathbb{N}$.

\begin{lemma} \label{lem2}
Assume that $s>5$, $R>1$ and $\epsilon>0$. Let $T>0$ and $(n^{k,R,\epsilon},c^{k,R,\epsilon},u^{k,R,\epsilon}) \in \mathcal {C}([0,T]; \textbf{H}^s(\mathbb{R}^2))$ be the unique solutions to \eqref{Mod-2} constructed in Lemma \ref{lem1}. Then we have
$$
\sup_{k \in \mathbb{N}}\mathbb{E}\left\|(n^{k,R,\epsilon},c^{k,R,\epsilon},u^{k,R,\epsilon})\right\| _{\mathcal {C}([0,T]; \textbf{H}^s )}^p \lesssim_{R,s,p,\phi,\kappa,\chi,\epsilon,n_0,c_0,u_0,T} 1,\quad \textrm{for all}~p\geq2.
$$
Moreover, for any $p>2$ and  $\alpha \in (0,\frac{p-2}{2p} )$, there holds
\begin{equation*}
\begin{split}
\sup_{k \in \mathbb{N}}\left(\mathbb{E}\left\|(n^{k,R,\epsilon}, c^{k,R,\epsilon})\right\|^p_{   \emph{\textrm{Lip}}([0,T]; H^{s-2} ) }+ \mathbb{E}\left\|u^{k,R,\epsilon}\right\|^p_{ \mathcal {C}^{ \alpha}([0,T];  H^{s-2} )  }  \right) \lesssim _{R,p,\epsilon,\phi,\chi,\kappa,n_0,c_0,u_0,T}1.
\end{split}
\end{equation*}
\end{lemma}

\begin{proof}[\emph{\textbf{Proof}}] To simplify the notations, we   use the notation $(n,c,u)$ instead of  $(n^{k,R,\epsilon},c^{k,R,\epsilon},u^{k,R,\epsilon})$ in the proof.

\textsf{Step 1 ($L^2$-estimate).}  Applying the Bessel operator $\Lambda^s= (1-\Delta)^{\frac{s}{2}}$ with $ s> 5$ (which is an isometry isomorphism between $H^{r+2}(\mathbb{R}^2)$ and $H^{r}(\mathbb{R}^2)$, cf. \cite{bahouri2011fourier}) to both sides of $n $-equation in \eqref{Mod-2}, and then  taking the scalar product with  $\Lambda^s n $ over $\mathbb{R}^2$, we get
\begin{equation}\label{2.4}
\begin{split}
&\frac{1}{2}  \mathrm{d}  \| n\|^2_{H^s}+\|\nabla J_kn\|^2_{H^s}  \mathrm{d} t
 =- \theta_R\left(\|(u,n)\|_{W^{1,\infty}}\right)( \Lambda^s J_kn,\Lambda^s \left(J_k u\cdot \nabla J_k n \right))_{L^2} \mathrm{d} t \\
&\quad \quad-  \theta_R\left(\|(n,c)\|_{W^{1,\infty}}\right)( \Lambda^s \nabla J_k n,\Lambda^s  \left(J_kn[\chi(c)\nabla J_kc]*\rho\right))_{L^2} \mathrm{d} t := (A_1+A_2)  \mathrm{d} t.
\end{split}
\end{equation}
For $A_1$, first recall the following commutator estimate (cf. \cite{miao2012littlewood}):
\begin{equation}\label{2.5}
\begin{split}
\|[\Lambda^s,f]g\|_{L^p} \lesssim_{p,s}\|\nabla f\|_{L^\infty}\|\Lambda^{s-1} g\|_{L^p}+\|\Lambda^s f\|_{L^p}\|g\|_{L^\infty}, \quad \textrm{for all} ~
1<p<\infty~ \textrm{and} ~ s \geq 0.
\end{split}
\end{equation}
By taking $p=2$ in \eqref{2.5} and using the divergence-free condition $ \textrm{div}  u =0$, we get
\begin{equation}\label{2.6}
\begin{split}
 |A_1| &= \theta_R\left(\|(u,n)\|_{W^{1,\infty}}\right)|([\Lambda^s,J_k u]\cdot \nabla J_k n, \Lambda^s J_kn) _{L^2} + (J_k u\cdot \nabla \Lambda^s J_k n , \Lambda^s J_kn)_{L^2} |\\
 &\leq C\theta_R\left(\|(u,n)\|_{W^{1,\infty}}\right)\| J_kn\|_{H^s}(\|J_k u\|_{H^s}\|\nabla J_k n\|_{L^\infty}+\|\nabla J_k u\|_{L^\infty}\|J_k n\|_{H^s})\\
 &\lesssim _R  \|J_k u\|_{H^s}^2+\|J_k n\|_{H^s}^2\lesssim _R  \| u\|_{H^s}^2+\|  n\|_{H^s}^2,
\end{split}
\end{equation}
where the last inequality used the fact of $\| J_kn\|_{H^s}\lesssim \|n\|_{H^s}$, for any $k\geq 1$. Note that
$
\chi(J_kc)\nabla J_kc = \nabla g(J_kc):=\nabla \int_0^ {J_kc} \chi(s)\mathrm{d}s.
$
it follows from  Theorem 2.87 in \cite{bahouri2011fourier} that
$$
\|\chi(J_kc)\nabla J_kc\|_{H^{s-1}}= \|\nabla g(J_kc)\|_{H^{s-1}} \leq \| g(J_kc)\|_{H^{s }} \leq _{s,\chi }\|J_kc\|_{H^{s}}.
$$
By \eqref{2.3}, the term $A_2$ can be estimated as
\begin{equation}\label{2.7}
\begin{split}
 |A_2| 
 & \leq \frac{1}{2}   \|\nabla J_k n\| _{H^s}^2+ C  \theta_R\left(\|(n,c)\|_{W^{1,\infty}}\right) ^2\big(\|J_kn\|_{L^\infty}\| J_kc\|_{H^{s }} +\|J_kn\|_{H^{s}}\|\chi(J_kc)\|_{L^\infty}\|\nabla J_kc\|_{L^\infty} \big)^2\\
 & \leq  \frac{1}{2} \|\nabla  J_kn\| _{H^s}^2+ C_{R,s,\chi }  \big( \| c\|_{H^s}^2 +\| n\|_{H^s} ^2 \big),
\end{split}
\end{equation}
where the last inequality used the fact of $\|f* \rho^\epsilon\|_{H^{s}}\leq \frac{C}{\epsilon} \|f \|_{H^{s-1}}$, and the continuity of $\chi(\cdot)$ such that
$$
\theta_R\left(\|(n,c)\|_{W^{1,\infty}}\right)\|\chi(J_kc)\|_{L^\infty}\leq \theta_R\left(\|(n,c)\|_{W^{1,\infty}}\right) \max_{s\in[0,\|J_kc\|_{L^\infty}]}\chi (s)
\lesssim _R 1.
$$
By the estimates \eqref{2.4}-\eqref{2.7}, we get
\begin{equation} \label{2.8}
\begin{split}
\mathbb{E}\sup_{s\in [0,t]} \| n(s)\|^2_{H^s}+\mathbb{E}\int_0^t\|\nabla J_kn(s)\|^2_{H^s} \mathrm{d}s
 \lesssim_{R,s,\chi }  \| n_0 \|^2_{H^s} +  \mathbb{E}\int_0^t(\|  u\|_{H^s}^2+\|  c\|_{H^s}^2+\|  n\|_{H^s}^2) \mathrm{d}s.
\end{split}
\end{equation}
Thanks to the assumption $\kappa(0)=0$ and the Paralinearization Theorem (cf. \cite[Theorem 2.89, p.105]{bahouri2011fourier}), we gain
$$
\theta_R(\|(n,c)\|_{W^{1,\infty}})\|\kappa(J_k c)\|_{L^\infty}\leq \theta_R(\|(n,c)\|_{W^{1,\infty}})\max_{s\in [0,\| J_k c \|_{L^\infty}]}\kappa(s)\leq \max_{s\in [0,2R]}\kappa(s)\lesssim _{R,\kappa} 1.
$$
In a similar manner to \eqref{2.8}, one can derive that
\begin{equation}\label{2.9}
\begin{split}
\mathbb{E}\sup_{s\in [0,s]}\|c(t)\|^2_{H^s}+ \mathbb{E}\int_0^t\|\nabla J_kc(s)\|^2_{H^s}\mathrm{d}s \lesssim_{R, \kappa } \| c_0\|^2_{H^s}+ \mathbb{E}\int_0^t\big(\|  u\|_{H^s}^2+ \|n\|_{H^s} ^2 +\|c\|_{H^s}^2 \big)\mathrm{d}s.
\end{split}
\end{equation}
To deal with the estimation for the $u$-equation perturbed by random noises, we first take the $\Lambda^s $ operator to the $u$-equation ($\Lambda^s$ is interchangeable with the operators $ \textbf{P} $ and $J_k$, cf. \cite{majda2002vorticity}),
and then apply It\^{o}'s formula (cf. \cite{da2014stochastic}) to $ \|\Lambda^s u\|_{L^2}^2$ to find
\begin{equation}\label{2.10}
\begin{split}
& \| \Lambda^s u(t)\|_{L^2}^2 +2 \int_0^t\|\nabla\Lambda^s J_ku\|_{L^2 }^2 \mathrm{d} t =\| \Lambda^s u_0\|_{L^2}^2+ \int_0^t(B_1 +B_2 +B_3 )  \mathrm{d} t+\sum_{j\geq 1}\int_0^tB_4^j  \mathrm{d}  W_j,
\end{split}
\end{equation}
where
\begin{equation*}
\begin{split}
&B_1 =  \theta_R(\|u\|_{W^{1,\infty}}) ^2 \| \Lambda^s \textbf{P}  f (t,u)\|_{L_2(U;L^2)}^2, \\
&
B_2 =2 (\Lambda^s J_k u, \Lambda^s \textbf{P} J_k((J_kn\nabla \phi)*\rho))_{L^2},\\
&B_3 =- 2  \theta_R  (\|u\|_{W^{1,\infty}}) (\Lambda^s J_k u, \Lambda^s \textbf{P} (J_k u\cdot \nabla J_k u ))_{L^2},\\
&
 B_4 ^j=2   \theta_R(\|u\|_{W^{1,\infty}})(\Lambda^s u,   \Lambda^s \textbf{P}  f _j(t,u))_{L^2},~~j\geq 1.
\end{split}
\end{equation*}
By virtue of the assumption {(\textsf{A$_3$})}, we gain
\begin{equation}\label{2.11}
\begin{split}
& |B_1 | \lesssim \theta_R\left(\|u\|_{W^{1,\infty}}\right)^2 (1+ \|u\|_{H^s}^2)\lesssim_R  1+ \|u\|_{H^s}^2 .
\end{split}
\end{equation}
For $B_2 $,
\begin{equation}\label{2.12}
\begin{split}
  |B_2  | \lesssim  \|\Lambda^s J_k u\|_{L^2}^2+\| \Lambda^s \textbf{P} J_k\left((J_kn\nabla \phi)*\rho\right)\|_{L^2}^2 \lesssim _\phi  \|u\|_{H^s}^2+ \|n\|_{H^s}^2 .
\end{split}
\end{equation}
For $B_3 $, first note that
$$
(\Lambda^s J_k u, \textbf{P} \Lambda^s\left(J_k u\cdot \nabla J_k u \right))_{L^2}= (\Lambda^s J_k u,  \textbf{P} [\Lambda^s,J_k u]\cdot \nabla J_k u)_{L^2},
$$
where $[A,B]=AB-BA$. Then we get by applying the commutator estimate \eqref{2.5} that
\begin{equation}\label{2.13}
\begin{split}
  |B_3  |  \lesssim  \theta_R\left(\|u\|_{W^{1,\infty}}\right)\|\Lambda^s J_k u\| _{L^2}\|J_ku\| _{H^{s}} \|\nabla J_ku\|_{L^\infty}  \lesssim _R \|u\| _{H^{s}}^2.
\end{split}
\end{equation}
For the stochastic term involving $B_4 ^j$, we get by using the Burkholder-Davis-Gundy (BDG) inequality (cf. \cite{da2014stochastic}) and Young inequality that
\begin{equation}\label{2.14}
\begin{split}
\mathbb{E}\sup_{r\in [0,t]}\left|\sum_{j\geq 1}\int_0^rB_4^j  \mathrm{d}  W_j\right|&\leq C\mathbb{E}\left(\int_0^ t\theta_R\left(\|u\|_{W^{1,\infty}}\right)^2\sum_{j\geq 1}(\Lambda^s u,   \Lambda^s \textbf{P}  f_j(s,u))_{L^2}^2 \mathrm{d}s\right)^{1/2} \\
&\leq C\mathbb{E}\left[\sup_{r\in [0,t]}\|\Lambda^s u\|_{L^2}\left( \int_0^t \theta_R \left(\|u\|_{W^{1,\infty}}\right)^2 \sum_{j\geq 1}\|\Lambda^s \textbf{P}  f_j(s,u)\|_{L^2}^2 \mathrm{d}s\right)^{1/2}\right] \\
&\leq \frac{1}{2} \mathbb{E}\sup_{s\in [0,t]}\|\Lambda^s u\|_{L^2}^2+ C  \mathbb{E} \int_0^ t \left(1+\|u(s)\|_{H^{s}}^2\right) \mathrm{d}s.
\end{split}
\end{equation}
Thereby, by taking the supremum over  $[0,t]$ in \eqref{2.10}, it follows from \eqref{2.11}-\eqref{2.14} that
\begin{equation*}
\begin{split}
 \mathbb{E}\sup_{r\in [0,t]}\| u(r)\|_{H^s}^2+ \mathbb{E}\int_0^t\|\nabla  J_ku\|_{H^s}^2 \mathrm{d} t \lesssim_{R,\phi} \| u_0\|_{H^s}^2+  \mathbb{E} \int _0^t(1+ \|u\|_{H^s}^2 + \|n\|_{H^s}^2) \mathrm{d}  s,
\end{split}
\end{equation*}
which together with \eqref{2.8}-\eqref{2.9} imply that
\begin{equation*}
\begin{split}
\mathbb{E}\sup_{s\in [0,t]}\|(n,c,u)(s)\|^2_{\textbf{H}^s} \lesssim_{R,s,\phi,\chi,\kappa,\epsilon} \|(n_0,c_0,u_0)\|^2_{\textbf{H}^s}+  \mathbb{E}\int_0^t\|(n,c,u)(s)\|^2_{\textbf{H}^s}\mathrm{d}  s.
\end{split}
\end{equation*}
Applying  Gronwall Lemma leads to
$$
\mathbb{E}\sup_{s\in [0,T]}\|(n,c,u)\|^2_{\textbf{H}^s} \lesssim_{R,s,\phi,\chi,\kappa,\epsilon,T} \|(n_0,c_0,u_0)\|^2_{\textbf{H}^s},~~ \forall T > 0.
$$

\textsf{Step 2: $L^p$-estimate ($p>2$).}  Let us first treat the estimate for $n$-equation. Indeed, by applying the chain rule to $ \| n(t)\|^p_{H^s}= (\| n(t)\|^2_{H^s})^{p/2}$, we get from \eqref{2.4} that
\begin{equation}\label{2.15}
\begin{split}
 \| n(t)\|^p_{H^s}&= \| n_0\|^p_{H^s} -p\int_0^t\| n(\tau)\|_{H^s}^{p-2}\|\nabla J_kn\|^2_{H^s} \mathrm{d}  \tau +p\int_0^t\| n(\tau)\|_{H^s}^{p-2} (A_1+A_2) \mathrm{d}  \tau,
\end{split}
\end{equation}
where $A_1$ and $A_2$ are defined in \eqref{2.4}. In view of the estimates in \textsf{Step 1}, we have
\begin{equation}\label{2.16}
\begin{split}
 & |A_1|+|A_2| \leq \frac{1}{2} \|\nabla  J_kn\| _{H^s}^2 + C_{R,s,\chi}  \big(\|u\|_{H^s}^2+ \|c\|_{H^s}^2 +\|n\|_{H^s} ^2 \big).
\end{split}
\end{equation}
It then follows from \eqref{2.15} and \eqref{2.16} that
\begin{equation}\label{2.17}
\begin{split}
& \| n(t)\|^p_{H^s}+\frac{p}{2}\int_0^t\| n(\tau)\|_{H^s}^{p-2}\|\nabla J_kn\|^2_{H^s} \mathrm{d}  \tau \\
&\quad  \lesssim_{p,R,s,\chi}  \| n_0\|^p_{H^s} +  \int_0^t(\| n(\tau)\|_{H^s}^{p} +\| c(\tau)\|_{H^s}^{p}+\| u(\tau)\|_{H^s}^{p})  \mathrm{d}  \tau.
\end{split}
\end{equation}
By applying the similar argument for the $c$-equation, one can obtain
\begin{equation}\label{2.18}
\begin{split}
& \| c(t)\|^p_{H^s}+\frac{p}{2}\int_0^t\| c(\tau)\|_{H^s}^{p-2}\|\nabla J_k c\|^2_{H^s} \mathrm{d}  \tau\\
 &\quad\lesssim_{p,R,s,\kappa}   \| c_0\|^p_{H^s} +  \int_0^t(\| n(\tau)\|_{H^s}^{p} +\| c(\tau)\|_{H^s}^{p}+\| u(\tau)\|_{H^s}^{p})  \mathrm{d}  \tau.
\end{split}
\end{equation}
Now we apply It\^{o}'s formula to $ \|\Lambda^s u(t)\|^p_{L^2}= (\|\Lambda^s u(t)\|^2_{L^2})^{p/2}$, we find
\begin{align}\label{2.19}
&\|\Lambda^s u(t)\|^p_{L^2}+ p \int_0^t\|\Lambda^s u(t)\| _{L^2} ^{p-2} \|\nabla \Lambda^s J_ku\|_{L^2 }^2\mathrm{d}  \tau\\
&\quad= \|\Lambda^s u_0\|^p_{L^2} + \frac{p}{2} \int_0^t\|\Lambda^s u(t)\| _{L^2} ^{p-2} (K_1 +K_2 +K_3 ) \mathrm{d}  \tau\nonumber\\
 &\quad\quad+\frac{p(p-2)}{4}\sum_{j \geq 1} \int_0^t\theta_R\left(\|u\|_{W^{1,\infty}}\right)\|\Lambda^s u(t)\|_{L^2} ^{p-4} \left(\Lambda^s u,   \Lambda^s \textbf{P}  f(t,u)e_j\right)_{L^2}^2 \mathrm{d}  \tau\nonumber\\
 &\quad\quad+\frac{p}{2}\sum_{j \geq 1}\int_0^t \theta_R\left(\|u\|_{W^{1,\infty}}\right)\|\Lambda^s u(\tau)\| _{L^2} ^{p-2}  \left(\Lambda^s u,   \Lambda^s \textbf{P}  f(t,u)e_j\right)_{L^2} \mathrm{d}  W_j \nonumber\\
 &\quad:= \|\Lambda^s u_0\|^p_{L^2}+ L_1+L_2+L_3,
\end{align}
where $B_i $, $i=1,2,3$ are defined in \eqref{2.10}. By \eqref{2.11}-\eqref{2.13}, one can estimate as
\begin{equation*}
\begin{split}
 \int_0^t\|\Lambda^s u\| _{L^2} ^{p-2} K_1 \mathrm{d}  \tau &\lesssim_R  \int_0^t\|\Lambda^s u\| _{L^2} ^{p-2}(1+ \|u\|_{H^s}^2) \mathrm{d}  \tau \lesssim_{p,R} \int_0^t(1+ \|u\|_{H^s}^p)\mathrm{d}  \tau,\\
 \int_0^t\|\Lambda^s u\| _{L^2} ^{p-2} K_2 \mathrm{d}  \tau &\lesssim_{\phi}  \int_0^t \| u\| _{H^S} ^{p-2}(\|u\|_{H^s}^2+ \|n\|_{H^s}^2)\mathrm{d}  \tau\lesssim_{p,\phi} \int_0^t(\|n\|_{H^s}^p+ \|u\|_{H^s}^p)\mathrm{d}  \tau,\\
 \int_0^t\|\Lambda^s u\| _{L^2} ^{p-2} K_3 \mathrm{d}  \tau &\lesssim_{ R} \int_0^t\|\Lambda^s u\| _{L^2} ^{p-2}\|u\| _{H^{s}}^2 \mathrm{d}  \tau \lesssim_{ R} \int_0^t\|u\| _{H^{s}}^p \mathrm{d}  \tau,
\end{split}
\end{equation*}
which imply the following estimation for term $L_1$:
\begin{equation}\label{2.20}
\begin{split}
 |L_1| \lesssim_{R,p,\phi}\int_0^t(1+\|n\|_{H^s}^p+ \|u\|_{H^s}^p)\mathrm{d}  \tau.
\end{split}
\end{equation}
For $L_2$, we get by Young inequality that
\begin{equation}\label{2.21}
\begin{split}
|L_2|&\lesssim \int_0^t\theta_R\left(\|u\|_{W^{1,\infty}}\right)\|\Lambda^s u \|_{L^2} ^{p-2}  \sum_{j \geq 1}\| \Lambda^s \textbf{P}  f(t,u)e_j \|_{L^2} ^2  \mathrm{d}  \tau\\
&\lesssim \int_0^t \|\Lambda^s u(\tau)\|_{L^2} ^{p-2}  (1+ \|\Lambda^s u(\tau)\|_{L^2}^2) \mathrm{d}  \tau \lesssim_p \int_0^t (1+ \| u(\tau)\|_{H^s}^p) \mathrm{d}  \tau.
\end{split}
\end{equation}
For $L_3$, by using the BDG inequality, we obtain
\begin{align}\label{2.22}
\mathbb{E}\sup_{\tau \in [0,t]}|L_3|
&\lesssim_{p,R}\mathbb{E}\left(\sup_{\tau\in [0,t]}\|\Lambda^s u\|^p_{L^2}\int_0^t\theta_R^2\left(\|u\|_{W^{1,\infty}}\right)  \|\Lambda^s u(\tau)\| _{L^2} ^{ p-2} (1+\|\Lambda^s u\|^2_{L^2})\mathrm{d}  \tau\right)^{1/2}\nonumber\\
  & \leq  \frac{1}{2}\mathbb{E} \sup_{\tau\in [0,t]}\|\Lambda^s u(\tau)\|^p_{L^2}  + C_{p,R} \mathbb{E}\int_0^t (1+ \|\Lambda^s u(\tau)\|^p_{L^2}) \mathrm{d}  \tau.
\end{align}
By taking the supremum over $[0,t]$ on both sides of \eqref{2.19}, we get from \eqref{2.20}-\eqref{2.22} that
\begin{equation*}
\begin{split}
 \mathbb{E}\sup_{\tau \in [0,t]}\|\Lambda^s u(\tau)\|^p_{L^2}+   \mathbb{E}\int_0^t\|\Lambda^s u\| _{L^2} ^{p-2} \|\nabla\Lambda^s J_ku\|_{L^2 }^2\mathrm{d}  \tau \lesssim_{R,p,\phi} \| u_0\|^p_{H^s} + \int_0^t(1+\|n\|_{H^s}^p+ \|u\|_{H^s}^p)\mathrm{d}  \tau,
\end{split}
\end{equation*}
which together with \eqref{2.17}-\eqref{2.18} yield that
\begin{equation*}
\begin{split}
  \mathbb{E}\sup_{\tau \in [0,t]}\|(n,c,u)\|^p_{\textbf{H}^s}\lesssim_{R,s,p,\phi,\kappa,\chi,\epsilon}\|(n_0,c_0,u_0)\|^p_{\textbf{H}^s} + \int_0^t(\| n(\tau)\|_{H^s}^{p} +\| c(\tau)\|_{H^s}^{p}+\| u(\tau)\|_{H^s}^{p})  \mathrm{d}  \tau.
\end{split}
\end{equation*}
An application of Gronwall Lemma leads to
\begin{equation}\label{dfd}
\begin{split}
 \mathbb{E}\sup_{t \in [0,T]}\|(n,c,u)\|^p_{\textbf{H}^s} \lesssim_{R,s,p,\phi,\kappa,\chi,\epsilon,n_0,c_0,u_0,T} 1,~~\forall T>0.
\end{split}
\end{equation}
Thereby  $(n,c,u)\in L^p(\Omega; \mathcal {C}([0,T];\textbf{H}^s(\mathbb{R}^2) ))$, for any $T>0$ and $p>2$.

\textsf{Step 3: H\"{o}lder regularity.} According to \eqref{dfd},  it is standard to verify that
\begin{equation}\label{2.23}
\begin{split}
(n,c)\in L^p\left(\Omega; \textrm{Lip}([0,T]; H^{s-2}(\mathbb{R}^2)\times H^{s-2}(\mathbb{R}^2) )\right).
\end{split}
\end{equation}
Now, we prove that the $u$-component is fractionally differentiable, that is, $u\in \mathcal {C}^{\alpha}([0,T];H^{s-2}(\mathbb{R}^2))$ $\mathbb{P}$-a.s.  Indeed, by $\eqref{Mod-2}_3$, we infer that
\begin{equation}\label{2.24}
\begin{split}
\|u(t)-u(\tau)\|_{H^{s-2 }}  \leq& \left\|\int_\tau^t \Delta J_k^2u\mathrm{d}  s \right\|_{H^{s-2 }}+ \left\|\int_\tau^t \textbf{P} J_k ((n\nabla \phi)*\rho  )\mathrm{d}  s\right\|_{H^{s}} \\
  & +\left\|\int_\tau^t\theta_R (\|u \|_{W^{1,\infty}} ) \textbf{P} J_k  (J_k u \cdot \nabla J_k u  ) \mathrm{d}  s\right\|_{H^{s-1}}\\
 & +\left\|\int_\tau^t\theta_R (\|u^{k,\epsilon}\|_{W^{1,\infty}} )    \textbf{P}  f(t,u^{k,\epsilon}) \mathrm{d}  W\right\|_{H^{s }} \\
:=& D_1+D_2+D_3+D_4 .
\end{split}
\end{equation}
Thanks to Lemma \ref{lem1}, we see that $
D_1+D_2+D_3\lesssim _{\phi,R}\int_\tau^t  (1+\|n\|_{H^s}+\|u\|_{H^s} ) \mathrm{d}  s$, which implies
\begin{equation}\label{2.25}
\begin{split}
\mathbb{E}\left(D_1+D_2+D_3\right)^p
&\lesssim _{p,\phi,R}\mathbb{E}  \sup_{s\in [\tau,t]}\left(1+\|n(s)\|_{H^s}^p+\|u(s)\|_{H^s}^p \right) |t-\tau|^p \\
&\lesssim _{p,\phi,R,n_0,c_0,u_0,\epsilon,T} |t-\tau|^p.
\end{split}
\end{equation}
For $D_4$, first note that for any $\gamma>0$, there must be a subinterval $[\tau',t']\subset [\tau,t]$ such that
$$
\sup_{t\neq \tau}\frac{D_4(\tau,t)}{|t-\tau|^\sigma}< \sup_{t\neq \tau}\frac{D_4(\tau',t')}{|t'-\tau'|^\sigma}+\gamma^{\frac{1}{p}}.
$$
It then follows from the BDG inequality that
\begin{equation*}
\begin{split}
\mathbb{E} \left(\sup_{t\neq \tau}\frac{D_4(\tau,t)}{|t-\tau|^\sigma}\right) ^p 
&\lesssim_p   \frac{\mathbb{E}(\int_{\tau'}^{t'}\theta_R^2(\|u \|_{W^{1,\infty}})   \|f(t,u )\|_{L_2(U;L^2)}^2 \mathrm{d}  s)^\frac{p}{2}}{|t'-\tau'|^{\sigma p}} +C(p)\gamma\\
&\lesssim_{p,R} |t'-\tau'|^{\frac{p}{2}-\sigma p} \mathbb{E}\sup_{s\in [0,T]}(1+ \|u(s)\|_{H^s}^p)  + \gamma\\
& \lesssim_{p,R,n_0,c_0,u_0,\epsilon,T}  |t'-\tau'|^{\frac{p}{2}-\sigma p}+ \gamma.
\end{split}
\end{equation*}
Since $\gamma>0$ is arbitrary and $\frac{p}{2}-\sigma p>0$  as $\sigma\in (\frac{1}{p},\frac{1}{2})$, it follows from the last estimate that
\begin{equation}\label{2.26}
\begin{split}
\mathbb{E} \left\|\int_\tau^t\theta_R(\|u \|_{W^{1,\infty}})    \textbf{P}  f(t,u ) \mathrm{d}  W\right\|_{H^{s }}^p \lesssim_{p,R,n_0,c_0,u_0,\epsilon,T} |t-\tau|^{\sigma p}.
\end{split}
\end{equation}
Combining the estimates \eqref{2.25} with \eqref{2.26} lead to
$$
\mathbb{E} \|u(t)-u(\tau)\|_{H^{s-2}} ^p \lesssim_{p,R,n_0,c_0,u_0,\epsilon,T}  |t-\tau|^{\sigma p}.
$$
By means of the Kolmogorov Continuity Theorem (cf. Theorem 3.3 in \cite{da2014stochastic}), we infer that the solution $u$ has an undistinguishable  version in $  \mathcal {C}^{\alpha}([0,T];H^{s-2}(\mathbb{R}^2))$, for all  $  0\leq {\alpha}\leq\sigma -\frac{1}{p} \leq\frac{1}{2}-\frac{1}{p}$. The proof of Lemma \ref{lem2} is completed.
\end{proof}

\subsection{Taking the limits $k\rightarrow\infty$ and $R\rightarrow\infty$}

Our purpose is to show that the family of
$\{\textbf{y}^{k}=(n^{k} , c^{k}, u^{k})\}_{k\in\mathbb{N}}$ comprises a subsequence which converges strongly in $\mathcal {C}([0,T];\textbf{H}^{s-3}(\mathbb{R}^2))$ with $s>5$,  $\mathbb{P}$-a.s. The key ingredient of the proof lies in proving that $\textbf{y}^{k} $ converges in probability (up to an subsequence) to an element $\textbf{y}:=(n, c, u)$ in $L^\infty(0,T;\textbf{H}^{s-3}(\mathbb{R}^2))$ as $k\rightarrow\infty$. Here and in the following of this subsection, we will simply write $\textbf{y}^{k}$, $\textbf{y}$, $F ^{k}(\cdot)$ and $G(\cdot)$ instead of $\textbf{y}^{k,R,\epsilon}$, $\textbf{y}^{R,\epsilon}$, $F ^{k,R}(\cdot)$ and $G^{R} (\cdot)$, respectively.

Being inspired by the direct convergence method in \cite{li2021stochastic}, we first prove that
\begin{equation}\label{2.27}
\begin{split}
\lim_{k\rightarrow\infty}\sup_{l\geq k} \mathbb{P} \left\{\sup_{t\in [0,T]} \|\textbf{y}^{k,R,\epsilon}-\textbf{y}^{l,R,\epsilon}\|_{\textbf{H}^{s-3} }> \frac{1}{k} \right\}=0.
\end{split}
\end{equation}
To this end, for any fixed $R>0$ and $\epsilon \in (0,1)$, let $\textbf{y}^{k}$ and $\textbf{y}^{l}$ be solutions to \eqref{Mod-2} associated to the frequency truncations $J_k$ and $J_l$, respectively. Define $\textbf{y}^{k,l}:=\textbf{y}^{k}-\textbf{y}^{l}$, it follows from \eqref{Mod-2} that
\begin{equation}\label{2.28}
\left\{
\begin{aligned}
& \mathrm{d}  \textbf{y}^{k,l}= (F ^{k} (\textbf{y}^{k}  )-F ^{l } (\textbf{y}^{l}  ))  \mathrm{d} t
+(G (\textbf{y}^{k}  )-G   (\textbf{y}^{l}  ))\mathrm{d}  \mathcal {W}(t),\\
& \textbf{y}^{k,l}(0) =0,
\end{aligned}
\right.
\end{equation}
where
\begin{equation*}
\begin{split}
& F ^{k}_1 (\textbf{y}^{k}  )-F ^{l}_1 (\textbf{y}^{l}  )\\
&=\Delta J_k^2  n^{k,l}+(J_k^2-J_l^2)\Delta n^l+\theta_R(\|(u^k,n^k)\|_{W^{1,\infty}})[J_l(J_l u^l\cdot \nabla J_l n^l)- J_k(J_k u^k\cdot \nabla J_k n^k)]   \\
&\quad +[\theta_R(\|(u^l,n^l)\|_{W^{1,\infty}}) -\theta_R(\|(u^k,n^k)\|_{W^{1,\infty}}) ] J_l(J_l u^l\cdot \nabla J_l n^l)\\
 &\quad+[\theta_R(\|(n^l,c^l)\|_{W^{1,\infty}})-\theta_R(\|(n^k,c^k)\|_{W^{1,\infty}})] \textrm{div}  J_k(J_kn^k[\chi(J_kc^k)\nabla J_kc^k]*\rho^{\epsilon}) \\
 &\quad+\theta_R(\|(n^l,c^l)\|_{W^{1,\infty}})[ \textrm{div}  J_l(J_ln^l[\chi(J_lc^l)\nabla J_lc^l]*\rho^{\epsilon})- \textrm{div}  J_k(J_kn^k[\chi(J_kc^k)\nabla J_kc^k]*\rho^{\epsilon})]\\
 &:= \Delta J_k^2  n^{k,l}+ M_1+\cdots + M_5,
 \end{split}
\end{equation*}
\begin{equation*}
\begin{split}
&F ^{k}_2 (\textbf{y}^{k}  )-F ^{l}_2 (\textbf{y}^{l}  )\\
&=\Delta J_k^2  c^{k,l}+(J_k^2-J_l^2)\Delta c^l  +[\theta_R(\|(u^l,c^l)\|_{W^{1,\infty}}) \quad-\theta_R(\|(u^k,c^k)\|_{W^{1,\infty}}) ] J_l(J_l u^l\cdot \nabla J_k c^l) \\
&\quad+\theta_R(\|(u^k,c^k)\|_{W^{1,\infty}})[J_l(J_l u^l\cdot \nabla J_l c^l)- J_k(J_k u^k\cdot \nabla J_k c^k)]\\
 &\quad+[\theta_R(\|(n^l,c^l)\|_{W^{1,\infty}})-\theta_R(\|(n^k,c^k)\|_{W^{1,\infty}})]J_k([J_kn^k \kappa(J_kc^k)]*\rho^{\epsilon}) \\
 &\quad+\theta_R(\|(n^l,c^l)\|_{W^{1,\infty}})[J_l([J_ln^l \kappa(J_lc^l)]*\rho^{\epsilon})-J_k([J_kn^k \kappa(J_kc^k)]*\rho^{\epsilon})]\\
 &:= \Delta J_k^2  c^{k,l}+ N_1+\cdots + N_5,
 \end{split}
\end{equation*}
and
\begin{equation*}
\begin{split}
&F ^{k}_3 (\textbf{y}^{k}  )-F ^{l}_3 (\textbf{y}^{l}  )\\
&=\Delta J_k^2  u^{k,l}+(J_k^2-J_l^2)\Delta u^l  +[\theta_R(\| u^l \|_{W^{1,\infty}}) -\theta_R(\| u^k \|_{W^{1,\infty}}) ] J_l(J_l u^l\cdot \nabla J_l c^l)\\
&\quad+ \theta_R(\| u^k \|_{W^{1,\infty}}) \textbf{P} [  J_l(J_l u^l\cdot \nabla J_l u^l)- J_k(J_k u^k\cdot \nabla J_k u^k)]      \\
&\quad + \textbf{P} [J_k((J_kn^k\nabla \phi)*\rho^{\epsilon})- J_l((J_ln^l\nabla \phi)*\rho^{\epsilon})]\\
 &:=\Delta J_k^2  c^{k,l}+ P_1+\cdots + P_4.
 \end{split}
\end{equation*}
Concerning the stochastic integrals, we have
\begin{equation*}
\begin{split}
&\int_0^t(\textbf{y}^{k},(G (\textbf{y}^{k}  )-G (\textbf{y}^{l}  ))\mathrm{d}  \mathcal {W})_{\textbf{H} ^{s-3}}\\
 &=\sum_{j\geq 1}\left(\int_0^t(\theta_R (\|u^{k} \|_{W^{1,\infty}} )-\theta_R (\|u^{l} \|_{W^{1,\infty}} )) (u^k,\textbf{P}   f_j(r,u^{k} ))_{H^{s-3}}\mathrm{d}  W ^j\right.\\
 &\left.\quad +\int_0^t\theta_R (\|u^{l} \|_{W^{1,\infty}} )  (u^k,\textbf{P}   f_j(r,u^{k} )-  \textbf{P} f_j(r,u^{l}  ))_{H^{s-3}}\mathrm{d}  W ^j\right) \\
 &:= \sum_{j\geq 1}\left(\int_0^t Q_j\mathrm{d}  W ^j+\int_0^t S_j\mathrm{d}  W ^j\right).
 \end{split}
\end{equation*}

Let us now derive some useful estimations for terms associated to $M_j,N_j,P_j,Q_j$ and $S_j$, respectively.
\begin{lemma}\label{lem3}
For the integrals associated to $M_i$, $i=1,...,5$, we have
\begin{subequations}
\begin{align}
& |(n^{k,l},M_1)_{H^{s-3}}| \lesssim \|n^{k,l}\|_{H^{s-3}}^2+\max\{\frac{1}{k^2},\frac{1}{l^2}\} \|  n^l\|_{H^{s}}^2 ,\label{2.30a}\\
&|(n^{k,l},M_2)_{H^{s-3}}| \lesssim  \|n^{k,l}\|_{H^{s-3}} ^2+  \max\{\frac{1}{k^{2}} ,\frac{1}{l^{2}} \}\big( \| u^k\|_{H^{s}}^4+\| u^l\|_{H^{s}}^4+\| n^l\|_{H^{s}}^4+\| n^k\|_{H^{s}} ^4\big),\label{2.30b}\\
&|(n^{k,l},M_3)_{H^{s-3}}| \lesssim\|(u^{k,l},n^{k,l})\|_{H^{s-3}}^2+ \frac{1}{l^2} (\|u^{k}\|_{H^{s}}^2+\|n^{k}\|_{H^{s}}^2+\|u^{l}\|_{H^{s}}^2+\|n^{l}\|_{H^{s}}^2),\label{2.30c}\\
&|(n^{k,l},M_4)_{H^{s-3}}| \lesssim_ \chi \|(n^{k,l},c^{k,l})\|_{H^{s-3}}^2+  \frac{1}{k^2} (\|n^k\|_{H^{s}} ^2+\|n^l\|_{H^{s}} ^2+\|c^k \|_{H^{s}}^2),\label{2.30d}\\
&|(n^{k,l},M_5)_{H^{s-3}}| \lesssim_ \chi \|n^{k,l}\|_{H^{s-3}} ^2   +  \max\{\frac{1}{k^2},\frac{1}{l^2}\} (\|n^k \|_{H^{s}}^2+\|n^l \|_{H^{s}}^2+\|c^l \|_{H^{s}}^2+\|c^k\|_{H^{s}}^2).\label{2.30e}
\end{align}
\end{subequations}
\end{lemma}

\begin{proof}[\emph{\textbf{Proof}}]
By using the Bernstein inequality \eqref{2.1}  and the fact of $J_k^2-J_l^2=(J_k-J_l)(J_k+J_l)$, we get
\begin{equation*}
\begin{split}
 |(n^{k,l},M_1)_{H^{s-3}}|
 &\leq \|n^{k,l}\|_{H^{s-3}}\|(J_k +J_l )(J_k -J_l )\Delta n^l\|_{H^{s-3}}\\
 & \lesssim \max\{\frac{1}{k},\frac{1}{l}\} \|n^{k,l}\|_{H^{s-3}}\|  n^l\|_{H^{s}} \lesssim \|n^{k,l}\|_{H^{s-3}}^2+\max\{\frac{1}{k^2},\frac{1}{l^2}\} \|  n^l\|_{H^{s}}^2.
 \end{split}
\end{equation*}
Generally, it follows from the boundedness of the operator $J_k$ and Bernstein inequality that
\begin{equation*}
\begin{split}
\|J_l n^l- J_k n^k\|_{H^{s-d}}\lesssim\left\{
\begin{aligned}
& \max\{\frac{1}{k^d},\frac{1}{l^d}\} ( \|n^k\|_{H^{s}}+\|n^l\|_{H^{s}}),\\
&   \max\{\frac{1}{k^d},\frac{1}{l^d}\} \|n^l\|_{H^{s}}+ \|n^{k,l}\|_{H^{s-d}} ,
\end{aligned}
\right.\quad  \textrm{for all} ~s \in \mathbb{R},~d\in \mathbb{N}.
 \end{split}
\end{equation*}
By virtue of the Moser-type estimate (\cites{bahouri2011fourier,miao2012littlewood}) and the fact that $H^{s-3}(\mathbb{R}^2)$ is a Banach algebra as $s>5$, we have
\begin{equation*}
\begin{split}
 |(n^{k,l},M_2)_{H^{s-3}}| &=  \theta_R(\|(u^k,n^k)\|_{W^{1,\infty}}) |(n^{k,l},(J_l-J_k)(J_l u^l\cdot \nabla J_l n^l))_{H^{s-3}}|\\
 &\quad+|(n^{k,l},J_k[ (J_l u^l-J_ku^k)\cdot \nabla J_l n^l +  J_k u^k\cdot \nabla( J_l n^l- J_k n^k) ] )\\
 & \lesssim  \|n^{k,l}\|_{H^{s-3}} \left(\frac{1}{l^5}\| u^l\|_{H^{s}}\| n^l\|_{H^{s}}+ \max\{\frac{1}{k^3l^2},\frac{1}{l^5} \}(\|n^k\|_{H^{s}}+\|n^l\|_{H^{s}})\| n^l \|_{H^{s}}\right.\\
 &\quad\left.+ \max\{\frac{1}{k^3l^2},\frac{1}{k^5}\} \|  u^k\|_{H^{s-3}} (\|n^k\|_{H^{s-2}}+\|n^l\|_{H^{s-2}})\right)\\
 &\lesssim  \|n^{k,l}\|_{H^{s-3}} ^2+  \max\{\frac{1}{l^{2}}, \frac{1}{k^{2}} \}\big( \| u^k\|_{H^{s}}^4+\| u^l\|_{H^{s}}^4+\| n^l\|_{H^{s}}^4+\| n^k\|_{H^{s}} ^4\big).
 \end{split}
\end{equation*}
By using the Mean Value Theorem, the embedding $H^{s}(\mathbb{R}^2) \subset H^{s-3}(\mathbb{R}^2)\subset W^{1,\infty}(\mathbb{R}^2)$ as well as the boundedness of  $\theta'(\cdot)$, we have
\begin{equation*}
\begin{split}
 |(n^{k,l},M_3)_{H^{s-3}}| 
 &\lesssim \|(u^{k,l},n^{k,l})\|_{H^{s-3}}(\|(u^{k},n^{k})\|_{H^{s-3}}+\|(u^{l},n^{l})\|_{H^{s-3}})\|J_l u^l\|_{H^{s-3}}\|\nabla J_l n^l\|_{H^{s-3}}\\
 &\lesssim (\|u^{k,l}\|_{H^{s-3}}+\|n^{k,l}\|_{H^{s-3}})\\
 &\quad\times(\| u^{k}\|_{H^{s}}+\|n^{k} \|_{H^{s}}+\| u^{l}\|_{H^{s}}+\|n^{l} \|_{H^{s}})\frac{1}{l^3}\| u^l\|_{H^{s}}\frac{1}{l^2}\| n^l\|_{H^{s}}\\
 &\lesssim \|u^{k,l}\|_{H^{s-3}}^2+\|n^{k,l}\|_{H^{s-3}}^2+ \frac{1}{l^2} (\|u^{k}\|_{H^{s}}^2+\|n^{k}\|_{H^{s}}^2+\|u^{l}\|_{H^{s}}^2+\|n^{l}\|_{H^{s}}^2).
 \end{split}
\end{equation*}
By applying the Paralinearization Theorem (cf. \cite[Theorem 2.89, p.105]{bahouri2011fourier}) in Sobolev spaces $H^s(\mathbb{R}^2)$ with respect to the smooth function $h(t)=\int_0^t \chi(r)\mathrm{d}r$, we have
\begin{equation*}
\begin{split}
 |(n^{k,l},M_4)_{H^{s-3}}| 
 &\lesssim\|(n^{k,l},c^{k,l})\|_{H^{s-3}}(\| n^k \|_{H^{s}}+\| n^l\|_{H^{s}})\|J_kn^k\|_{H^{s-1}} \| \int_0^{J_kc^k}\chi(r)\mathrm{d}r *\rho^{\epsilon} \|_{H^{s-1}} \\
 &\lesssim _ \chi\frac{1}{k} (\|n^{k,l}\|_{H^{s-3}}+\|c^{k,l}\|_{H^{s-3}})(\| n^k \|_{H^{s}}^2+\| n^l\|_{H^{s}}^2) \| \int_0^{J_kc^k}\chi(r)\mathrm{d}r  \|_{H^{s-1}}\\
 &\lesssim _ \chi \|n^{k,l}\|_{H^{s-3}}^2+\|c^{k,l}\|_{H^{s-3}}^2+  \frac{1}{k^2} (\|n^k\|_{H^{s}} ^2+\|n^l\|_{H^{s}} ^2+\|c^k \|_{H^{s}}^2).
  \end{split}
\end{equation*}
For the term associated  to $M_5$, we have
\begin{equation*}
\begin{split}
 |(n^{k,l},M_5)_{H^{s-3}}|
 & \lesssim \|n^{k,l}\|_{H^{s-3}} \left(  \Big\|(J_l-J_k) \Big(J_ln^l[\nabla \int_0^{J_lc^l}\chi(r)\mathrm{d}r*\rho^{\epsilon}]\Big)      \Big\|_{H^{s-2}} \right.\\
 &\left.+    \Big\|(J_l n^l-J_k n^k) [ \nabla  \int_0^{J_kc^l}\chi(r)\mathrm{d}r*\rho^{\epsilon}]\Big\|_{H^{s-2}}\right.\\
 &\left.+    \Big\| J_k n^k [ \nabla  \int_0^{J_kc^k}\chi(r)\mathrm{d}r*\rho^{\epsilon}-\nabla  \int_0^{J_kc^k}\chi(r)\mathrm{d}r*\rho^{\epsilon}]       \Big\|_{H^{s-2}}\right)\\
 & \lesssim_ \chi \|n^{k,l}\|_{H^{s-3}} \left( \frac{1}{l^2} \| n^l \|_{H^{s}}\|c^l\|_{H^{s}}  +  \max\{\frac{1}{k^2},\frac{1}{l^2}\} ( \|n^k\|_{H^{s}}+\|n^l\|_{H^{s}})\|c^k\|_{H^{s}}\right.\\
 &\left.\quad+   \max\{\frac{1}{k},\frac{1}{l}\} \| n^k\|_{H^{s-2}}(\|c^l \|_{H^{s}}+\|c^k\|_{H^{s}})\right)\\
 &\lesssim _ \chi\|n^{k,l}\|_{H^{s-3}} ^2   +  \max\{\frac{1}{k^2},\frac{1}{l^2}\} (\|n^k \|_{H^{s}}^2+\|n^l \|_{H^{s}}^2+\|c^l \|_{H^{s}}^2+\|c^k\|_{H^{s}}^2),
  \end{split}
\end{equation*}
where the last inequality used   the Paralinearization Theorem and Bernstein inequality such that
$
 \| \int_{0}^{J_lc^l}\chi(r)\mathrm{d}r \|_{H^{s-1}}\lesssim \|J_kc^k\|_{H^{s-1}}\lesssim \frac{1}{k}\|c^k\|_{H^{s}}.
$
This completes the proof of Lemma \ref{lem3}.
\end{proof}

\begin{lemma}\label{lem4}
For the integrals associated to $N_i$, $i=1,...,5$, we have
\begin{subequations}
\begin{align}
& |(c^{k,l},N_1)_{H^{s-3}}| \lesssim \|c^{k,l}\|_{H^{s-3}}^2+\max\{\frac{1}{k^2},\frac{1}{l^2}\} \|c^l\|_{H^{s}}^2 \label{2.31a},\\
&|(c^{k,l},N_2)_{H^{s-3}}| \lesssim \|(u^{k,l},c^{k,l})\|_{H^{s-3}}^2 + \frac{1}{ l^2}(\|u^l\|_{H^{s}}^4+\|c^l\|_{H^{s}}^4),\label{2.31b}\\
&|(c^{k,l},N_3)_{H^{s-3}}| \lesssim \|(u^{k,l},c^{k,l})\|_{H^{s-3}}^2 + \frac{1}{k^2}(\|c^k\|_{H^{s-3}}^2+\|c^l\|_{H^{s-3}}^2+\|u^l\|_{H^{s-3}}^2),\label{2.31c}\\
&|(c^{k,l},N_4)_{H^{s-3}}| \lesssim_\kappa \|(n^{k,l},c^{k,l})\|_{H^{s-3}}^2+\frac{1}{k^2} (\|n^k\|_{H^{s}}^6 +\| c^{k} \|_{H^{s}}^6+\| c^{l} \|_{H^{s}}^6), \label{2.31d}\\
&|(c^{k,l},N_5)_{H^{s-3}}| \lesssim_\kappa \| c^{k,l} \|_{H^{s-3}}^2 + \max\{\frac{1}{k^2},\frac{1}{l^2}\}(\| n^l \|_{H^{s}}^4+\| c^l \|_{H^{s}}^4+\| n^k \|_{H^{s}}^4+\| c^k\|_{H^{s}}^4).\label{2.31e}
\end{align}
\end{subequations}
\end{lemma}
\begin{proof}[\emph{\textbf{Proof}}]
The proof of \eqref{2.31a}, \eqref{2.31b} and \eqref{2.31c} is similar to \eqref{2.30a}, \eqref{2.30c} and \eqref{2.30b}, respectively. For \eqref{2.31d}, we get by the Mean Value Theorem and the fact of $\kappa(0)=0$ that
\begin{equation*}
\begin{split}
 |(c^{k,l},N_4)_{H^{s-3}}| 
 &\lesssim (\|n^{k,l}\|_{H^{s-3}}+\|c^{k,l}\|_{H^{s-3}})(\| c^{k} \|_{H^{s}}+\| c^{l} \|_{H^{s}})\|J_kn^k\|_{H^{s-3}}\| \kappa(J_kc^k)\|_{H^{s-3}}\\
 &\lesssim \frac{1}{k^6}(\|n^{k,l}\|_{H^{s-3}}+\|c^{k,l}\|_{H^{s-3}})(\| c^{k} \|_{H^{s}}+\| c^{l} \|_{H^{s}})\|n^k\|_{H^{s}}\| c^k\|_{H^{s}}\\
 &\lesssim \|n^{k,l}\|_{H^{s-3}}^2+\|c^{k,l}\|_{H^{s-3}}^2+\frac{1}{k^2} (\|n^k\|_{H^{s}}^6 +\| c^{k} \|_{H^{s}}^6+\| c^{l} \|_{H^{s}}^6) .
  \end{split}
\end{equation*}
For $(c^{k,l},N_5)_{H^{s-3}}$, first note that
\begin{equation*}
\begin{split}
\| N_5\|_{H^{s-3}} ^2&=\|J_l([J_ln^l \kappa(J_lc^l)]*\rho^{\epsilon})-J_k([J_kn^k \kappa(J_kc^k)]*\rho^{\epsilon})\|_{H^{s-3}}\\
&   \lesssim \max\{\frac{1}{k},\frac{1}{l}\}\| J_ln^l \kappa(J_lc^l) \|_{H^{s}}+\frac{1}{k}\|J_ln^l \kappa(J_lc^l)\|_{H^{s-3}}+\frac{1}{k}\|J_kn^k \kappa(J_kc^k)\|_{H^{s-3}}\\
&   \lesssim \max\{\frac{1}{k},\frac{1}{l}\}(\| n^l \|_{H^{s}}^2+\| c^l \|_{H^{s}}^2+\| n^k \|_{H^{s}}^2+\| c^k\|_{H^{s}}^2).
  \end{split}
\end{equation*}
Then we get by H\"{o}lder inequality that
\begin{equation*}
\begin{split}
 |(c^{k,l},N_5)_{H^{s-3}}| &\lesssim \| c^{k,l} \|_{H^{s-3}}^2 + \| N_5\|_{H^{s-3}} ^2\\
 &\lesssim \| c^{k,l} \|_{H^{s-3}}^2 + \max\{\frac{1}{k^2},\frac{1}{l^2}\}(\| n^l \|_{H^{s}}^4+\| c^l \|_{H^{s}}^4+\| n^k \|_{H^{s}}^4+\| c^k\|_{H^{s}}^4).
  \end{split}
\end{equation*}
This completes the proof of Lemma \ref{lem4}.
\end{proof}

\begin{lemma}\label{lem5}
For the integrals associated to $P_i$, $i=1,...,5$,  we have
\begin{subequations}
\begin{align}
& |(u^{k,l},P_1)_{H^{s-3}}| \lesssim \|u^{k,l}\|_{H^{s-3}}^2+\max\{\frac{1}{k^2},\frac{1}{l^2}\} \|  u^l\|_{H^{s}}^2 ,\label{2.32a}\\
&|(u^{k,l},P_2)_{H^{s-3}}| \lesssim  \|u^{k,l}\|_{H^{s-3}} ^2+  \frac{1}{l^{2}} \big( \| u^l\|_{H^{s}}^4+\| c^l\|_{H^{s}}^4\big),\label{2.32b}\\
&|(u^{k,l},P_3)_{H^{s-3}}| \lesssim  \|u^{k,l}\|_{H^{s-3}}^2(\|u^k\|_{H^{s}}^2+\|u^l\|_{H^{s}}^2)+ \max\{\frac{1}{k^2},\frac{1}{l^2}\}(1+\|u^{k}\|_{H^{s}}^4+\|u^{l}\|_{H^{s}}^4 ),\label{2.32c}\\
&|(u^{k,l},P_4)_{H^{s-3}}|\lesssim_{\phi,\epsilon}  \|u^{k,l}\|_{H^{s-3}}^2+ \max\{\frac{1}{k^2},\frac{1}{l^2}\} (\|n^k \|_{H^s}^2+\|n^l \|_{H^s}^2).\label{2.32d}
\end{align}
\end{subequations}
For the terms with respect to $Q_j$ and $S_j$, $j\geq 1$,  we have
\begin{equation}\label{2.33}
\begin{split}
\sum_{j\geq 1} |Q_j|^2\lesssim \|u^{k,l}\|_{H^{s-3}}^2(1+\|u^{k}\|_{H^{s}}^4+\|u^{l}\|_{H^{s}}^4),~~~ \sum_{j\geq 1} |S_j|^2\lesssim \|u^{k,l}\|_{H^{s-3}}^2 \|u^{k}\|_{H^{s}}^2.
  \end{split}
\end{equation}
\end{lemma}

\begin{proof}[\emph{\textbf{Proof}}]
The proof of \eqref{2.32a} and \eqref{2.32b} is similar to \eqref{2.31a} and \eqref{2.31b}, respectively. For \eqref{2.32c}, one can estimate  as
\begin{equation*}
\begin{split}
&|(u^{k,l},P_3)_{H^{s-3}}| \\
&\quad\lesssim \frac{1}{l}\|u^{k,l}\|_{H^{s-3}}\|u^{l}\|_{H^{s}}^2+\|u^{k,l}\|_{H^{s-3}}\Big(\|(J_l u^l-J_ku^k)\nabla J_lu^l    \|_{H^{s-3}}+ \|J_ku^k\nabla(J_l u^l-J_ku^k) \|_{H^{s-3}}\Big)\\
&\quad\lesssim \|u^{k,l}\|_{H^{s-3}}^2+ \frac{1}{l^2} \|u^{l}\|_{H^{s}}^4+ (\max\{\frac{1}{k^2},\frac{1}{l^2}\}\|u^{k}\|_{H^{s}}+ \|u^{k,l}\|_{H^{s-3}} )\|u^{k,l}\|_{H^{s-3}}\|u^l    \|_{H^{s}} \\
&\quad\quad+ \max\{\frac{1}{k},\frac{1}{l}\} \|u^k\|_{H^{s}}\|u^{k,l}\|_{H^{s-3}}(\|u^k\|_{H^{s}}+\|u^l\|_{H^{s}})\\
&\quad\lesssim \|u^{k,l}\|_{H^{s-3}}^2(\|u^k\|_{H^{s}}^2+\|u^l\|_{H^{s}}^2)+ \max\{\frac{1}{k^2},\frac{1}{l^2}\}(1+\|u^{k}\|_{H^{s}}^4+\|u^{l}\|_{H^{s}}^4 ).
  \end{split}
\end{equation*}
For the term associated to \eqref{2.31d}, by using the boundedness of  $\textbf{P}$ in $L^p(\mathbb{R}^2)$ ($1< p <\infty$) and the isometry property of operators $\Lambda^{s-3}$ from $H^{s-3}(\mathbb{R}^2)$ into $L^2(\mathbb{R}^2)$, we see that
\begin{equation*}
\begin{split}
|(u^{k,l},P_4)_{H^{s-3}}| &\lesssim  \|u^{k,l}\|_{H^{s-3}}(\|(J_kn^k\nabla \phi)*\rho^{\epsilon}\|_{H^{s-3}}+\|(J_ln^l\nabla \phi)*\rho^{\epsilon}\|_{H^{s-3}})\\
&\lesssim  \|u^{k,l}\|_{H^{s-3}}(\|(J_kn^k\nabla \phi)*\Lambda^{s-3}\rho^{\epsilon}\|_{L^2}+\|(J_ln^l\nabla \phi)*\Lambda^{s-3}\rho^{\epsilon}\|_{L^2})\\
&\lesssim_{\phi,\epsilon}  \|u^{k,l}\|_{H^{s-3}}^2+ \max\{\frac{1}{k^2},\frac{1}{l^2}\} (\|n^k \|_{H^s}^2+\|n^l \|_{H^s}^2).
  \end{split}
\end{equation*}
For the first estimate in \eqref{2.33}, it follows from the Mean Value Theorem and the embedding $H^{s-3}(\mathbb{R}^2)\subset W^{1,\infty}(\mathbb{R}^2)$ that
\begin{equation*}
\begin{split}
 \sum_{j\geq 1} |Q_j|^2&\lesssim \sum_{j\geq 1}|\theta_R (\|u^{k} \|_{W^{1,\infty}} )-\theta_R (\|u^{l} \|_{W^{1,\infty}} )| ^2 \|u^k\|_{H^{s-3}}^2\|\textbf{P}   f_j(t,u^{k} )\|_{H^{s-3}}^2\\
 &\lesssim\|u^{k,l}\|_{W^{1,\infty}}^2 \|u^k\|_{H^{s-3}}^2\sum_{j\geq 1}\| f_j(t,u^{k} )\|_{H^{s-3}}^2\\
 &\lesssim\|u^{k,l}\|_{H^{s-3}}^2 \|u^k\|_{H^{s-3}}^2(1+\|u^k\|_{H^{s-3}}^2) \lesssim\|u^{k,l}\|_{H^{s-3}}^2 (1+\|u^k\|_{H^{s}}^4).
 \end{split}
\end{equation*}
For the second term in \eqref{2.33}, we get from the condition {(\textsf{A$_3$})} that
\begin{equation*}
\begin{split}
 \sum_{j\geq 1} |S_j|^2 \lesssim \sum_{j\geq 1} \|u^k\|_{H^{s-3}}^2 \|\textbf{P}   f_j(t,u^{k} )-  \textbf{P} f_j(t,u^{l}  )\|_{H^{s-3}}^2\lesssim \|u^{k,l}\|_{H^{s-3}}^2\|u^k\|_{H^{s}}^2.
 \end{split}
\end{equation*}
This completes the proof of Lemma \eqref{lem5}.
\end{proof}

Now we are ready to prove that the sequence $\{\textbf{y}^{k}\}_{k\geq1}$ (up to a subsequence) converges strongly in $  \mathcal {C}([0,T];H^{s-3}(\mathbb{R}^2))$, $\mathbb{P}$-a.s. Precisely, we have the following result.

\begin{lemma}\label{lem6}
Let $s>5$, $p\geq 2$. For any $R>1$, $0<\epsilon<1$ and $T>0$, there exists a subsequence of $\{\textbf{y} ^{k,R,\epsilon} \}_{k\geq1}$ (still denoted by itself) and a $\mathcal {F}_t$-progressively measurable stochastic  process $\textbf{y}^{ R,\epsilon} \in L^p(\Omega; L^\infty(0,T;\textbf{H}^{s}(\mathbb{R}^2)))$, such that
\begin{equation}
\begin{split}
\mathbb{P} \left\{\textbf{y}^{k,R,\epsilon}\rightarrow \textbf{y}^{R,\epsilon}~ \textrm{strongly in}~ \mathcal {C}(0,T;\textbf{H}^{s-3}(\mathbb{R}^2)) ~ \textrm{as}~ k\rightarrow\infty \right\}=1.
 \end{split}
\end{equation}
\end{lemma}

\begin{proof}[\emph{\textbf{Proof}}]
Applying the operator $\Lambda^{s-3}:=(1-\Delta)^{\frac{s-3}{2}}$ to \eqref{2.28}, and then using It\^{o}'s formula to  $\mathrm{d}  \|\Lambda^{s-3}\textbf{y}^{k,l}\|_{L^2}^2$ with respect to the resulting system, one can infer that
\begin{equation}\label{2.34}
\begin{split}
 & \|\textbf{y}^{k,l}(t)\|_{\textbf{H}^{s-3}}^2 + 2 \int_0^t\|\nabla J_k \textbf{y}^{k,l}\|_{\textbf{H}^{s-3}}^2 \mathrm{d}r= 2 \sum_{1\leq i \leq 5}\int_0^t(n^{k,l},M_i)_{H^{s-3}} \mathrm{d}r\\
  &\quad +2 \sum_{1\leq i \leq 5}\int_0^t(c^{k,l},N_i)_{H^{s-3}} \mathrm{d}r+2 \sum_{1\leq i \leq 4}\int_0^t(u^{k,l},P_i)_{H^{s-3}} \mathrm{d}r+  \sum_{j\geq 1}\int_0^t|Q_j+S_j| ^2  \mathrm{d}r \\
 & \quad+2 \int_0^t(\textbf{y}^{k},(G^{k} (\textbf{y}^{k}  )-G ^{l } (\textbf{y}^{l}  ))\mathrm{d}  \mathcal {W})_{\textbf{H} ^{s-3}}.
 \end{split}
\end{equation}
Plugging the estimates in Lemma \ref{lem3}-Lemma \ref{lem6} together, we get by Young inequality that
\begin{equation}\label{2.35}
\begin{split}
 &\textrm{The first four terms on the R.H.S. of} ~\eqref{2.34}\\
 &\quad \lesssim_{\chi,\kappa,\phi,\epsilon}  \int_0^t (1+\|\textbf{y}^{k }(r)\|_{\textbf{H}^{s }}^4+\|\textbf{y}^{ l}(r)\|_{\textbf{H}^{s }}^4)\|\textbf{y}^{k,l}(r)\|_{\textbf{H}^{s-3}}^2 \mathrm{d}r\\
 &\quad \quad+\max\{\frac{1}{k^2},\frac{1}{l^2}\}\int_0^t(1+\|\textbf{y}^{k }(r)\|_{\textbf{H}^{s }}^6+\|\textbf{y}^{l}(r)\|_{\textbf{H}^{s }}^6 )\mathrm{d}r.
 \end{split}
\end{equation}
For all $k,l \in \mathbb{N}$ and $N \geq 1$, we define a sequence of stopping times
\begin{equation}\label{stop1}
\begin{split}
\textbf{t}^{k,l,T,N}:=\textbf{t}^{k,T,N}\bigwedge\textbf{t}^{l,T,N},~~ ~\textbf{t}^{k,T,N} =T \bigwedge \inf\{t\geq 0;~ \|\textbf{y}^{k }(t)\|_{\textbf{H}^{s }}\geq N\}.
 \end{split}
\end{equation}
By raising the $p$-th power on both sides of \eqref{2.34} and using the estimates in Lemmas \ref{lem3}-\ref{lem6}, we get from \eqref{2.35} and the Young inequality that
\begin{equation} \label{e2}
\begin{split}
 \mathbb{E}\sup_{t\in[0,\textbf{t}^{k,l,T,N}]}\|\textbf{y}^{k,l}(t)\|_{\textbf{H}^{s-3}}^{2p}& \lesssim_{R,\chi,\kappa,\phi,\epsilon,T} \mathbb{E}\int_0^{\textbf{t}^{k,l,T,N}} (1+\|\textbf{y}^{k }\|_{\textbf{H}^{s }}^ {4p}+\|\textbf{y}^{ l}\|_{\textbf{H}^{s }}^{4p})\|\textbf{y}^{k,l}(r)\|_{\textbf{H}^{s-3}}^{2p} \mathrm{d}r\\
  & +\max\{\frac{1}{k^2},\frac{1}{l^2}\}\mathbb{E}\int_0^{\textbf{t}^{k,l,T,N}}(1+\|\textbf{y}^{k }\|_{\textbf{H}^{s }}^{6p}+\|\textbf{y}^{l}\|_{\textbf{H}^{s }}^{6p} )\mathrm{d}r\\
 & +  \mathbb{E}\sup_{t\in[0,\textbf{t}^{k,l,T,N}]}\left| \int_0^t(\textbf{y}^{k},(G  (\textbf{y}^{k}  )-G (\textbf{y}^{l}  ))\mathrm{d}  \mathcal {W})_{\textbf{H} ^{s-3}}\right|.
 \end{split}
\end{equation}
By using the BDG inequality, we have
\begin{equation}\label{e3}
\begin{split}
 &\textrm{Last term in}~\eqref{e2}\\ 
 &\quad\leq \frac{1}{2}\mathbb{E}\sup_{t\in[0,\textbf{t}^{k,l,T,N}]}\|\textbf{y}^{k,l}(t)\|_{\textbf{H}^{s-3}}^{2p}
 +C  \mathbb{E} \left(\int_0^{\textbf{t}^{k,l,T,N}} \|G (\textbf{y}^{k}  )-G  (\textbf{y}^{l} ) \|_{L_2(U;\textbf{H}^{s-3})}^{2}\mathrm{d}r\right)^p \\
 &\quad\leq \frac{1}{2}\mathbb{E}\sup_{t\in[0,\textbf{t}^{k,l,T,N}]}\|\textbf{y}^{k,l}(t)\|_{\textbf{H}^{s-3}}^{2p}
 +C_{T,p} \mathbb{E} \int_0^{\textbf{t}^{k,l,T,N}} \|\textbf{y}^{k,l} (r)\|_{\textbf{H}^{s-3}}^{2p}\mathrm{d}r.
 \end{split}
\end{equation}
It then follows from \eqref{e2}-\eqref{e3} and the definition of $\textbf{t}^{k,l,T,N}$ that
\begin{equation*}
\begin{split}
 &\mathbb{E}\sup_{t\in[0,\textbf{t}^{k,l,T,N}]}\|\textbf{y}^{k,l}(t)\|_{\textbf{H}^{s-3}}^{2p} \\
 &\quad\lesssim_{R,\chi,\kappa,\phi,\epsilon,T} \max\{\frac{1}{k^p},\frac{1}{l^p}\} (1+N^{6p}) +(1+N^{4p}) \mathbb{E}\int_0^{\textbf{t}^{k,l,T,N}} \|\textbf{y}^{k,l}(r)\|_{\textbf{H}^{s-3}}^{2p} \mathrm{d}r,
 \end{split}
\end{equation*}
which combined with the Gronwall Lemma lead to
\begin{equation*}
\begin{split}
 \mathbb{E}\sup_{t\in[0,\textbf{t}^{k,l,T,N}]}\|\textbf{y}^{k,l}(t)\|_{\textbf{H}^{s-3}}^{2p} \lesssim_{R,\chi,\kappa,\phi,\epsilon,T} \max\{\frac{1}{k^p},\frac{1}{l^p}\} (1+N^{6p}) e^{(1+N^{4p})}.
 \end{split}
\end{equation*}
This immediately implies the convergence of $\{\textbf{y}^{k}\}_{k\geq 1}$ in $p$-th momentum up to stopping times, i.e.,
\begin{equation}\label{2.39}
\begin{split}
 \lim_{k\rightarrow\infty} \sup_{l\geq k}\mathbb{E}\sup_{t\in[0,\textbf{t}^{k,l,T,N}]}\|\textbf{y}^{k}(t)-\textbf{y}^{l}(t)
 \|_{\textbf{H}^{s-3}}^{2p}=0,\quad N\geq 1.
 \end{split}
\end{equation}
Now we introduce the following events
$$
E^{k,l}(r):= \left\{\omega;~\sup_{t\in[0,r]}\|\textbf{y}^{k}(t,\omega)-\textbf{y}^{l}(t,\omega)
 \|_{\textbf{H}^{s-3}}^{2p}>\frac{1}{k}\right\},\quad\textrm{for all}~r> 0.
$$
In view of the definition of the stopping times  $\textbf{t}^{k,l,T,N}$ in \eqref{stop1}, we get for any $T>0$ that
\begin{equation*}
\begin{split}
 E^{k,l}(T)
  & = \left(E^{k,l}(T)\bigcap \{\textbf{t}^{k,l,T,N} =T\}\right) \bigcup \left(E^{k,l}(T)\bigcap\{\textbf{t}^{k,l,T,N} <T\}\right) \\
 & \subset E^{k,l}(\textbf{t}^{k,l,T,N})\bigcup \left(E^{k,l}(T)\bigcap\{\textbf{t}^{k,l,T,N} <T\}\right) \\
 & \subset E^{k,l}(\textbf{t}^{k,l,T,N})\bigcup \{\textbf{t}^{k,T,N} \leq T\}\bigcup\{\textbf{t}^{ l,T,N} \leq T\}.
 \end{split}
\end{equation*}
By applying the Chebyshev inequality and the uniform bound in Lemma \ref{lem2}, we get
\begin{equation*}
\begin{split}
 \mathbb{P}\{E^{k,l}(T)\} &\leq \mathbb{P}\{E^{k,l}(\textbf{t}^{k,T,N})\}+\mathbb{P}\{\textbf{t}^{k,T,N} \leq T\}+\mathbb{P}\{\textbf{t}^{l,T,N}\leq T\}\\
 &\lesssim_{R,s,p,\phi,\kappa,\chi,\epsilon,\textbf{y}_0,T} \frac{1}{N^2}+ \mathbb{P}\left\{\sup_{t\in[0,\textbf{t}^{k,l,T,N}]}\|\textbf{y}^{k}(t)-\textbf{y}^{l}(t)
 \|_{\textbf{H}^{s-3}}^{2p}>\frac{1}{k}\right\} .
 \end{split}
\end{equation*}
By \eqref{2.39},  the last inequality provides
\begin{equation*}
\begin{split}
 \lim_{k\rightarrow\infty} \sup_{l\geq k}\mathbb{P}\left\{ \sup_{t\in[0,T]}\|\textbf{y}^{k}(t)-\textbf{y}^{l}(t)
 \|_{\textbf{H}^{s-3}}^{2p}>\frac{1}{k}\right\} \lesssim_{R,s,p,\phi,\kappa,\chi,\epsilon,\textbf{y}_0,T} \frac{1}{N^2}\rightarrow 0,~~ \textrm{as}~~N\rightarrow+\infty,
 \end{split}
\end{equation*}
which indicates that $\textbf{y}^{k}$ converges in probability to an element $\textbf{y}$ in $\mathcal {C}([0,T];\textbf{H}^{s-3}(\mathbb{R}^2) )$. As a result, it follows from the Riesz Theorem that there exists a subsequence of $\{\textbf{y}^{k}\}_{k\geq1}$, still denoted by itself, such that $\textbf{y}^{k}\rightarrow\textbf{y}$ strongly in $\mathcal {C}([0,T];\textbf{H}^{s-3}(\mathbb{R}^2) )$, $\mathbb{P}$-a.s.

To finish the proof, it suffices to verify the spacial regularity of the limit $\textbf{y}$ in $ \mathcal {C}([0,T];\textbf{H}^{s}(\mathbb{R}^2) )$.  Indeed, in view of the uniform bound in Lemma \ref{lem2}, we get by taking $p=2$
\begin{equation*}
\begin{split}
\sup_{k \in\mathbb{N}^+} \mathbb{E} \sup_{t\in [0,T]}\|\textbf{y}^{k}(t)\|^2_{\textbf{H}^s} \leq C,
\end{split}
\end{equation*}
which indicates that there exists a subsequence of $\{\textbf{y}^k\}_{k\geq1}$, still denoted by itself, such that $\textbf{y}^k\rightarrow \varrho$ weak star in $L^2(\Omega;L^\infty(0,T;\textbf{H}^{s}(\mathbb{R}^2) )$.
Since we have proved that $\textbf{y}^{k }\rightarrow \textbf{y}$ in $\mathcal {C}([0,T];\textbf{H}^{s-3}(\mathbb{R}^2))$ as $k\rightarrow\infty$, $\mathbb{P}$-a.s., we infer that $\textbf{y}=\varrho \in L^2(\Omega;L^\infty(0,T;\textbf{H}^{s}(\mathbb{R}^2) )$. This completes the proof of Lemma \ref{lem6}.
\end{proof}

With the help of Lemma \ref{lem6}, we are now ready to prove the existence and uniqueness of global pathwise solutions to the modified systsem \eqref{Mod-2} with cut-off operators.

\begin{lemma}\label{lem7}
Let $s>5$, and assume that the hypothesises {(\textsf{A$_1$})}-{(\textsf{A$_3$})} hold. Then for any fixed $R>1$, $ \epsilon \in (0,1)$ and $T>0$, the system \eqref{Mod-1} with cutoffs has a unique solution $\textbf{y}^{R,\epsilon} \in L^p(\Omega;\mathcal {C}([0,T];\textbf{H}^{s}(\mathbb{R}^2)))$ such that
$$
\mathbb{E}\sup_{t\in [0,T]}\|\textbf{y}^{R,\epsilon} (t)\|_{\textbf{H}^{s}}^p\lesssim_{R,p,\phi,\kappa,\chi,\epsilon,\textbf{y}_0,T} 1,\quad \textrm{for any} ~p\geq 2.
$$
\end{lemma}

\begin{proof}[\emph{\textbf{Proof}}]
\textsf{Step 1 (Existence and regularity).} According to Lemma \ref{lem6} and the continuously embedding from $H^{s-3}(\mathbb{R}^2)$ into $W^{1,\infty}(\mathbb{R}^2)$, one can take the limit as $k\rightarrow \infty$ in \eqref{Mod-2} to conclude that the limit $\textbf{y}$ obtain the in Lemma \ref{lem6} solves \eqref{Mod-1} with cut-off operators. It remains to prove that $\textbf{y} \in L^p(\Omega;\mathcal {C}([0,T];\textbf{H}^{s}(\mathbb{R}^2)))$. Indeed, recalling that (cf. Lemma \ref{lem6})
$$
\textbf{y} \in L^\infty([0,T];\textbf{H}^{s}(\mathbb{R}^2))\bigcap  \mathcal {C}([0,T];\textbf{H}^{s-3}(\mathbb{R}^2)) .
$$
It then follows from the Lemma 1.4 in \cite{temam2001navier} that $\textbf{y} \in  \mathcal {C}_{\textrm{weak}}([0,T];\textbf{H}^{s}(\mathbb{R}^2))$,  that is,  for any $r \in [0,T]$ and smooth function $\varphi\in \mathcal {C}_0^\infty(\mathbb{R}^2)$, we have
\begin{equation}\label{2.40}
\begin{split}
\lim_{t\rightarrow r} (\textbf{y}(t),\varphi)_{\textbf{H}^{s},(\textbf{H}^{s})'}= (\textbf{y}(r),\varphi)_{\textbf{H}^{s},(\textbf{H}^{s})'}.
 \end{split}
\end{equation}

Our next goal is to prove the continuity of the map $t\mapsto \|\textbf{y}(t)\|_{\textbf{H}^{s}}$, which together with \eqref{2.40} imply that $\textbf{y}$ is strongly continuous in time. To overcome the difficulty cased by the low-regularity of convection terms in $H^{s-1}(\mathbb{R}^2)$ which prevents us applying the It\^{o}'s formula in $H^{s}(\mathbb{R}^2)$, we first apply the operator $L_\eta :=\varrho_\eta* $ to \eqref{Mod-1} with cut-off operators, where $\varrho_\eta $ is another  standard spatial mollifier.
By utilizing the It\^{o} formula in Hilbert space to $\mathrm{d}  \|\Lambda^sL_\eta \textbf{y}\|_{\textbf{L}^{2}}^{2p}(p\geq 2)$,  we find
\begin{equation}\label{2.41}
\begin{split}
 &\mathbb{E}\left[\Big(\|\Lambda^sL_\eta \textbf{y}(t_2)\|_{\textbf{L}^{2}} ^{2p} -\|\Lambda^sL_\eta \textbf{y}(t_1)\|_{\textbf{L}^{2}} ^{2p}\Big)^4\right]\\
  &\quad\lesssim_p  \mathbb{E}\left(\int_{t_1}^{t_2} \|\Lambda^sL_\eta \textbf{y}(t)\|_{\textbf{L}^{2}}^{2(p-1)} (\Lambda^sL_\eta \textbf{y}(t), \Lambda^sL_\eta F   (\textbf{y} ))_{\textbf{L}^{2}}  \mathrm{d} t\right)^4\\
&\quad+  \mathbb{E}\left(\int_{t_1}^{t_2}\|\Lambda^sL_\eta \textbf{y}(t)\|_{\textbf{L}^{2}}^{2(p-1)} \|\Lambda^sL_\eta G  (\textbf{y} )\|_{L_2(U;\textbf{L}^{2})}^2 \mathrm{d} t\right)^4\\
 &\quad+  \mathbb{E}\left(\sum_{j\geq 1}\int_{t_1}^{t_2} \|\Lambda^sL_\eta \textbf{y}(t)\|_{\textbf{L}^{2}}^{2(p-1)} (\Lambda^sL_\eta \textbf{y}(t),\Lambda^sL_\eta G (\textbf{y} )e_j)_{\textbf{L}^{2}} \mathrm{d}  \mathcal {W}^j\right)^4\\
 &\quad+  \mathbb{E}\left(\int_{t_1}^{t_2} \|\Lambda^sL_\eta \textbf{y}(t)\|_{\textbf{L}^{2}}^{2(p-2)}\sum_{j\geq 1}|(\Lambda^sL_\eta \textbf{y}(t),\Lambda^sL_\eta G (\textbf{y} )e_j)_{\textbf{L}^{2}}|^2 \mathrm{d} t\right)^4,
 \end{split}
\end{equation}
where $F  =(F_1  ,F _2 ,F _3 )$,
$F_1 (\textbf{y} ) = \Delta  n -\theta_R(\|(u,n)\|_{W^{1,\infty}})u\cdot \nabla  n- \theta_R(\|(n,c)\|_{W^{1,\infty}}) \textrm{div}  (n[\chi(c)\nabla c]*\rho^{\epsilon})$, $F_2 (\textbf{y} )= \Delta  c-\theta_R(\|(u,c)\|_{W^{1,\infty}})  u\cdot \nabla  c  -\theta_R(\|(n,c)\|_{W^{1,\infty}})[n \kappa(c)]*\rho^{\epsilon}$ and
$F_3 (\textbf{y} )= \Delta  u-\theta_R(\|u\|_{W^{1,\infty}}) \textbf{P} (u\cdot \nabla)  u + \textbf{P} (n\nabla \phi)*\rho^{\epsilon} $.

Now for any $K>0$, we define a sequence of stopping times by
$$
\textbf{t}^K:=\inf\{t\geq  0;~  \|\textbf{y}(t)\|_{\textbf{H}^{s} }\geq K\}.
$$
In view of Lemma \ref{lem6}, we see that $\textbf{t}^K\rightarrow\infty$ as $K\rightarrow\infty$. For any $T>0$, we replace $t_i$ by $\textbf{t}^K \wedge t_i$, $i=1,2$ in \eqref{2.41}. Since $\|\Lambda^sL_\eta \textbf{y}\|_{\textbf{L}^{2}}\approx \| L_\eta \textbf{y}\|_{\textbf{H}^{s}}\lesssim \| \textbf{y}\|_{\textbf{H}^{s}}$, one can derive from \eqref{2.41} and the definition of $\textbf{t}^K$ that
\begin{equation}\label{2.42}
\begin{split}
 &\mathbb{E}\left[\Big(\| L_\eta \textbf{y}(\textbf{t}^K \wedge t_2)\|_{\textbf{H}^s } ^{2p} -\| L_\eta \textbf{y}(\textbf{t}^K \wedge t_1)\|_{\textbf{H}^s } ^{2p}\Big)^4\right]\\
  &\quad \lesssim_{p,K} |\textbf{t}^K \wedge t_2-\textbf{t}^K \wedge t_1|^4+ \mathbb{E}\left(\int_{\textbf{t}^K \wedge t_1}^{\textbf{t}^K \wedge t_2}  \|  L_\eta F   (\textbf{y}(t) )\|_{\textbf{H}^{s}}  \mathrm{d} t\right)^4\\
 &\quad+  \mathbb{E}\left(\sum_{j\geq 1}\int_{\textbf{t}^K \wedge t_1}^{\textbf{t}^K \wedge t_2} \|L_\eta \textbf{y}(t)\|_{\textbf{H}^{s}}^{2(p-1)} (L_\eta \textbf{y}(t),L_\eta G (\textbf{y} )e_j)_{\textbf{H}^{s}} \mathrm{d}  \mathcal {W}^j\right)^4.
 \end{split}
\end{equation}
The term $\|  L_\eta F   (\textbf{y}  )\|_{\textbf{H}^{s}}$ can be estimated similar to the \textsf{Step 1} in Lemma \ref{lem2}, and we have
\begin{equation*}
\begin{split}
\|  L_\eta F   (\textbf{y}(t)  )\|_{\textbf{H}^{s}} \lesssim _{p,K,\eta}\| \textbf{y}(t)\|_{\textbf{H}^{s}}\lesssim _{p,K,\eta}1,\quad \textrm{for all} ~t\in [0,\textbf{t}^K].
 \end{split}
\end{equation*}
For the stochastic term in \eqref{2.42}, we get by using the BDG inequality that
\begin{equation*}
\begin{split}
\textrm{Last term in}~\eqref{2.42}&\lesssim \mathbb{E}\left(\int_{\textbf{t}^K \wedge t_1}^{\textbf{t}^K \wedge t_2}\|L_\eta \textbf{y}(t)\|_{\textbf{H}^{s}}^{4(p-1)} \sum_{j\geq 1}|(L_\eta \textbf{y}(t),L_\eta G (\textbf{y} )e_j)_{\textbf{H}^{s}}|^2  \mathrm{d} t\right)^2\\
&\lesssim \mathbb{E}\left(\int_{\textbf{t}^K \wedge t_1}^{\textbf{t}^K \wedge t_2} \sum_{j\geq 1}\|L_\eta G (\textbf{y} )e_j\|_{\textbf{H}^{s}}^2  \mathrm{d} t\right)^2\\
&\lesssim \mathbb{E}\left(\int_{\textbf{t}^K \wedge t_1}^{\textbf{t}^K \wedge t_2} (1+ \| \textbf{y}(t)\|_{\textbf{H}^{s}}^2)  \mathrm{d} t\right)^2\lesssim _{K} |\textbf{t}^K \wedge t_2-\textbf{t}^K \wedge t_1|^2.
 \end{split}
\end{equation*}
Putting the last two estimates into \eqref{2.42}, we get
\begin{equation*}
\begin{split}
 \mathbb{E} \Big|\| L_\eta \textbf{y}(\textbf{t}^K \wedge t_2)\|_{\textbf{H}^s } ^{2p} -\| L_\eta \textbf{y}(\textbf{t}^K \wedge t_1)\|_{\textbf{H}^s } ^{2p}\Big|^4 \lesssim_{p,K,\eta} | t_2- t_1|^2.
 \end{split}
\end{equation*}
By taking the limit as $\eta\rightarrow 0$ and then $K\rightarrow\infty$ in the last estimate lead to
\begin{equation*}
\begin{split}
 \mathbb{E} \Big|\| \textbf{y}( t_2)\|_{\textbf{H}^s } ^{2p} -\|  \textbf{y}(  t_1)\|_{\textbf{H}^s } ^{2p}\Big|^4  \lesssim_{p,K,\eta} | t_2- t_1|^2,
 \end{split}
\end{equation*}
which combined with Kolmogorov's Continuity Theorem imply that the real-valued process $\| \textbf{y}(  t)\|_{\textbf{H}^s } ^{2p}$ is continuous in time $\mathbb{P}$-almost surely.

\textsf{Step 2 (Uniqueness).} Assume that $(\textbf{y}^1,\textbf{t}^1)$ and  $(\textbf{y}^2,\textbf{t}^2)$ are two local solutions to \eqref{Mod-1} with cut-off operators with respect to the same initial data $\textbf{y}_0$.  Let $L> 2M = 2\| \textbf{y}_0\|_{\textbf{H}^s}$. For any $T>0$, define $\textbf{t}^L_T:= \textbf{t}^L \wedge T$, where
\begin{equation}\label{2.43}
\begin{split}
\textbf{t}^L:=\inf\{t> 0;~\| \textbf{y}^1(t)\|_{\textbf{H}^s}+\| \textbf{y}^2(t)\|_{\textbf{H}^s}\geq L\}.
 \end{split}
\end{equation}
Then $\mathbb{P}\{\textbf{t}^L>0\}=1$ for $L$ large enough. Recall that $\mathbb{E}(\sup_{t\in [0,\textbf{t}^i]}\| \textbf{y}^i(t)\|_{\textbf{H}^s}^2) <\infty$, $i=1,2$, we have
\begin{equation}\label{2.44}
\begin{split}
\mathbb{P}\left\{\liminf_{L\rightarrow\infty}\textbf{t}^L > \textbf{t}^1\wedge\textbf{t}^2\right\}=1.
 \end{split}
\end{equation}
Define $\textbf{p}=\textbf{y}^1-\textbf{y}^2$, then there holds
\begin{equation}\label{2.45}
\begin{split}
\mathrm{d}  \textbf{p}= (F(\textbf{y}^1)-F(\textbf{y}^1))  \mathrm{d} t+ (G(\textbf{y}^1)-G(\textbf{y}^1))\mathrm{d} \mathcal {W},
 \end{split}
\end{equation}
where the functionals $F(\cdot)$ and $G(\cdot)$ are defined in \eqref{2.41}. Note that all of the terms in \eqref{2.45} have spacial regularity in $\textbf{H}^{s-2}(\mathbb{R}^2)$. Applying It\^{o}'s formula to $   \|\textbf{p}(t)\|_{\textbf{H}^{s-2}}^2$, we deduce that
\begin{equation}\label{2.46}
\begin{split}
 &\|\textbf{p}(t)\|_{\textbf{H}^{s-2}}^2 + \int_0^t \|\nabla\textbf{p}(r)\|_{\textbf{H}^{s-2}}^2 \mathrm{d}r \\
  &\quad\leq 2\int_0^t(\textbf{p}(t), F(\textbf{y}^1)-\Delta \textbf{y}^1 -F(\textbf{y}^2)+ \Delta\textbf{y}^2  )_{\textbf{H}^{s-2}}\mathrm{d}r+ \int_0^t\|G(\textbf{y}^1)-G(\textbf{y}^2)\|_{L_2(U;\textbf{H}^{s-2})}^2 \mathrm{d}r \\
 &\quad\quad+2\sum_{j\geq 1}\int_0^t(\textbf{p}(t),(G(\textbf{y}^1)-G(\textbf{y}^2))e_j )_{\textbf{H}^{s-2}} \mathrm{d} \mathcal {W}^j := \mathscr{T}_1 (t) +\mathscr{T}_2 (t)+\mathscr{T}_3 (t).
 \end{split}
\end{equation}
For $\mathscr{T}_1(t)$, we will only deal with the third component, say $\mathscr{T}_1^{(3)}(t)$,  in $\mathscr{T}_1(t)$ with respect to the fluid equation, since the terms in the other equations of \eqref{Mod-1} can be treated in a similar manner (in view of the Lemma \ref{lem2}). Note that
\begin{equation}\label{2.47}
\begin{split}
\mathscr{T}_1^{(3)}(t)&=\left(u^1-u^2, \theta_R(\|u^1\|_{W^{1,\infty}})\textbf{P} (u^1\cdot \nabla)  u^1-\theta_R(\|u^2\|_{W^{1,\infty}})\textbf{P} (u^2\cdot \nabla)  u^2 \right)_{H^{s-2}}\\
&\quad+\left(u^1-u^2, \textbf{P} ((n^1-n^2)\nabla \phi)*\rho^{\epsilon}\right) _{H^{s-2}} \\
& := T_1(t)+T_2(t).
 \end{split}
\end{equation}
For $T_1(t)$,  we first get by the divergence-free condition $ \textrm{div}  u^2 =0$ that
\begin{equation*}
\begin{split}
| (\Lambda ^{s-2}(u^1-u^2),   \Lambda ^{s-2}[u^2\cdot \nabla    (u^1 -u^2)] )_{L^2}|&=| (\Lambda ^{s-2}\nabla(u^1-u^2),   \Lambda ^{s-2} [u^2\otimes (u^1 -u^2)]  )_{L^2}|\\
&\leq \frac{1}{2}  \|\nabla (u^1-u^2)\|_{H^{s-2}}^2+ C\| u^2\|_{H^{s }}^2\|u^1-u^2\|_{H^{s-2}}^2.
\end{split}
\end{equation*}
Then, it follows from the facts that $H^{s-2}(\mathbb{R}^2)$ is a Banach algebra and $H^{s-2}(\mathbb{R}^2)\subset W^{1,\infty}(\mathbb{R}^2)$ that
\begin{equation}\label{2.48}
\begin{split}
| T_1 |&\leq \theta_R(\|u^1\|_{W^{1,\infty}})|(u^1-u^2,\textbf{P} (u^1\cdot \nabla)  u^1- \textbf{P} (u^2\cdot \nabla)  u^2)_{H^{s-2}}|\\
&\quad+|\theta_R(\|u^1\|_{W^{1,\infty}}) -\theta_R(\|u^2\|_{W^{1,\infty}})|(\Lambda ^{s-2}(u^1-u^2),  \Lambda ^{s-2}[(u^2\cdot \nabla)  u^2])_{H^{s-2}}\\
&\lesssim  |(\Lambda ^{s-2}(u^1-u^2),  \Lambda ^{s-2}[(u^1-u^2)\cdot \nabla  u^1])_{L^2}|\\
&\quad+|(\Lambda ^{s-2}(u^1-u^2),   \Lambda ^{s-2}[u^2\cdot \nabla    (u^1 -u^2)])_{L^2}|\\
&\quad+ \|u^1-u^2\|_{W^{1,\infty}}\|u^1-u^2\|_{H^{s-2}}\| (u^2\cdot \nabla)  u^2\|_{H^{s-2}}\\
&\leq \frac{1}{2}  \|\nabla (u^1-u^2)\|_{H^{s-2}}^2+ \|u^1-u^2\|_{H^{s-2}}^2(\| u^1\|_{H^{s }}^2+\| u^2\|_{H^{s }}^2).
\end{split}
\end{equation}
For $T_2(t)$, we get from the convolution inequality $\|f*g\|_{L^2}\lesssim \|f \|_{L^2}\| g\|_{L^1}$ that
\begin{equation}\label{2.49}
\begin{split}
| T_2| &\leq \|u^1-u^2\| _{H^{s-2}} \|\textbf{P} ((n^1-n^2)\nabla \phi)*\Lambda ^{s-2}\rho^{\epsilon}\| _{L^2 } \\
&\lesssim_{\phi,\epsilon} (\|u^1-u^2\| _{H^{s-2}} ^2+\| n^1-n^2 \|_{H^{s-2} } ^2).
\end{split}
\end{equation}
Therefore by estimates \eqref{2.47}-\eqref{2.49}, we gain
\begin{equation}\label{2.50}
\begin{split}
 \mathscr{T}_1^{(3)} (t)\leq \frac{1}{2}  \|\nabla (u^1-u^2)\|_{H^{s-2}}^2 +C_{\phi,\epsilon} (\| u^1\|_{H^{s }}^2+\| u^2\|_{H^{s }}^2) \|\textbf{p}\| _{\textbf{H}^{s-2}}^2.
\end{split}
\end{equation}
Similarly, by using the Paralinearization Theorem and the Mean Value Theorem as well as condition $ \textrm{div}  u^1= \textrm{div}  u^2=0$, one can estimate the first two components $\mathscr{T}_1^{(1)}$ and $\mathscr{T}_1^{(2)}$ in $\mathscr{T}_1 $ to obtain
\begin{equation*}
\begin{split}
\mathscr{T}_1^{(1)} (t)&\leq \frac{1}{2}  \|\nabla (n^1-n^2)\|_{H^{s-2} }^2 +C_{\chi} \left(\| n^1\|_{H^{s }} (\| c^1\|_{H^{s } }+1)+\| n^2\|_{H^{s }}+\| u^2\|_{H^{s }}^2\right)\|\textbf{p}\| _{\textbf{H}^{s-2} }^2,
\end{split}
\end{equation*}
and
\begin{equation*}
\begin{split}
\mathscr{T}_1^{(2)} (t)&\leq \frac{1}{2}  \|\nabla (c^1-c^2)\|_{H^{s-2} }^2\\
&\quad+C_{R,\kappa} \left(\| c^1\|_{H^{s }} (\| u^1\|_{H^{s } }+\| n^1\|_{H^{s } }+1)+\| n^2\|_{H^{s }}+\| u^2\|_{H^{s }}^2\right)\|\textbf{p}\| _{\textbf{H}^{s-2}}^2,
\end{split}
\end{equation*}
which together with \eqref{2.50} imply that
\begin{equation}\label{2.51}
\begin{split}
 \mathscr{T}_1 (t)\leq \frac{1}{2} \int_0^t  \|\nabla \textbf{p}(r)\|_{\textbf{H}^{s-2}}^2 \mathrm{d}r +C_{R,\kappa,\chi,\phi,\epsilon}\int_0^t \sum_{i=1}(\| n^i\|_{H^{s } }^2+\| c^i\|_{H^{s } }^2 +\| u^i\|_{H^{s }}^2)\|\textbf{p}(r)\| _{\textbf{H}^{s-2}}^2\mathrm{d}r.
\end{split}
\end{equation}
For $\mathscr{T}_2 (t)$, we have
\begin{equation}\label{2.52}
\begin{split}
 \mathscr{T}_2 (t)\lesssim_{L,R}  \int_0^t \|\textbf{p}(r)\| _{H^{s-2}}^2\mathrm{d}r.
\end{split}
\end{equation}
By applying the BDG inequality, we get
\begin{equation} \label{2.53}
\begin{split}
\mathbb{E}\sup_{r\in[0,\textbf{t}^L_T\wedge t]}|\mathscr{T}_3 (r)|&\lesssim \mathbb{E}\left(\int_0^{\textbf{t}^L_T\wedge t}\|\textbf{p}(r)\|_{\textbf{H}^{s-2}}^2\sum_{j\geq 1}\|(G(\textbf{y}^1)-G(\textbf{y}^2))e_j \|_{\textbf{H}^{s-2}}^2 \mathrm{d}r\right)^{1/2}\\
&\lesssim \mathbb{E}\left[\sup_{r\in[0,\textbf{t}^L_T\wedge t]}\|\textbf{p}(r)\|_{\textbf{H}^{s-2}} \bigg(\int_0^{\textbf{t}^L_T\wedge t} \|\textbf{p}(r)\|_{\textbf{H}^{s-2}} ^2\mathrm{d}r\bigg)^{1/2}\right]\\
&\leq \frac{1}{2} \mathbb{E} \sup_{r\in[0,\textbf{t}^L_T\wedge t]}\|\textbf{p}(r)\|_{\textbf{H}^{s-2}}^2+ C_R \mathbb{E}\int_0^{\textbf{t}^L_T\wedge t} \|\textbf{p}(r)\|_{\textbf{H}^{s-2}} ^2\mathrm{d}r .
\end{split}
\end{equation}
Thereby, by taking the supremum on both sides of \eqref{2.46} over $[0,\textbf{t}^L_T\wedge t]$, it follows from the definition of $\textbf{t}^L_T$ and the estimates \eqref{2.51}-\eqref{2.53} that
\begin{equation*}
\begin{split}
\mathbb{E}\sup_{r\in[0,\textbf{t}^L_T\wedge t]}\|\textbf{p}(r)\|_{\textbf{H}^{s-2}}^2  \lesssim_{L,R,\chi,\kappa,\phi,\epsilon} \int_0^{ t} \mathbb{E}\sup_{r\in[0,\textbf{t}^L_T\wedge \tau]} \|\textbf{p}(r)\|_{\textbf{H}^{s-2}} ^2\mathrm{d}  \tau,\quad \textrm{for all}~ t>0.
\end{split}
\end{equation*}
By applying Gronwall Lemma to the last inequality leads to $\mathbb{E}\sup_{r\in[0,\textbf{t}^L_T ]}\|\textbf{p}(r)\|_{\textbf{H}^{s-2}}^2=0$, which implies
$$
\mathbb{E}\sup_{r\in[0,\textbf{t}^L\wedge \textbf{t}^1\wedge\textbf{t}^2 ]}\|\textbf{p}(r)\|_{\textbf{H}^{s-2}}^2=0.
$$
Taking the limit as $L\rightarrow\infty$ and using \eqref{2.44}, we get from the Monotone Convergence Theorem that
$$
\mathbb{P}\{\textbf{y}^1(t)=\textbf{y}^1(t),~~ \forall t \in [0,\textbf{t}^1\wedge\textbf{t}^2]\}=1.
$$
The proof of Lemma \ref{lem7} is completed.
\end{proof}

\begin{lemma}\label{lem9}
Any maximal local solution $(n_{\epsilon},c_{\epsilon},u_{\epsilon}, \textbf{t}^\epsilon)$ to \eqref{Mod-1} satisfies that, for all $T>0$,
\begin{equation}\label{2.54}
\begin{split}
n^{\epsilon}(x,t)\geq 0,~~ c^{\epsilon}(x,t)\geq 0,~~\forall t\in[0,T\wedge\textbf{t}^\epsilon),~~x\in \mathbb{R}^2,~~\mathbb{P}\textrm{-a.s.,}
\end{split}
\end{equation}
\begin{equation}\label{2.55}
\begin{split}
\| n^\epsilon(\cdot,t)\|_{L^1}\equiv\| n^\epsilon_0\|_{L^1},~~ \|c^\epsilon(\cdot,t)\|_{L^1\cap L^\infty}\leq\|c_0\|_{L^1\cap L^\infty},~~\forall t\in[0,T\wedge\textbf{t}^\epsilon),~~\mathbb{P}\textrm{-a.s.}
\end{split}
\end{equation}
\end{lemma}

\begin{proof}[\emph{\textbf{Proof}}]
The proof is similar to \cite{zhai20202d}, and we omit the details here.
\end{proof}

\begin{lemma}\label{lem8}
Let $s>5$, and assume that the hypothesises {(\textsf{A$_1$})}-{(\textsf{A$_3$})} hold. Then for any $0<\epsilon<1$ and $T>0$, the system \eqref{Mod-1} has a global unique pathwise solution
\begin{equation}\label{eee}
\begin{split}
\textbf{y}^{\epsilon} \in L^2(\Omega;\mathcal {C}([0,T];\textbf{H}^{s}(\mathbb{R}^2)))\bigcap L^2(\Omega;L^2([0,T];\textbf{H}^{s+1}(\mathbb{R}^2))).
\end{split}
\end{equation}
\end{lemma}

\begin{proof}[\emph{\textbf{Proof}}]
The proof will be divided into two steps. In the first step we show the existence and uniqueness of local maximal pathwise solutions. In the second step we prove that the lifespan of the maximal solution can be extended to infinity.

\textsf{Step 1.} For any give $R>1$, let $\textbf{y}^{R,\epsilon}(t)$ be the global pathwise solution to \eqref{Mod-1} with cut-off operators associated to the initial data $\textbf{y}^{R,\epsilon}(0)= (n_0^\epsilon,c_0^\epsilon,u_0^\epsilon) $ (cf. Lemma \ref{lem7}). To remove the cutoffs, we define
\begin{equation}
\begin{split}
\textbf{t}^\epsilon :=\inf\{t>0;~ \sup_{r\in [0,t]}\|\textbf{y}^{R,\epsilon}(r)\|_{\textbf{H}^{s}}> \|\textbf{y}^{R,\epsilon}(0)\|_{\textbf{H}^{s}}+1\}.
\end{split}
\end{equation}
For any $R>0$, we have $\mathbb{P}\{\textbf{t}^\epsilon > 0\}=1$. Let $C_\epsilon>0$ be a constant such that $\|\textbf{y}^{R,\epsilon}(0)\|_{\textbf{H}^{s}}\leq C_\epsilon$, and   $C_{\textrm{emb}}>0$ be the Sobolev embedding constant such that $\|\cdot\|_{W^{1,\infty}}\leq C_{\textrm{emb}} \|\cdot\|_{\textbf{H}^{s}} $. Then we have
\begin{equation}
\begin{split}
 \|\textbf{y}^{R,\epsilon}(t)\|_{W^{1,\infty}} \leq C_{\textrm{emb}} \|\textbf{y}^{R,\epsilon}(t)\|_{\textbf{H}^{s}} \leq C_{\textrm{emb}} (\|\textbf{y}^{R,\epsilon}(0)\|_{\textbf{H}^{s}} +1) \leq C_{\textrm{emb}} (C_\epsilon +1),
\end{split}
\end{equation}
for any $t \in [0,\textbf{t}^\epsilon ]$.   If $R > C_{\textrm{emb}} (C_\epsilon +1)$, then we have
$$
\theta_R(\|(u,n)\|_{W^{1,\infty}})=\theta_R(\|(n,c)\|_{W^{1,\infty}})=\theta_R(\|(u,c)\|_{W^{1,\infty}})
=\theta_R(\|u\|_{W^{1,\infty}})\equiv 1 , ~~ \mathbb{P} \textrm{-a.s.}$$ Therefore, $(\textbf{y}^{R,\epsilon}, \textbf{t}^\epsilon )$ is a local pathwise solution to \eqref{Mod-1}. In this case,  the solution $\textbf{y}^{R,\epsilon}$ will be denoted by $\textbf{y}^{\epsilon}=(n^\epsilon,c^\epsilon,u^\epsilon)$.  By using a standard method (cf. \cites{bensoussan1995stochastic,breit2018stochastically}), one can extend the solution $(n^\epsilon,c^\epsilon,u^\epsilon)$ to a maximal time of existence $\widetilde{\textbf{t}^\epsilon}$. This proves the existence of local solution for \eqref{Mod-1}.

\textsf{Step 2.} We show that the local maximal solution constructed in the last step is actually a global-in-time one, that is $\mathbb{P}\{\widetilde{\textbf{t}^\epsilon}=\infty\}=1$. This will be done by deriving a suitable $\textbf{H}^{s}$-energy estimate for solutions to \eqref{Mod-1} uniformly in $\epsilon \in (0,1)$.

By virtue of $ \textrm{div}  u^\epsilon =0$ and the nonnegativity of $c^\epsilon$ and $n^\epsilon$, similar to the proof in Lemma \ref{lem2} (see also \cites{duan2010global,chae2014global,zhai20202d}), one can deduce that for any $T>0$
$$
\mathbb{E}\sup_{t\in [0,T]}\|(n^\epsilon,c^\epsilon,u^\epsilon) \|_{\textbf{L}^2}^2 +\mathbb{E} \int_0^T  \|(n^\epsilon,c^\epsilon,u^\epsilon) \|_{\textbf{H}^1}^2  \mathrm{d} t \lesssim _{n_0,c_0,u_0,\phi} 1 .
$$
Moreover, we have
\begin{equation}\label{2.57}
\begin{split}
&\mathbb{E} \sup_{r\in [0,T\wedge \tau_K^\epsilon]}\|(n^\epsilon,c^\epsilon,u^\epsilon)(r) \|_{\textbf{H}^{1}}^2 +\mathbb{E} \int_0^{T\wedge \tau_K^\epsilon} \|( n^\epsilon, c^\epsilon,u^\epsilon)  (r)\|_{\textbf{H}^{2} }^2 \mathrm{d}  r \lesssim _{n_0,c_0,u_0,K,\phi,T} 1.
\end{split}
\end{equation}
where
$
\tau_K^\epsilon:= \inf\{t>0;~ \sup_{r\in [0,t]}\|( n^\epsilon, c^\epsilon, u^\epsilon)(r)\|_{\textbf{L}^2}^2>K~ \textrm{or}~  \int_0^t \|( n^\epsilon, c^\epsilon, u^\epsilon)\|_{\textbf{H}^1}^2 \mathrm{d}r>K\}.
$ Clearly, we have $\tau_K^\epsilon\rightarrow\infty$ as $K\rightarrow\infty$ almost surely.

To obtain the $\textbf{H}^{s}(s>1)$-estimate of $({n^\epsilon},{c^\epsilon},{u^\epsilon})$, let us first exploring the estimate for  $u^\epsilon$ in $W^{1,\infty}(\mathbb{R}^2)\supset \textbf{H}^{s}(\mathbb{R}^2)$ with $s>2$. Note that this can be achieved by consider the estimate for  $\|\varpi^\epsilon\|_{H^2}$ in view of the Biot-Savart law, where $\varpi^\epsilon = \nabla\wedge u^\epsilon= \partial _{x_2}u_1^\epsilon-\partial _{x_2}u^\epsilon_2$ denotes the vorticity of $u^\epsilon$.

\textsf{Estimate for $\|u^\epsilon\|_{L^2(\Omega;L^2(0,T;W^{1,\infty})}$.} For any $R>0$, define
$$
 \textbf{t}_ R^\epsilon := \left\{t>0;  \int_0^t \|n^\epsilon(r)\|_{H^{2}}^2 \mathrm{d}r> R~~ \textrm{or} ~~ \int_0^t  \|\nabla u^\epsilon(r)\|_{L^2}^2 \mathrm{d}r> R \right\}.
$$
Then we have \eqref{2.57} that $T\wedge \textbf{t}_ R^\epsilon\wedge \tau_K^\epsilon\rightarrow T \wedge \tau_K^\epsilon$ as $R\rightarrow\infty$, $\mathbb{P}$-a.s.  Applying the operator $\nabla\wedge $ to \eqref{Mod-1}$_3$, we get
\begin{equation*}
\begin{split}
 \mathrm{d} \nabla \varpi^\epsilon+   \nabla\textbf{P} (u^\epsilon\cdot \nabla) \varpi^\epsilon  \mathrm{d} t =   A \nabla\varpi^\epsilon \mathrm{d} t +  \nabla\textbf{P}\nabla\wedge(n^\epsilon\nabla \phi) \mathrm{d} t+  \nabla\textbf{P}\nabla\wedge f(t,u^\epsilon) \mathrm{d} W.
\end{split}
\end{equation*}
Applying It\^{o}'s formula to $\mathrm{d}  \|\nabla \varpi^\epsilon\|_{L^2}^2$, we find
\begin{equation}\label{2.60}
\begin{split}
  & \|\nabla \varpi^\epsilon(t)\|_{L^2}^2 + 2 \int_0^t\|A^{\frac{1}{2}}\nabla \varpi^\epsilon\|_{L^2}^2\mathrm{d}r  = \|\nabla \varpi^\epsilon_0\|_{L^2}^2\\
  &\quad-2 \int_0^t(\nabla \varpi^\epsilon,  \nabla\textbf{P} (u^\epsilon\cdot \nabla) \varpi^\epsilon )_{L^2}\mathrm{d}r+2 \int_0^t(\nabla \varpi^\epsilon,   \textbf{P}\nabla\{\nabla\wedge[(n^\epsilon\nabla \phi)*\rho^\epsilon]\} )_{L^2}\mathrm{d}r\\
  &\quad+  \int_0^t\|\nabla\textbf{P}\nabla\wedge f(t,u^\epsilon)\|_{L_2(U;L^2)}^2 \mathrm{d}r+2 \sum_{j\geq 1}\int_0^t(\nabla \varpi^\epsilon,  \nabla\textbf{P}\nabla\wedge f_j(t,u^\epsilon))_{L^2}\mathrm{d} W^j\\
  &\quad:= \|\nabla \varpi^\epsilon_0\|_{L^2}^2+\mathcal {B}_1(t)+\cdots+\mathcal {B}_4(t).
\end{split}
\end{equation}
For $\mathcal {B}_1$, by using the Sobolev embedding $H^1(\mathbb{R}^2)\subset L^p(\mathbb{R}^2)$ for all $1<p< \infty$  (cf. \cite{ladyzhenskaya1969mathematical}) and the fact of $(u^\epsilon\cdot \nabla) \varpi^\epsilon= \textrm{div}  (u^\epsilon \varpi^\epsilon)$, we have
\begin{equation}\label{2.61}
\begin{split}
  \mathcal {B}_1(t)=2 \int_0^t(\Delta \varpi^\epsilon,    u^\epsilon \varpi^\epsilon )_{L^2}\mathrm{d}r
  &\leq \int_0^t\|\Delta \varpi^\epsilon\|_{L^2}^2 \mathrm{d}r+ C \int_0^t\|\nabla u^\epsilon\|_{L^2}^2\| \nabla\varpi^\epsilon \|_{L^2}^2\mathrm{d}r.
\end{split}
\end{equation}
For $\mathcal {B}_2$, first note that
$$
|\nabla\{\nabla\wedge[(n^\epsilon\nabla \phi)*\rho^\epsilon]\}|= |\nabla [(\nabla n^\epsilon\wedge\nabla \phi)*\rho^\epsilon]|\leq\sum_k\sum_{i\neq j}(|\partial_k\partial_i n^\epsilon \partial_j \phi|  + | \partial_i n^\epsilon \partial_k\partial_j \phi|)*\rho^\epsilon.
$$
Then it follows from  the Young inequality and the fact of $\|\rho^\epsilon\|_{L^1}=1$ that
\begin{equation}\label{2.62}
\begin{split}
  \mathcal {B}_2(t)
  &\leq \int_0^t\|\nabla \varpi^\epsilon\|_{L^2}^2\mathrm{d}r+  C \int_0^t (\|\Delta n^\epsilon\|_{L^2}^2 +\|\nabla n^\epsilon\|_{L^2}^2) \|\phi\|_{W^{2,\infty}}^2 \mathrm{d}r\\
  &\leq  \int_0^t\|\nabla \varpi^\epsilon\|_{L^2}^2\mathrm{d}r+  C _ \phi\int_0^t \| n^\epsilon(r)\|_{H^2}^2 \mathrm{d}r.
\end{split}
\end{equation}
For $\mathcal {B}_3$,  we get from the  relation between $\varpi^\epsilon$ and $u^\epsilon$ (cf. \cite{majda2002vorticity}) that
\begin{equation}\label{2.63}
\begin{split}
  \mathcal {B}_3(t)&\lesssim   \int_0^t\| f(t,u^\epsilon)\|_{L_2(U;H^2)}^2 \mathrm{d}r \lesssim   \int_0^t (1+\| u^\epsilon \|_{ H^2 }^2) \mathrm{d}r\lesssim_T   \int_0^t  (1+\| \nabla\varpi^\epsilon \|_{L^2 }^2 ) \mathrm{d}r  .
\end{split}
\end{equation}
By \eqref{2.60}-\eqref{2.63}, we get
\begin{equation}\label{2.64}
\begin{split}
 \|\nabla \varpi^\epsilon(t)\|_{L^2}^2   \leq \|\nabla \varpi^\epsilon_0\|_{L^2}^2+ C \int_0^t(1+\|\nabla u^\epsilon\|_{L^2}^2)(1+\| \nabla\varpi^\epsilon \|_{L^2 }^2 )\mathrm{d}r+   \int_0^t \| n^\epsilon(r)\|_{H^2}^2 \mathrm{d}r +\mathcal {B}_4(t).
\end{split}
\end{equation}
By applying the Gronwall Lemma, it follows from \eqref{2.64} that
\begin{equation*}
\begin{split}
 &\|\nabla \varpi^\epsilon(t\wedge  \textbf{t}_ R^\epsilon  \wedge \tau_K^\epsilon)\|_{L^2}^2 +\int_0^{t\wedge  \textbf{t}_ R^\epsilon \wedge \tau_K^\epsilon}\|A  \varpi^\epsilon\|_{L^2}^2\mathrm{d}r  \leq  e^{C\int_0^{t\wedge  \textbf{t}_ R^\epsilon \wedge \tau_K^\epsilon}  (1+\|\nabla u^\epsilon\|_{L^2}^2) \mathrm{d}r }\\
 &\quad\times \left(\| \varpi^\epsilon_0\|_{H^1}^2+ \int_0^{t\wedge  \textbf{t}_ R^\epsilon \wedge \tau_K^\epsilon} \| n^\epsilon(r)\|_{H^2}^2 \mathrm{d}r+\bigg|\sum_{j\geq 1}\int_0^{t\wedge  \textbf{t}_ R^\epsilon \wedge \tau_K^\epsilon}(\nabla \varpi^\epsilon,  \nabla\textbf{P}\nabla\wedge f_j(t,u^\epsilon))_{L^2}\mathrm{d} W^j\bigg|\right).
\end{split}
\end{equation*}
From the definition of $\textbf{t}_ R^\epsilon$ and the BDG inequality,  the last inequality implies that
\begin{equation}\label{2.65}
\begin{split}
 &\mathbb{E}\sup_{r\in[0,t]}\|\nabla \varpi^\epsilon(r\wedge  \textbf{t}_ R^\epsilon\wedge\tau_K^\epsilon )\|_{L^2}^2+\mathbb{E}\int_0^{t\wedge  \textbf{t}_ R^\epsilon\wedge\tau_K^\epsilon}\| \varpi^\epsilon(r)\|_{H^2}^2\mathrm{d}r\\
 &\quad\leq  e^{C_R}(\| \varpi^\epsilon_0\|_{H^1}^2+R) + e^{C_R} \mathbb{E}\sup_{r\in[0,t]}\left|\sum_{j\geq 1}\int_0^{r\wedge  \textbf{t}_ R^\epsilon \wedge\tau_K^\epsilon}(\nabla \varpi^\epsilon,  \nabla\textbf{P}\nabla\wedge f_j(r,u^\epsilon))_{L^2}\mathrm{d} W^j\right| \\
 &\quad\leq  e^{C_R}(\| \varpi^\epsilon_0\|_{H^1}^2+R) + C e^{C_R} \mathbb{E} \left(\int_0^{t\wedge  \textbf{t}_ R^\epsilon \wedge\tau_K^\epsilon}\|\nabla \varpi^\epsilon\|^2\|  f (r,u^\epsilon)\|_{L_2(U;H^2)}^2\mathrm{d}r \right)^{1/2}\\
 &\quad\leq  \frac{1}{2}\mathbb{E}\sup_{r\in[0,t]}\|\nabla \varpi^\epsilon(r\wedge  \textbf{t}_ R^\epsilon \wedge\tau_K^\epsilon)\|_{L^2}^2+ e^{C_R}(\| \varpi^\epsilon_0\|_{H^1}^2+R) + C_R \mathbb{E} \int_0^{t\wedge  \textbf{t}_ R^\epsilon\wedge\tau_K^\epsilon}(1+\|\nabla \varpi^\epsilon \|^2)\mathrm{d}r.
\end{split}
\end{equation}
Absorbing the first term on the R.H.S. of \eqref{2.65} and using the Gronwall Lemma, we get
\begin{equation*}
\begin{split}
\mathbb{E}\sup_{r\in[0,T\wedge  \textbf{t}_ R^\epsilon \wedge\tau_K^\epsilon]}\|\nabla \varpi^\epsilon(r)\|_{L^2}^2+\mathbb{E}\int_0^{T\wedge  \textbf{t}_ R^\epsilon \wedge\tau_K^\epsilon}\| \varpi^\epsilon(r)\|_{H^2}^2\mathrm{d}r \lesssim _{R,K,\phi, n_0,c_0,u_0,T}1,~~~\forall T> 0,
\end{split}
\end{equation*}
which implies that
\begin{equation}\label{2.66}
\begin{split}
 \mathbb{E}\int_0^{T\wedge  \textbf{t}_ R^\epsilon \wedge\tau_K^\epsilon}\| u^\epsilon(r)\|_{W^{1,\infty}}^2\mathrm{d}r \lesssim _{R,K,\phi, n_0,c_0,u_0,T}1, ~~~\forall T> 0.
\end{split}
\end{equation}
\textsf{Estimate for $\|\textbf{y}^\epsilon \|_{L^2(\Omega;C([0,T];\textbf{H}^s))}$ ($s> 2$).}  For any  $D>0$, define
$$
\textbf{t}_{D,R,K}^\natural= \bar{\textbf{t}}_D^\epsilon\wedge\textbf{t}_ R^\epsilon\wedge\tau_K^\epsilon,
$$
where $\tau_K^\epsilon$, $\textbf{t}_ R^\epsilon$ is defined as before, and
$$
 \bar{\textbf{t}}_D^\epsilon :=\left\{t>0; \int_0^{t} \|\nabla u^{\epsilon}\|_{L^\infty}\mathrm{d}  s > D ~~\textrm{or}~~\int_0^{t} \| n^{\epsilon}\|_{H^2}^2\mathrm{d}  s > D~~\textrm{or}~~\int_0^{t} \|  c^{\epsilon}\|_{H^2}^2\mathrm{d}  s > D \right\}.
$$
Then for any fixed $R>0$,  \eqref{2.57} and \eqref{2.66} imply that $T\wedge {\textbf{t}}_{D,R}^\natural\rightarrow T\wedge \textbf{t}_ R^\epsilon$ as $D\rightarrow\infty$.
%
First, by applying the operator  $\Lambda^s n^\epsilon \Lambda^s$ to the first equation in \eqref{Mod-1} and integrating by parts over $\mathbb{R}^2$, we get $\mathbb{P}$-a.s.
\begin{equation}\label{2.67}
\begin{split}
  \frac{ \mathrm{d} }{ \mathrm{d} t} \|\Lambda^s n^\epsilon \|_{L^2}^2 + 2\|\Lambda^s \nabla n^\epsilon\|_{L^2}^2 \lesssim |(\Lambda^s n^\epsilon,\Lambda^s(u^\epsilon\cdot \nabla n^\epsilon))_{L^2}| + |(\Lambda^s\nabla n^\epsilon, \Lambda^s\left(n^\epsilon[(\chi(c^\epsilon)\nabla c^\epsilon)*\rho^\epsilon]\right) )_{L^2}|.
\end{split}
\end{equation}
For the first term on the R.H.S. of \eqref{2.67}, we get by the Moser-type estimate that
\begin{equation*}
\begin{split}
2|(\Lambda^s n^\epsilon,\Lambda^s(u^\epsilon\cdot \nabla n^\epsilon))_{L^2}|&=2|(\Lambda^s n^\epsilon,\Lambda^s \textrm{div} (u^\epsilon n^\epsilon))_{L^2}|=2|(\Lambda^s \nabla n^\epsilon,\Lambda^s(u^\epsilon n^\epsilon))_{L^2}|\\
&\leq C\|\Lambda^s \nabla n^\epsilon\|_{L^2}(\| u^\epsilon \|_{H^s}\| n^\epsilon \|_{L^\infty}+\| u^\epsilon \|_{L^\infty}\| n^\epsilon \|_{H^s})\\
&\leq \frac{1}{2}\|\Lambda^s \nabla n^\epsilon\|_{L^2}^2+ C(\| n^\epsilon \|_{H^2}^2\| u^\epsilon \|_{H^s}^2+\| u^\epsilon \|_{L^\infty}^2\|n^\epsilon \|_{H^s}^2),
\end{split}
\end{equation*}
and by Lemma \ref{lem9} that
\begin{equation*}
\begin{split}
 &2|(\Lambda^s\nabla n^\epsilon, \Lambda^s\left(n^\epsilon[(\chi(c^\epsilon)\nabla c^\epsilon)*\rho^\epsilon]\right) )_{L^2}|\\
 &\quad \leq\frac{1}{2}\|\Lambda^s \nabla n^\epsilon\|_{L^2}^2+ C(\|n^\epsilon\|_{L^\infty}^2      \| \int_0^{c^\epsilon} \chi(r)\mathrm{d}r *\nabla \rho^\epsilon \|_{H^s}^2+\|n^\epsilon\|_{H^s}^2\|\int_0^{c^\epsilon} \chi(r)\mathrm{d}r *\nabla \rho^\epsilon\|_{L^\infty}^2)\\
 &\quad\leq\frac{1}{2}\|\Lambda^s \nabla n^\epsilon\|_{L^2}^2+ C_\epsilon  \|n^\epsilon\|_{H^2}^2\| c^\epsilon \|_{H^s}^2  + C_{\epsilon,c_0,\chi} \|n^\epsilon\|_{H^s}^2  .
\end{split}
\end{equation*}
Inserting the last two estimates into \eqref{2.67} leads to
\begin{equation} \label{2.68}
\begin{split}
&    \|\Lambda^s n^\epsilon (t) \|_{L^2}^2 +  \int_0^t\|\Lambda^s \nabla n^\epsilon\|_{L^2}^2 \mathrm{d}r \\
 &\quad\lesssim _{\epsilon,c_0,\chi} \|\Lambda^s n^\epsilon_0 \|_{L^2}^2+\int_0^t (\|n^\epsilon\|_{H^2}^2+\| u^\epsilon \|_{L^\infty}^2+ 1 )(\|n^\epsilon\|_{H^s}^2  +\| c^\epsilon \|_{H^s}^2+ \| u^\epsilon \|_{H^s}^2 ) \mathrm{d}r.
\end{split}
\end{equation}
In a similar manner,
\begin{equation}
\begin{split}
 &\|\Lambda^s c^\epsilon (t) \|_{L^2}^2 +  \int_0^t\|\Lambda^s \nabla c^\epsilon\|_{L^2}^2 \mathrm{d}r
  \lesssim _{\epsilon,c_0,\kappa} \|\Lambda^s c^\epsilon_0 \|_{L^2}^2\\
 &\quad+\int_0^t (\|n^\epsilon\|_{H^2}^2+\|c^\epsilon\|_{H^2}^2+\| u^\epsilon \|_{L^\infty}^2+ 1 )(\|n^\epsilon\|_{H^s}^2  +\| c^\epsilon \|_{H^s}^2+ \| u^\epsilon \|_{H^s}^2 ) \mathrm{d}r.
\end{split}
\end{equation}
For the stochastic equation in \eqref{Mod-1}, we apply It\^{o}'s formula to $ \|\Lambda^s u^\epsilon\|_{L^2}^2$ to obtain
\begin{equation}\label{2.69}
\begin{split}
 & \|\Lambda^s u^\epsilon(t)\|_{L^2}^2+\int_0^t\|A^{\frac{1}{2}}\Lambda^s  u^\epsilon(r)\|_{L^2}^2 \mathrm{d}r =\|\Lambda^s u^\epsilon_0\|_{L^2}^2 -2\int_0^t(\Lambda^s u^\epsilon ,\Lambda^s\textbf{P} (u^\epsilon\cdot \nabla) u^\epsilon)_{L^2} \mathrm{d}r  \\
  & \quad +2 \int_0^t(\Lambda^s u^\epsilon ,\Lambda^s\textbf{P}[ (n^\epsilon\nabla \phi)*\rho^\epsilon])_{L^2} \mathrm{d}r  +2 \sum_{j\geq 1}\int_0^t(\Lambda^s u^\epsilon , \Lambda^s\textbf{P} f_j(t,u^\epsilon) )_{L^2}\mathrm{d} W ^j.
\end{split}
\end{equation}
For the terms on the R.H.S. of \eqref{2.69}, one can estimate as
\begin{equation*}
\begin{split}
  |(\Lambda^s u^\epsilon ,\Lambda^s\textbf{P} (u^\epsilon\cdot \nabla) u^\epsilon)_{L^2}|&= |(\Lambda^s u^\epsilon ,\textbf{P}[\Lambda^s,  u^\epsilon]\cdot \nabla u^\epsilon)_{L^2}|+\underbrace{|(\Lambda^s u^\epsilon ,\textbf{P} u^\epsilon \cdot \nabla \Lambda^s u^\epsilon)_{L^2}|}_{= 0 }\\
 & \lesssim\|\Lambda^s u^\epsilon \|_{L^2}( \|\Lambda^s u^\epsilon \|_{L^2} \| u^\epsilon \|_{L^\infty}+\|\Lambda^{s-1} \nabla u^\epsilon \|_{L^2} \| \nabla u^\epsilon \|_{L^\infty})\\
 & \lesssim (1+\| \nabla u^\epsilon \|_{L^\infty}^2)\| u^\epsilon \|_{H^s}^2,
\end{split}
\end{equation*}
 and\begin{equation*}
\begin{split}
  |(\Lambda^s u^\epsilon ,\Lambda^s\textbf{P}[ (n^\epsilon\nabla \phi)*\rho^\epsilon])_{L^2}| \lesssim \|\Lambda^s u^\epsilon \|_{L^2}\| n^\epsilon\nabla \phi\|_{L^2}\|\Lambda^s\rho^\epsilon \|_{L^1}\lesssim_{\epsilon,\phi} \| u^\epsilon \|_{H^s}\| n^\epsilon \|_{L^2}.
\end{split}
\end{equation*}
It then follows from \eqref{2.69} and the last two estimates that
\begin{equation*}
\begin{split}
 \| u^\epsilon(t)\|_{H^s}^2 +\int_0^t\| u^\epsilon(r)\|_{H^{s+1}}^2 \mathrm{d}r&\lesssim_{\epsilon,\phi} \|\Lambda^s u^\epsilon_0\|_{L^2}^2+1 + \int_0^t(1+\| \nabla u^\epsilon \|_{L^\infty}^2)\| u^\epsilon \|_{H^s}^2\mathrm{d}r  \\
  & +  \int_0^t\| u^\epsilon \|_{H^s}^2\| n^\epsilon \|_{L^2}^2 \mathrm{d}r  +2 \sum_{j\geq 1}\int_0^t(\Lambda^s u^\epsilon , \Lambda^s\textbf{P} f_j(t,u^\epsilon) )_{L^2}\mathrm{d} W ^j,
\end{split}
\end{equation*}
which combined with \eqref{2.68} and \eqref{2.69} lead to
\begin{equation*}
\begin{split}
&    \| (n^\epsilon,c^\epsilon,u^\epsilon) (t\wedge \textbf{t}_{D,R,K}^\natural) \|_{\textbf{H}^s}^2+\int_0^{t\wedge \textbf{t}_{D,R,K}^\natural}\| (n^\epsilon,c^\epsilon,u^\epsilon) (r) \|_{\textbf{H}^{s+1}}^2\mathrm{d}r\\
  &\quad\lesssim _{\epsilon,c_0,\chi,\phi,\kappa} \| (n^\epsilon_0,c^\epsilon_0,u^\epsilon_0) \|_{\textbf{H}^s}^2+1+\int_0^{t\wedge \textbf{t}_{D,R,K}^\natural} \left(\| (n^\epsilon,c^\epsilon,u^\epsilon) \|_{\textbf{H}^2}^2+\| u^\epsilon \|_{W^{1,\infty}}^2+ 1 \right)\\
&\quad\quad\times \| (n^\epsilon,c^\epsilon,u^\epsilon) (r) \|_{\textbf{H}^s}^2 \mathrm{d}r +2 \sum_{j\geq 1}\int_0^t(\Lambda^s u^\epsilon , \textbf{P}\Lambda^s f_j(t,u^\epsilon) )_{L^2}\mathrm{d} W ^j.
\end{split}
\end{equation*}
An application of Gronwall Lemma to the last inequality implies that
\begin{equation*}
\begin{split}
& \| (n^\epsilon,c^\epsilon,u^\epsilon) (t\wedge \textbf{t}_{D,R,K}^\natural) \|_{\textbf{H}^s}^2+\int_0^{t\wedge \textbf{t}_{D,R,K}^\natural}\| (n^\epsilon,c^\epsilon,u^\epsilon) (r) \|_{\textbf{H}^{s+1}}^2\mathrm{d}r\\
  &\quad\lesssim _{\epsilon,c_0,\chi,\phi,\kappa,T}\exp\left\{\int_0^{t\wedge \textbf{t}_{D,R,K}^\natural} (\| (n^\epsilon,c^\epsilon,u^\epsilon) \|_{\textbf{H}^2}^2+\| \nabla u^\epsilon \|_{L^\infty}^2+ 1 ) \mathrm{d}r\right\}\\
  &\quad\quad\times \left (\| (n^\epsilon_0,c^\epsilon_0,u^\epsilon_0) \|_{\textbf{H}^s}^2+1+2 \sum_{j\geq 1}\int_0^{t\wedge \textbf{t}_{D,R,K}^\natural}(\Lambda^s u^\epsilon , \Lambda^s\textbf{P} f_j(t,u^\epsilon) )_{L^2}\mathrm{d} W ^j\right).
\end{split}
\end{equation*}
In view of the definition of  $\textbf{t}_{D,R,K}^\natural$, we have
$$
\exp\left\{\int_0^{t\wedge \textbf{t}_{D,R,K}^\natural} (\| (n^\epsilon,c^\epsilon,u^\epsilon) \|_{\textbf{H}^2}^2+\| u^\epsilon \|_{W^{1,\infty}}^2+ 1 ) \mathrm{d}r\right\}\leq e^{C_{R}+T}.
$$
Therefore,  by taking the supremum over $[0,T]$ in last inequality,  we deduce that
\begin{equation*}
\begin{split}
& \mathbb{E}\sup_{t\in [0,T]}\| (n^\epsilon,c^\epsilon,u^\epsilon) (t\wedge \textbf{t}_{D,R,K}^\natural) \|_{\textbf{H}^s}^2\\
  &\quad\lesssim _{R,\textbf{y}_0^\epsilon,\epsilon,\chi,\phi,\kappa,T}  1+\mathbb{E} \left(\sup_{t\in[0,T\wedge \textbf{t}_{D,R,K}^\natural]}\|\Lambda^s u^\epsilon\|_{L^2}^2\sum_{j\geq 1}\int_0^{T\wedge \textbf{t}_{D,R,K}^\natural}\| \Lambda^s f_j(r,u^\epsilon) \|_{L^2}^2\mathrm{d}r\right)^{1/2}\\
  &\quad\leq\frac{1}{2} \mathbb{E}\sup_{t\in [0,T]}\| (n^\epsilon,c^\epsilon,u^\epsilon) (r\wedge \textbf{t}_{D,R,K}^\natural) \|_{\textbf{H}^s}^2\\
  &\quad\quad+C_{R,\textbf{y}_0^\epsilon,\epsilon,\chi,\phi,\kappa,T} \left(\| (n^\epsilon_0,c^\epsilon_0,u^\epsilon_0) \|_{\textbf{H}^s}^2+1+\mathbb{E}  \int_0^{T\wedge \textbf{t}_{D,R,K}^\natural}(1+\| u^\epsilon\|_{H^s}^2 )\mathrm{d}r \right).
\end{split}
\end{equation*}
By using the Gronwall Lemma, we arrive at
\begin{equation}\label{2.72}
\begin{split}
 \mathbb{E}\sup_{t\in [0,T\wedge \textbf{t}_{D,R,K}^\natural]}\| (n^\epsilon,c^\epsilon,u^\epsilon) (t ) \|_{\textbf{H}^s}^2+\mathbb{E}\int_0^{t\wedge \textbf{t}_{D,R,K}^\natural}\| (n^\epsilon,c^\epsilon,u^\epsilon) (r) \|_{\textbf{H}^{s+1}}^2\mathrm{d}r \lesssim_{R,K,\textbf{y}_0^\epsilon,\epsilon, \chi,\phi,\kappa,T} 1.
\end{split}
\end{equation}
Define
$$
n^{\epsilon}_{R,D}(t,x)=n^{\epsilon}(t\wedge \textbf{t}_{D,R,K}^\natural), ~~ c^{\epsilon}_{R,D}(t,x)=c^{\epsilon}(t\wedge\textbf{t}_{D,R,K}^\natural),~~ u^{\epsilon}_{R,D}(t,x)=u^{\epsilon}(t\wedge\textbf{t}_{D,R,K}^\natural),~~\forall t> 0.
$$
Apparently, the triple $(n^{\epsilon}_{R,D},c^{\epsilon}_{R,D},u^{\epsilon}_{R,D})$ satisfies \eqref{Mod-1} for any given $R,D>0$, and there holds
$$
\mathbb{E}\sup_{t\in [0,T]}\| (n^{\epsilon}_{R,D},c^{\epsilon}_{R,D},u^{\epsilon}_{R,D}) (t ) \|_{\textbf{H}^s}^2 \lesssim_{R,K,\textbf{y}_0^\epsilon,\epsilon, \chi,\phi,\kappa,T} 1.
$$
It then follows that
\begin{equation} \label{2.73}
\begin{split}
\widetilde{\textbf{t}^\epsilon}\geq T\wedge \textbf{t}_{D,R,K}^\natural =T \wedge \bar{\textbf{t}}_D^\epsilon\wedge\textbf{t}_ R^\epsilon \wedge\tau_K^\epsilon,\quad \textrm{for all}~~ R,D,K>0,~~\mathbb{P}\textrm{-a.s.},
\end{split}
\end{equation}
for any $T>0$, where $\widetilde{\textbf{t}^\epsilon}$ is the maximal existence time constructed in \textsf{Step 1}.   Sending $D\rightarrow +\infty$ and then $R\rightarrow +\infty$ in \eqref{2.73}, we obtain $\mathbb{P}\{\widetilde{\textbf{t}^\epsilon}=\infty\}=1$, which implies that the solution $(n^\epsilon,c^\epsilon,u^\epsilon)$ exists globally. The regularity \eqref{eee} is a direct consequence of \eqref{2.72}.  The proof of Lemma \ref{lem8} is now completed.
\end{proof}

\section{Proof of Theorem \ref{th1}}\label{sec3}

In this section, we shall first derive a uniform a priori estimate for the smooth approximate solutions $\{\textbf{y}^{\epsilon}\}_{\epsilon >0} $, and then we prove that the KS-SNS system \eqref{KS-SNS} admits a unique weak solution by taking the limit $\epsilon\rightarrow0$ via stochastic compactness method. Let us start with the following uniform a priori estimate.

\subsection{A new stochastic entropy estimate}

\begin{lemma} \label{lem10}
Let $\epsilon \in (0,1)$, $T>0$,  and assume that  $(n^\epsilon,c^\epsilon,u^\epsilon)$ is the unique global pathwise solution in $L^2(\Omega;\mathcal {C}([0,T];\textbf{H}^{s}(\mathbb{R}^2)))\bigcap L^2(\Omega;L^2([0,T];\textbf{H}^{s+1}(\mathbb{R}^2)))$ constructed in Lemma \ref{lem8}. Define
$
h(t):=\int_1^t\sqrt{\chi/\kappa}(s)\mathrm{d}  s$ and $g(t):=  \sqrt{\kappa/\chi}'(t).
$
Then we have for all $p\geq1$
\begin{equation}\label{unf}
\begin{split}
&\mathbb{E} \sup_{t\in [0,T]} \left(\|n^\epsilon(t)\|_{L^1 \cap~  {L \emph{\textrm{log}} L}}^2+ \|\nabla h (c^{\epsilon}(t)) \|_{L^2} ^{2 }+ \|u^\epsilon(t) \|_{L^2}^{2 }\right) ^p\\
&\quad+ \mathbb{E}\left(\int_0^T\left(\|\nabla \sqrt{n^{\epsilon}+1}\|^2_{L^2} +   \|\Delta h (c^{\epsilon})\|_{L^2}^2  +  \| \sqrt{g (c^\epsilon)}\nabla h (c^{\epsilon})\|_{L^4}^4  +  \| \nabla u^\epsilon\|_{L^2}^2\right)  \mathrm{d} t\right)^p \\
&\quad\lesssim _{n_0,c_0,u_0,\phi,\kappa,\chi ,p,T}  1.
\end{split}
\end{equation}
Here, $L \emph{\textrm{log}} L(\mathbb{R}^2)$ denotes the  Zygmund space which is equipped with the norm
$$
\|f\|_{L \log L}=\inf  \{k>0; \int_{\mathbb{R}^2} Z(f/ k) \mathrm{d}x \leq 1 \}
$$
with respect to $Z(t) =t\log^+ t$ if $t\geq1$ and $Z(t) =0$ otherwise.
\end{lemma}

\begin{proof}[\emph{\textbf{Proof}}]
Applying the chain rule to $(n^\epsilon+1) \ln (n^\epsilon+1 )$ associated to the $n^\epsilon$-equation in \eqref{Mod-1} and integrating by parts with the help of divergence-free condition $ \textrm{div}  u^\epsilon =0$, we find
\begin{equation}\label{3.1}
\begin{split}
&\mathrm{d}  \int_{\mathbb{R}^2} (n^\epsilon+1) \ln (n^\epsilon+1 ) \mathrm{d}x+ 4 \|\nabla \sqrt{n^{\epsilon}+1}\|^2_{L^2} \mathrm{d} t
=\int_{\mathbb{R}^2} \nabla n^{\epsilon} \cdot[ \sqrt{(\kappa\chi)(c^{\epsilon}) } \nabla h(c^{\epsilon})]*\rho^{\epsilon}  \mathrm{d}x \mathrm{d} t\\
&\quad+\int _{\mathbb{R}^2} \ln (n^\epsilon+1 )  \textrm{div}  [(\sqrt{(\kappa\chi)(c^{\epsilon}) }\nabla h(c^{\epsilon}))*\rho^\epsilon]\mathrm{d}x \mathrm{d} t.
\end{split}
\end{equation}
In terms of the $c^\epsilon$-equation in \eqref{Mod-1}, we infer that
\begin{equation}\label{3.2}
\begin{split}
\mathrm{d}  \emph{h} (c^{\epsilon}) = \left(-u^{\epsilon}\cdot \nabla \emph{h}(c^{\epsilon})  + \Delta \emph{h}(c^{\epsilon})-\emph{h} ''(c^{\epsilon}) |\nabla c^{\epsilon}|^2 - \sqrt{(\kappa\chi)(c^{\epsilon})}(n^{\epsilon}*\rho^{\epsilon})\right) \mathrm{d} t,
\end{split}
\end{equation}
where we used the identity $$\emph{h}'(c^{\epsilon})\Delta c^{\epsilon}= \Delta \emph{h}(c^{\epsilon}) - \emph{h}''(c^{\epsilon}) |\nabla c^{\epsilon}|^2.$$ To deal with the terms on the R.H.S. of \eqref{3.1}, we multiply both sides of \eqref{3.2} by $-\Delta \emph{h} (c^{\epsilon}) $ and integrating by parts over $\mathbb{R}^2$, we get from the fact of $h''(t)/(h'(t))^2=-g(t)$ that
\begin{equation}\label{3.3}
\begin{split}
&\frac{1}{2}\mathrm{d}  \|\nabla\emph{h} (c^{\epsilon}) \|_{L^2} ^2+ \|\Delta \emph{h} (c^{\epsilon})\|_{L^2}^2 \mathrm{d} t =\underbrace{-\int_{\mathbb{R}^2} g(c^\epsilon)  |\nabla \emph{h} (c^{\epsilon})|^2  \Delta \emph{h} (c^{\epsilon})\mathrm{d}x}_{:=(\star)} \mathrm{d} t\\
 &\quad+\int_{\mathbb{R}^2} \left(u^{\epsilon}\cdot \nabla \emph{h}(c^{\epsilon})\right)\Delta \emph{h} (c^{\epsilon}) \mathrm{d}x \mathrm{d} t + \int_{\mathbb{R}^2}  \sqrt{(\kappa\chi)(c^{\epsilon})}(n^{\epsilon}*\rho^{\epsilon})\Delta \emph{h} (c^{\epsilon})\mathrm{d}x \mathrm{d} t.
\end{split}
\end{equation}
For $(\star)$, we see that
\begin{equation}\label{3.4}
\begin{split}
 (\star)&=  \sum_{i,j } \int_{\mathbb{R}^2} \left(g'(c^\epsilon) \partial_j c^\epsilon (\partial_{i} \emph{h} (c^{\epsilon}))^2+2g(c^\epsilon)\partial_{i} \emph{h} (c^{\epsilon})\partial_{i} \partial_j\emph{h} (c^{\epsilon}) \right) \partial_j \emph{h} (c^{\epsilon})\mathrm{d}x\\
&=  \sum_{i,j } \int_{\mathbb{R}^2}  \frac{g'(c^\epsilon)}{ h'(c^\epsilon)} (\partial_{i} \emph{h} (c^{\epsilon}))^2(\partial_j  \emph{h} (c^{\epsilon}))^2\mathrm{d}x     + 2\sum_{i}   \int_{\mathbb{R}^2} g(c^\epsilon)(\partial_{i} \emph{h} (c^{\epsilon}))^2\partial_{i}^2 \emph{h} (c^{\epsilon})\mathrm{d}x \\
& \quad+2\sum_{i\neq j }  \int_{\mathbb{R}^2}g(c^\epsilon) \partial_{i} \emph{h} (c^{\epsilon})\partial_j  \emph{h} (c^{\epsilon})\partial_{i }\partial_{j } \emph{h} (c^{\epsilon})\mathrm{d}x
.
\end{split}
\end{equation}
For the first term on the R.H.S. of \eqref{3.4}, we observe that
\begin{equation*}
\begin{split}
 2\sum_{i}   \int_{\mathbb{R}^2} g(c^\epsilon)(\partial_{i} \emph{h} (c^{\epsilon}))^2\partial_{i}^2 \emph{h} (c^{\epsilon})\mathrm{d}x =  -2(\star)-2\sum_{i\neq j}  \int_{\mathbb{R}^2}g(c^\epsilon)(\partial_{i} \emph{h} (c^{\epsilon}))^2\partial_{j}^2 \emph{h} (c^{\epsilon})\mathrm{d}x,
\end{split}
\end{equation*}
which implies that
\begin{equation}
\begin{split}
(\star)&=  \int_{\mathbb{R}^2}  \frac{1}{3}\frac{g'(c^\epsilon)}{g^2(c^\epsilon) h'(c^\epsilon)} g^2(c^\epsilon)|\nabla\emph{h} (c^{\epsilon})|^4\mathrm{d}x   -\frac{2}{3}\sum_{i\neq j}  \int_{\mathbb{R}^2}g(c^\epsilon)(\partial_{i} \emph{h} (c^{\epsilon}))^2\partial_{j}^2 \emph{h} (c^{\epsilon})\mathrm{d}x \\
&\quad +\frac{2}{3}\sum_{i\neq j }  \int_{\mathbb{R}^2}g(c^\epsilon) \partial_{i} \emph{h} (c^{\epsilon})\partial_j  \emph{h} (c^{\epsilon})\partial_{i }\partial_{j } \emph{h} (c^{\epsilon})\mathrm{d}x\\
&\leq  \int_{\mathbb{R}^2}  \frac{1}{3}\frac{g'(c^\epsilon)}{g^2(c^\epsilon) h'(c^\epsilon)} g^2(c^\epsilon)|\nabla\emph{h} (c^{\epsilon})|^4\mathrm{d}x +\frac{1}{6}\sum_{i }  \int_{\mathbb{R}^2}g^2(c^\epsilon)|\partial_{i} \emph{h} (c^{\epsilon})|^4 \mathrm{d}x\\
  &  \quad+\frac{2}{3}\sum_{i }  \int_{\mathbb{R}^2}|\partial_{i}^2 \emph{h} (c^{\epsilon})|^2\mathrm{d}x +\frac{1}{6}\sum_{i\neq j }  \int_{\mathbb{R}^2}g^2(c^\epsilon) |\partial_{i} \emph{h} (c^{\epsilon})|^2|\partial_j  \emph{h} (c^{\epsilon})|^2 \mathrm{d}x  +\frac{2}{3}\sum_{i\neq j }  \int_{\mathbb{R}^2}|\partial_{i }\partial_{j } \emph{h} (c^{\epsilon})|^2\mathrm{d}x\\
&= \int_{\mathbb{R}^2} \bigg(\frac{1}{6}+ \frac{1}{3}\frac{g'(c^\epsilon)}{g^2(c^\epsilon) h'(c^\epsilon)}\bigg) g^2(c^\epsilon)|\nabla\emph{h} (c^{\epsilon})|^4\mathrm{d}x  +\frac{2}{3} \int_{\mathbb{R}^2}|\Delta\emph{h} (c^{\epsilon})|^2\mathrm{d}x,
\end{split}
\end{equation}
where the last inequality used the facts of
$$
\sum_{i, j } (\partial_{i} \emph{h} (c^{\epsilon}))^2(\partial_j  \emph{h} (c^{\epsilon}))^2= |\nabla \emph{h} (c^{\epsilon})|^4 ~~\textrm{and}~~ \|\Delta f\|_{L^2}^2=\|\nabla^2 f\|_{L^2}^2.
$$
In view of the definition of $g(\cdot)$ and $h(\cdot)$ as well as the assumption {(\textsf{A$_2$})}, we have
$$
h'=\sqrt{\chi/\kappa},~~ g'= \frac{1}{4}\sqrt{\kappa/\chi}\left[\left(\kappa/\chi\right)'\right]^2+\frac{1}{2} \sqrt{\chi/\kappa}\left(\kappa/\chi\right)'',
$$
and
$$
\frac{g'}{g^2}=  \frac{\left(\kappa/\chi\right)^{3/2} (\chi/\kappa)'}{(\kappa/\chi)'}+\frac{2\sqrt{\kappa/\chi}(\kappa/\chi)''}{[(\kappa/\chi)']^2}\leq  \frac{\left(\kappa/\chi\right)^{3/2}(\chi/\kappa)'}{(\kappa/\chi)'}=-\sqrt{\chi/\kappa}=-h',
$$
which imply that $\frac{g'(c^\epsilon)}{g^2(c^\epsilon) h'(c^\epsilon)}\leq - 1$ and hence $\frac{1}{6}+ \frac{1}{3}\frac{g'(c^\epsilon)}{g^2(c^\epsilon) h'(c^\epsilon)}< -\frac{1}{6}$. It then follows from \eqref{3.3} that
\begin{equation} \label{3.6}
\begin{split}
&\frac{1}{2}\mathrm{d}  \|\nabla\emph{h} (c^{\epsilon}) \|_{L^2} ^2+ \frac{1}{3} \|\Delta \emph{h} (c^{\epsilon})\|_{L^2}^2 \mathrm{d} t +\frac{1}{6}\int_{\mathbb{R}^2}  g^2(c^\epsilon)|\nabla\emph{h} (c^{\epsilon})|^4\mathrm{d}x  \mathrm{d} t \\
 &\quad= \int_{\mathbb{R}^2}  \sqrt{(\kappa\chi)(c^{\epsilon})}(n^{\epsilon}*\rho^{\epsilon})\Delta \emph{h} (c^{\epsilon})\mathrm{d}x \mathrm{d} t+\int_{\mathbb{R}^2} \left(u^{\epsilon}\cdot \nabla \emph{h}(c^{\epsilon})\right)\Delta \emph{h} (c^{\epsilon}) \mathrm{d}x \mathrm{d} t\\
 &\quad:=(\star\star) \mathrm{d} t+(\star\star\star) \mathrm{d} t.
\end{split}
\end{equation}
For $(\star\star)$, we have
\begin{equation*}
\begin{split}
(\star\star)&=-\int_{\mathbb{R}^2} (n^{\epsilon}*\rho^{\epsilon}) \nabla \sqrt{(\kappa\chi)(c^{\epsilon})}\cdot\nabla \emph{h} (c^{\epsilon})\mathrm{d}x-\int_{\mathbb{R}^2}  \sqrt{(\kappa\chi)(c^{\epsilon})}(\nabla n^{\epsilon}*\rho^{\epsilon})\cdot\nabla \emph{h} (c^{\epsilon})\mathrm{d}x\\
&=\underbrace{-\int_{\mathbb{R}^2} \frac{\sqrt{\kappa\chi}'(c^{\epsilon})}{\emph{h}' (c^{\epsilon})}(n^{\epsilon}*\rho^{\epsilon}) |\nabla \emph{h} (c^{\epsilon})|^2\mathrm{d}x}_{\leq0}-\int_{\mathbb{R}^2}  \nabla n^{\epsilon} \cdot[\sqrt{(\kappa\chi)(c^{\epsilon})}\nabla \emph{h} (c^{\epsilon})]*\rho^{\epsilon}\mathrm{d}x\\
&\leq -\int_{\mathbb{R}^2}  \nabla n^{\epsilon} \cdot[\sqrt{(\kappa\chi)(c^{\epsilon})}\nabla \emph{h} (c^{\epsilon})]*\rho^{\epsilon}\mathrm{d}x.
\end{split}
\end{equation*}
By using the Cauchy inequality, we have
\begin{equation*}
\begin{split}
(\star\star\star)
 \leq  \int_{\mathbb{R}^2}  |\partial_i \emph{h} (c^{\epsilon})\partial_ju^{\epsilon,i}  \partial_j \emph{h} (c^{\epsilon})| \mathrm{d}x \leq \frac{1}{12} \int_{\mathbb{R}^2} g^2(c^\epsilon) |\nabla  \emph{h} (c^{\epsilon}) |^4 \mathrm{d}x +C_{c_0} \int_{\mathbb{R}^2} |\nabla u^{\epsilon}|^2  \mathrm{d}x,
\end{split}
\end{equation*}
where we used the fact of $0<1/g^2(c^\epsilon) \leq C$, due to the continuity of $g(\cdot)$ and the boundedness of $\|c^\epsilon\|_{L^\infty}$. Moreover, by assumption {(\textsf{A$_2$})}, there exists a constant $c_1>0$ such that $\frac{\sqrt{\kappa\chi}'(c^{\epsilon})}{\emph{h}' (c^{\epsilon})}\geq c_1$.

Putting the last two inequalities into \eqref{3.6}, we get
\begin{equation}\label{3.7}
\begin{split}
&\mathrm{d}  \left(\int_{\mathbb{R}^2} (n^\epsilon+1) \ln (n^\epsilon+1 ) \mathrm{d}x+\frac{1}{2} \|\nabla\emph{h} (c^{\epsilon}) \|_{L^2} ^2\right)+ 4 \|\nabla \sqrt{n^{\epsilon}+1}\|^2_{L^2}  \mathrm{d} t\\
& \quad+ \frac{1}{3} \|\Delta \emph{h} (c^{\epsilon})\|_{L^2}^2 \mathrm{d} t +\frac{1}{12}\int_{\mathbb{R}^2}  g^2(c^\epsilon)|\nabla\emph{h} (c^{\epsilon})|^4\mathrm{d}x  \mathrm{d} t \\
 &\quad\leq  C_{c_0} \int_{\mathbb{R}^2} |\nabla u^{\epsilon}|^2  \mathrm{d}x +\int _{\mathbb{R}^2} \ln (n^\epsilon+1 )  \textrm{div}  [(\sqrt{(\kappa\chi)(c^{\epsilon}) }\nabla h(c^{\epsilon}))*\rho^\epsilon]\mathrm{d}x \mathrm{d} t.
\end{split}
\end{equation}
For the last term on the R.H.S. of \eqref{3.7}, we get by Young inequality that
\begin{equation*}
\begin{split}
& \int _{\mathbb{R}^2} \ln (n^\epsilon+1 )  \textrm{div}  [(\sqrt{(\kappa\chi)(c^{\epsilon}) }\nabla h(c^{\epsilon}))*\rho^\epsilon]\mathrm{d}x
\leq \eta \|\sqrt{(\kappa\chi)(c^{\epsilon}) }\|_{L^\infty} \| \Delta h(c^{\epsilon}) \|^2_{L^2}\\
&\quad +C_\eta\int _{\mathbb{R}^2}  (n^\epsilon+1 )\ln (n^\epsilon+1 )  \mathrm{d}x +\eta \|\frac{\sqrt{\kappa\chi}'(c^{\epsilon})}{g (c^{\epsilon})h'(c^{\epsilon})}\|_{L^\infty}^2\int _{\mathbb{R}^2} g^2(c^{\epsilon})|\nabla h(c^{\epsilon})|^4\mathrm{d}x.
\end{split}
\end{equation*}
In view of the assumption on $\kappa,\chi$ and Lemma \ref{lem9}, one can choose $\eta>0$ small enough such that
$$
\eta \|\sqrt{(\kappa\chi)(c^{\epsilon}) }\|_{L^\infty}< \frac{1}{4}~~\textrm{and}~~\eta \left\|\frac{\sqrt{\kappa\chi}'(c^{\epsilon})}{g (c^{\epsilon})h'(c^{\epsilon})}\right\|_{L^\infty}^2< \frac{1}{24}.
$$
Then estimate \eqref{3.7} implies that
\begin{equation} \label{rr}
\begin{split}
&\mathrm{d}  \left(\int_{\mathbb{R}^2} (n^\epsilon+1) \ln (n^\epsilon+1 ) \mathrm{d}x+\frac{1}{2} \|\nabla\emph{h} (c^{\epsilon}) \|_{L^2} ^2\right)+ 4 \|\nabla \sqrt{n^{\epsilon}+1}\|^2_{L^2}  \mathrm{d} t\\
&\quad + \frac{1}{12} \|\Delta \emph{h} (c^{\epsilon})\|_{L^2}^2 \mathrm{d} t +\frac{1}{24}\int_{\mathbb{R}^2}  g^2(c^\epsilon)|\nabla\emph{h} (c^{\epsilon})|^4\mathrm{d}x  \mathrm{d} t  \\
 &\quad\leq  C_{c_0} \int_{\mathbb{R}^2} |\nabla u^{\epsilon}|^2  \mathrm{d}x  +C_{\kappa,\chi,c_0}\int _{\mathbb{R}^2}  (n^\epsilon+1 )\ln (n^\epsilon+1 )  \mathrm{d}x .
\end{split}
\end{equation}
In order to estimate the term $\|\nabla u^{\epsilon}\|_{L^2}^2$ on the R.H.S. of \eqref{rr}, we apply It\^{o}'s formula to $ \|u^\epsilon \|_{L^2}^2$, it then follows from the facts of $(u^\epsilon, \textbf{P} (u^\epsilon\cdot \nabla) u^\epsilon)_{L^2}=0$ and $\| A^{\frac{1}{2}} u^\epsilon\|_{L^2}=\| \nabla u^\epsilon\|_{L^2}$ that
\begin{equation}\label{3.9}
\begin{split}
&\|u^\epsilon(t) \|_{L^2}^2 + 2\int_0^t \| \nabla u^\epsilon\|_{L^2}^2\mathrm{d}r  \leq \|u^\epsilon _0 \|_{L^2}^2+C\int_0^t(1+\|u^\epsilon(r) \|_{L^2}^2) \mathrm{d}r\\
 &\quad- 2\int_0^t(u^\epsilon*\rho^\epsilon , \textbf{P} n^\epsilon\nabla \phi )_{L^2 }\mathrm{d}r +2\sum_{j\geq 1}\int_0^t(u^\epsilon, \textbf{P} f_j(t,u^\epsilon))_{L^2}\mathrm{d} W^j.
\end{split}
\end{equation}
Exploring the following decomposition
$$
\mathbb{R}^2= E_1\bigcup E_2 ,\quad
E_1=\{x \in \mathbb{R}^2;~ 0< n^\epsilon(x) < 1\},~  E_2=\{x \in \mathbb{R}^2;~   n^\epsilon (x)\geq 1\}.
$$
We get from the Young inequality and GN inequality that
\begin{equation*}
\begin{split}
|2(u^\epsilon*\rho^\epsilon ,n^\epsilon\nabla \phi )_{L^2_ \textrm{div} }|&\leq \left|\int_{E_1}(u^\epsilon*\rho^\epsilon) n^\epsilon\nabla \phi \mathrm{d}x \right|+\left|\int_{E_2}(u^\epsilon*\rho^\epsilon)  n^\epsilon\nabla \phi \mathrm{d}x\right| \\
&\leq C_\phi \left( \|u^\epsilon \|_{L^2}^2+ \|n^\epsilon \|_{L^1}  + \|u^\epsilon \|_{L^2}  \|\textbf{1}_{E_2} \sqrt{n^\epsilon+1}  \|_{L^2}\| \nabla[\textbf{1}_{E_2}\sqrt{n^\epsilon+1}] \|_{L^2}   \right )\\
&\leq 2\| \nabla \sqrt{n^\epsilon+1}  \|_{L^2}^2+C_{\phi,n_0 } \|u^\epsilon \|_{L^2}^2+ C_{\phi,n_0 } .
\end{split}
\end{equation*}
Plugging the last estimate into \eqref{3.9} and multiplying both sides of the inequality by $C_{c_0}$. Adding the resulting inequality with \eqref{rr}, we get
\begin{equation}\label{ww}
\begin{split}
\mathcal {F}(t)+ \int_0^t    \mathcal {G}(r) \mathrm{d}r&\leq C_{c_0,n_0,\phi,\kappa,\chi }\left(\mathcal {F}(0)+ 1 + \int_0^t\mathcal {F}(r) \mathrm{d}r\right) +2C_{c_0}\sum_{j\geq 1}\int_0^t(u^\epsilon, \textbf{P} f_j(t,u^\epsilon))_{L^2}\mathrm{d} W^j,
\end{split}
\end{equation}
where
\begin{equation*}
\begin{split}
\mathcal {F}(t)&:=\| (n^\epsilon+1) \ln (n^\epsilon+1 )\|_{L^1}+\frac{1}{2} \|\nabla\emph{h} (c^{\epsilon}) \|_{L^2} ^2+C_{c_0}\|u^\epsilon(t) \|_{L^2}^2,\\
\mathcal {G}(t)&:=2\|\nabla \sqrt{n^{\epsilon}+1}\|^2_{L^2} +  \frac{1}{12}\|\Delta \emph{h} (c^{\epsilon})\|_{L^2}^2  +  \frac{1}{24}\| \sqrt{g (c^\epsilon)}|\nabla\emph{h} (c^{\epsilon})| \|_{L^4}^4  + C_{c_0}\| \nabla u^\epsilon\|_{L^2}^2.
\end{split}
\end{equation*}
By applying the Gronwall lemma to \eqref{ww}, we get
\begin{equation*}
\begin{split}
\mathcal {F}(t) \leq e^{C_{c_0,n_0,\phi,\kappa,\chi }t}\left(\mathcal {F}(0)+ 1+\left|\sum_{j\geq 1}\int_0^t(u^\epsilon, \textbf{P} f_j(t,u^\epsilon))_{L^2}\mathrm{d} W^j\right|\right),\quad\forall t>0.
\end{split}
\end{equation*}
By raising the $p$-th power on both sides of the last inequality, we get from the BDG inequality that
\begin{equation}
\begin{split}\label{3.10}
\mathbb{E}\sup_{t\in [0,T]}|\mathcal {F}(t)|^p &\leq e^{C_{c_0,n_0,\phi,\kappa,\chi }T}\left(\mathcal {F}(0)^p+ 1+\mathbb{E}\sup_{t\in [0,T]}\left|\sum_{j\geq 1}\int_0^t(u^\epsilon, \textbf{P} f_j(t,u^\epsilon))_{L^2}\mathrm{d} W^j\right|^p\right)\nonumber\\
&\leq  \frac{ 1}{2} \mathbb{E}\sup_{t\in [0,T]}\|u^\epsilon(r)\|^{2p}+ C_{c_0,n_0,\phi,\kappa,\chi ,T}\left( \mathcal {F}(0)^p+ 1+\mathbb{E}\int_0^t (1+ |\mathcal {F}(r)|^p)\mathrm{d}r \right),
\end{split}
\end{equation}
which implies that
$
\mathbb{E}\sup_{t\in [0,T]}|\mathcal {F}(t)|^p \lesssim _{c_0,n_0,\phi,\kappa,\chi ,T} \mathcal {F}(0)^p+ 1.
$
By raising the $p$-th power on both sides of \eqref{ww}, we arrive at
 \begin{equation} \label{uu}
\begin{split}
\mathbb{E}\sup_{t\in [0,T]}|\mathcal {F}(t)|^p+ \mathbb{E}\left(\int_0^t    \mathcal {G}(r) \mathrm{d}r\right) ^p \lesssim _{c_0,n_0,\phi,\kappa,\chi ,p,T} \mathcal {F}(0)^p+ 1.
\end{split}
\end{equation}
Moreover, direct calculation shows that
$$
\|(n^\epsilon+1) \ln (n^\epsilon+1)\|_{L^1}\asymp\|n^\epsilon\|_{L^1 \cap L \textrm{log} L}=\|n^\epsilon\|_{L^1 }+\|n^\epsilon\|_{ L \textrm{log} L},
$$
which combined with \eqref{uu} implies the desired inequality. The proof of Lemma \ref{lem10} is completed.
\end{proof}

\subsection{Further energy estimates}

For any $R>0$, define a subset
$$
\Omega_R^\epsilon := \left\{\omega\in \Omega;~\sup_{t\in [0,T]}\|\nabla c^\epsilon\|_{L^2}^2\bigvee\int_0^T \|\Delta c^\epsilon   (t) \|_{L^2}^2  \mathrm{d} t \bigvee\int_0^T \|\sqrt{g(c^\epsilon)}\nabla c^\epsilon  (t) \|_{L^4}^4  \mathrm{d} t \leq R \right\}.
$$

\begin{lemma} \label{cor}
Assume that the conditions in Lemma \ref{lem10} hold, and $(n^\epsilon,c^\epsilon,u^\epsilon)$ are unique  solutions for \eqref{Mod-1} ensured by Lemma \ref{lem9}. Then for all $T>0$ and $p\geq1$, the estimates
\begin{align}
\mathbb{E} \sup_{t\in [0,T]}\|\nabla c^\epsilon\|_{L^2}^p + \mathbb{E}\left( \int_0^T\|\Delta c^\epsilon(t)\|_{L^2}^2  \mathrm{d} t\right)^{p/2} &\lesssim _{n_0,c_0,u_0,\phi,\kappa,\chi ,p,T}  1,\label{3.16}\\
\mathbb{E} \sup_{t\in [0,T]}\|u^\epsilon\|_{L^2}^p + \mathbb{E}\left( \int_0^T\|\nabla u^\epsilon(t)\|_{L^2}^2 \mathrm{d} t\right)^{p/2} &\lesssim _{n_0,c_0,u_0,\phi,\kappa,\chi ,p,T}  1 \label{3.15}
\end{align}
hold uniformly in $\epsilon \in (0,1)$. Moreover, we have for all $\omega\in \Omega_R^\epsilon$
\begin{equation} \label{3.14}
\begin{split}
\sup_{t\in [0,T]}\|n^\epsilon(t,\omega)\|_{L^1\cap ~  {L \emph{\textrm{log}} L} }^2 +\sup_{t\in [0,T]}\|n^\epsilon(t,\omega)\|_{L^2}^2+ \int_0^T \|\nabla n^\epsilon(t,\omega)\|^2_{L^2} \mathrm{d} t  \lesssim_{n_0,c_0,R,\chi,\kappa} 1 .
\end{split}
\end{equation}
\end{lemma}

\begin{proof}[\emph{\textbf{Proof}}]
The property \eqref{3.15} is a direct consequence of Lemma \ref{lem10}. For \eqref{3.16}, it follows from the assumption {(\textsf{A$_2$})} and Lemma \ref{lem9} that
\begin{equation}\label{3.18}
\begin{split}
\sup_{t\in [0,T]}\left\| \frac{1}{h'(c^\epsilon)}\right\|_{L^\infty}=\|\sqrt{\kappa/\chi}(c^\epsilon)\|_{L^\infty}\lesssim_{c_0,\chi,\kappa} 1,\quad \mathbb{P}\textrm{-a.s.}
\end{split}
\end{equation}
The boundedness of $\mathbb{E}\sup_{t\in [0,T]}\|c^{\epsilon}\|_{L^2}^p$ can be obtained from \eqref{2.55} by a interpolation argument. By the uniform bound \eqref{unf}, we have
\begin{equation}\label{3.19}
\begin{split}
 \mathbb{E}\sup_{t\in [0,T]}\| \nabla c^\epsilon\|_{L^2}^p &= \mathbb{E}\sup_{t\in [0,T]}(\|1/h'(c^\epsilon)\|_{L^\infty}^p\| \nabla h(c^\epsilon)\|_{L^2}^{p})\\
  &\lesssim_{c_0,\chi,\kappa,p} \mathbb{E} \sup_{t\in [0,T]}  \|\nabla h (c^{\epsilon}) \|_{L^2} ^{p } \lesssim _{n_0,c_0,u_0,\phi,\kappa,\chi ,T,p}  1.
\end{split}
\end{equation}
Direct calculation shows that
$$
\Delta c^\epsilon =\frac{\Delta h(c^\epsilon)}{h'(c^\epsilon)}- \frac{h''(c^\epsilon)g(c^\epsilon)|\nabla h(c^\epsilon)|^2}{g(c^\epsilon)(h'(c^\epsilon))^3}\quad \textrm{and}\quad\frac{h''(c^\epsilon) }{g(c^\epsilon)(h'(c^\epsilon))^3}=-\frac{1 }{h'(c^\epsilon)} .
$$
It follows from \eqref{3.18} and the uniform bound \eqref{unf} that
\begin{equation*}
\begin{split}
& \mathbb{E}\left(\int_0^T\|\Delta c^\epsilon\|_{L^2} ^2 \mathrm{d} t\right)^{p/2}\\
 &\quad\lesssim_{p,T} \mathbb{E}\left[\left(\int_0^T\|\frac{\Delta h(c^\epsilon)}{h'(c^\epsilon)}\|_{L^2}^2 \mathrm{d}t\right)^{p/2}+\left(\int_0^T\|\frac{h''(c^\epsilon)g(c^\epsilon)|\nabla h(c^\epsilon)|^2}{g(c^\epsilon)(h'(c^\epsilon))^3}\|_{L^2}^2 \mathrm{d} t\right)^{p/2}\right]\\
 &\quad\lesssim_{c_0,\chi,\kappa,p,T}  \mathbb{E}\left[\bigg(\int_0^T\|\Delta h(c^\epsilon)\|_{L^2}^2 \mathrm{d} t\bigg)^{p}+\bigg(\int_0^T\|\sqrt{g(c^\epsilon)}\nabla h(c^\epsilon)\|_{L^4}^4 \mathrm{d} t\bigg)^{p}\right]\\
 &\quad\lesssim _{n_0,c_0,u_0,\phi,\kappa,\chi ,p,T}  1,
\end{split}
\end{equation*}
which together with \eqref{3.19} imply that $c^\epsilon \in L^p(\Omega;L^2(0,T;H^2(\mathbb{R}^2)))$ is bounded uniformly in $\epsilon$.

Applying the chain rule to $\mathrm{d}  \|n^\epsilon\|_{L^2}^2$, we find
\begin{equation}\label{3.20}
\begin{split}
 \|n^\epsilon(t)\|_{L^2}^2 +  \int_0^t\|\nabla n^\epsilon\|_{L^2}^2\mathrm{d}r
 &= \|n^\epsilon_0\|_{L^2}^2+\int_0^t((n^\epsilon)^2, \textrm{div} (\chi(c^\epsilon)\nabla c^\epsilon)*\rho^\epsilon)_{L^2}\mathrm{d}r.
\end{split}
\end{equation}
By Ladyzhenskaya's inequality, we gain
\begin{align}\label{3.21}
|((n^\epsilon)^2, \textrm{div} (\chi(c^\epsilon)\nabla c^\epsilon)*\rho^\epsilon)_{L^2}|
&\leq \|n^\epsilon\|_{L^2}\|\nabla n^\epsilon\|_{L^2} \| \textrm{div} (\chi(c^\epsilon)\nabla c^\epsilon\|_{L^2}\nonumber\\
&\leq \frac{1}{2}\|\nabla n^\epsilon\|_{L^2} ^2+C\|n^\epsilon\|_{L^2}^2(\|\chi'(c^\epsilon)|\nabla c^\epsilon|^2\|_{L^2}^2+ \|\chi(c^\epsilon)\Delta c^\epsilon\|_{L^2}^2)\nonumber\\
&\leq \frac{1}{2}\|\nabla n^\epsilon\|_{L^2} ^2+C_{c_0,\chi,\kappa}\|n^\epsilon\|_{L^2}^2(\|\sqrt{g(c^\epsilon)}\nabla h(c^\epsilon)\|_{L^4}^4+ \|\Delta c^\epsilon\|_{L^2}^2),
\end{align}
where the last inequality used the fact of
$$
\left\|\frac{\chi'(c^\epsilon)}{(h'(c^\epsilon))^2g(c^\epsilon)}\right\|_{L^\infty}=2 \|\chi'(c^\epsilon) (\kappa/\chi)^{3/2}\frac{1}{(\kappa/\chi)'}\|_{L^\infty}\lesssim_{c_0,\kappa,\chi}1.
$$
Plugging \eqref{3.21} into \eqref{3.20}, it then follows from the Gronwall Lemma that
\begin{equation*}
\begin{split}
 \sup_{t\in [0, T ]}\|n^\epsilon(t)\|_{L^2}^2 +   \int_0^  { T}\|\nabla n^\epsilon(t)\|_{L^2}^2 \mathrm{d} t &\lesssim_{n_0,c_0,\chi,\kappa} e^{ \int_0^ { T}(\|\sqrt{g(c^\epsilon)}\nabla h(c^\epsilon)\|_{L^4}^4 +  \|\Delta c^\epsilon\|_{L^2}^2)\mathrm{d}r},\quad \mathbb{P}\textrm{-a.s.},
\end{split}
\end{equation*}
which implies the desired inequality \eqref{3.14} by restricting the above estimate on the subset $\Omega_R^\epsilon \subset \Omega$. This completes the proof of Lemma \ref{cor}.
%
%
%
\end{proof}

\begin{remark}\label{rem3.3}
In view of Lemma \ref{lem9} and \eqref{3.15}, we infer that
\begin{equation*}
\begin{split}
\mathbb{P}\{\Omega_R^\epsilon\} &\geq 1 - \mathbb{P}\left\{\sup_{t\in [0,T]}\|\nabla c^\epsilon\|_{L^2}^2> R    \right\}- \mathbb{P}\left\{\int_0^T \|\Delta c^\epsilon   (t)\|_{L^2}^2  \mathrm{d} t> R\right\}\\
& - \mathbb{P}\left\{\int_0^T \|\sqrt{g(c^\epsilon)}\nabla c^\epsilon  (t) \|_{L^4}^4  \mathrm{d} t> R   \right \}\geq 1-\frac{3C_{n_0,c_0,u_0,\phi,\kappa,\chi , T}}{R}.
\end{split}
\end{equation*}
\end{remark}

In order to show the convergence of $(n^\epsilon,c^\epsilon,u^\epsilon)$ as $\epsilon\rightarrow 0$ rigorously, some boundedness information on the time regularity of  $(n^\epsilon,c^\epsilon,u^\epsilon)$ are needed.  To achieve this, we define another subset
$$
\widetilde{\Omega}_N^\epsilon := \left\{\omega\in \Omega;~\sup_{t\in[0,T]} \|u^\epsilon   (t) \|_{L^2}^2   \bigvee\int_0^T \| \nabla u^\epsilon  (t) \|_{L^2}^2  \mathrm{d} t \leq N \right\}.
$$

\begin{lemma} \label{lem11}
Let $(n^\epsilon,c^\epsilon,u^\epsilon)$ be the unique global pathwise solution constructed in Lemma \ref{lem9}, then
\begin{align}
 &\int_0^T\left\|  \frac{\mathrm{d}  n^\epsilon(t)}{ \mathrm{d} t} \right\|_{H^{-3}}^2  \mathrm{d} t\lesssim _{n_0,c_0,R,N,\chi,\kappa,T}1 , ~~~~\forall \omega\in  \Omega _R^\epsilon \bigcap \widetilde{\Omega}_N^\epsilon,\label{3.23}\\
  &\int_0^T\left\| \frac{\mathrm{d}  c^\epsilon(t)}{ \mathrm{d} t}\right\|_{L^2}^2  \mathrm{d} t \lesssim _{n_0,c_0,R,N,\chi,\kappa,T}1,~~~~ \forall \omega\in  \Omega _R^\epsilon \bigcap \widetilde{\Omega}_N^\epsilon.\label{3.255}
\end{align}
Moreover, let $\{\tau_{k}\}_{k \geq 1}$  be a family of stopping times satisfying $0 \leq \tau_{k} \leq T$, then
\begin{align} \label{3.24}
 \sup_{\epsilon\in (0,1)}\mathbb{E}  \|u_\epsilon(\tau_{k}+\theta)-u_\epsilon(\tau_{k})\|_{(D(A))^*}   \leq C \theta^{\frac{1}{2}},
\end{align}
for some constant $C>0$ independent of $\epsilon$, where $(D(A))^* $ denotes the dual space of $D(A)$.
\end{lemma}

\begin{proof}[\emph{\textbf{Proof}}]
For any $\varphi\in H^3(\mathbb{R}^2)$, we have
\begin{equation*}
\begin{split}
\left| \langle\frac{\mathrm{d}  n^\epsilon(t)}{ \mathrm{d} t }, \varphi\rangle_{H^{-3}, H^3}\right|&=|\langle -u^\epsilon\cdot \nabla n^\epsilon   + \Delta n ^\epsilon -  \textrm{div} \left(n^\epsilon[(\chi(c^\epsilon)\nabla c^\epsilon)*\rho^\epsilon]\right) , \varphi\rangle_{H^{-3}, H^3}|\\
& \lesssim_{c_0,\chi} \big(\|u^\epsilon\|_{L^2} \| n^\epsilon\|_{L^2}+  \|  n^\epsilon  \|_{L^2} +  \|n^\epsilon \|_{L^2} \| \nabla c^\epsilon \|_{L^2}    \big) (\|\nabla\varphi\|_{L^\infty}+\|\Delta\varphi\|_{L^2}).
\end{split}
\end{equation*}
Since $H^3(\mathbb{R}^2)\subset W^{1,\infty}(\mathbb{R}^2)$,  it follows from the last inequality that, for all $\omega\in  \Omega _R^\epsilon \bigcap \widetilde{\Omega} _N^\epsilon$
\begin{equation*}
\begin{split}
\int_0^T\left\| \frac{\mathrm{d}  n^\epsilon(t)}{ \mathrm{d} t }\right\|_{H^{-3}}^2  \mathrm{d} t& \lesssim_{c_0,\chi}  \int_0^T\|u^\epsilon\|_{L^2}^2 \| n^\epsilon\|_{L^2}^2 \mathrm{d} t+  \int_0^T\|  n^\epsilon  \|_{L^2}^2 \mathrm{d} t + \int_0^T \|n^\epsilon \|_{L^2}^2 \| \nabla c^\epsilon \|_{L^2} ^2 \mathrm{d} t\\
& \lesssim_{c_0,\chi,T}  \sup_{t\in [0,T]}\| n^\epsilon\|_{L^2}^2\left(\sup_{t\in [0,T]}\|u^\epsilon\|_{L^2}^2  +  \sup_{t\in [0,T]} \| \nabla c^\epsilon \|_{L^2} ^2 +1\right)\\
&\lesssim_{c_0,\chi,T} C_{n_0,c_0,R,\chi,\kappa} (N+R+1),
\end{split}
\end{equation*}
where the last inequality used (2) in Corollary \ref{cor}, and this implies \eqref{3.23}.

For \eqref{3.255}, it follows from \eqref{Mod-1}$_2$ that
\begin{equation*}
\begin{split}
\left\| \frac{\mathrm{d}  c^\epsilon(t)}{ \mathrm{d} t }\right\|_{L^2}^2 &\lesssim  \|u^\epsilon\cdot \nabla c^\epsilon \|_{L^2}^2+\|  \Delta c^\epsilon\|_{L^2}^2+ \|\kappa(c^\epsilon)(n^\epsilon*\rho^\epsilon) \|_{L^2}^2\\
&\lesssim _{c_0,\chi} \|u^\epsilon\|_{L^2}^2 \|\nabla c^\epsilon \|_{L^2}^2+\|  \Delta c^\epsilon\|_{L^2}^2+ \| n^\epsilon \|_{L^2}^2.
\end{split}
\end{equation*}
Then for all  $\omega\in  \Omega _R^\epsilon \bigcap \widetilde{\Omega} _N^\epsilon$, we have
\begin{equation*}
\begin{split}
\int_0^T\left\| \frac{\mathrm{d}  c^\epsilon(t)}{ \mathrm{d} t }\right\|_{L^2}^2  \mathrm{d} t &  \lesssim _{c_0,\chi,T} \sup_{t\in [0,T]}\|u^\epsilon\|_{L^2}^2 \sup_{t\in [0,T]}\|\nabla c^\epsilon \|_{L^2}^2  +\int_0^T \|  \Delta c^\epsilon\|_{L^2}^2 \mathrm{d} t+\sup_{t\in [0,T]} \| n^\epsilon \|_{L^2}^2 \\
&   \lesssim _{c_0,\chi,T} NR  +R+C_{n_0,c_0,R,\chi,\kappa}.
\end{split}
\end{equation*}
Now we prove \eqref{3.24}. For any stopping times $\{\tau_{k}\}_{k \geq 1}$ satisfying $0 \leq \tau_{k} \leq T$, we have
\begin{equation}\label{3.25}
\begin{split}
&u^\epsilon(\tau_{k}+\theta)-u^\epsilon(\tau_{k} )=-   \int_{\tau_{k}}^{\tau_{k}+\theta}\textbf{P} (u^\epsilon\cdot \nabla) u^\epsilon  \mathrm{d} t +   \int_{\tau_{k}}^{\tau_{k}+\theta}A u^\epsilon \mathrm{d} t \\
&\quad +   \int_{\tau_{k}}^{\tau_{k}+\theta}\textbf{P}[ (n^\epsilon\nabla \phi)*\rho^\epsilon] \mathrm{d} t+ \sum_{j\geq 1}  \int_{\tau_{k}}^{\tau_{k}+\theta}\textbf{P} f_j(t,u^\epsilon) \mathrm{d} W^j:= \mathcal {I}_1+\mathcal {I}_2+\mathcal {I}_3+\mathcal {I}_4.
\end{split}
\end{equation}
For $\mathcal {I}_1$, we get by using the divergence-free condition $u^\epsilon =0$ that
\begin{equation*}
\begin{split}
\mathbb{E}\|\mathcal {I}_1\|_{D(A)^*}&\lesssim \mathbb{E}\int_{\tau_{k}}^{\tau_{k}+\theta}\| u^\epsilon \otimes u^\epsilon  \|_{L^2}  \mathrm{d} t\lesssim \mathbb{E}\int_{\tau_{k}}^{\tau_{k}+\theta}\|u^\epsilon  \|_{L^2}\| \nabla u^\epsilon\|_{L^2} \mathrm{d} t\\
&\lesssim \theta^{1/2}\mathbb{E}\left(\sup_{t\in [\tau_{k} ,\tau_{k}+\theta]}\| u^\epsilon\|_{L^2}^2\int_{\tau_{k}}^{\tau_{k}+\theta}\| \nabla u^\epsilon\|_{L^2} ^2  \mathrm{d} t\right)^{1/2}\lesssim _{n_0,c_0,u_0,\phi,\kappa,\chi ,p,T}\theta^{ \frac{1}{2}}.
\end{split}
\end{equation*}
For $\mathcal {I}_2$, it follows from \eqref{3.15} that
\begin{equation*}
\begin{split}
\mathbb{E}\|\mathcal {I}_2\|_{D(A)^*}&\lesssim  \mathbb{E}\int_{\tau_{k}}^{\tau_{k}+\theta}\|u^\epsilon  \|_{L^2}  \mathrm{d} t  \lesssim  \theta\mathbb{E}\sup_{t\in [0,T]}\|u^\epsilon  \|_{L^2}^2 \lesssim _{n_0,c_0,u_0,\phi,\kappa,\chi ,p,T}\theta.
\end{split}
\end{equation*}
For $\mathcal {I}_3$, first note that
\begin{equation*}
\begin{split}
\mathbb{E}\|\mathcal {I}_3\|_{D(A)^*}&\lesssim  \mathbb{E} \int_{\tau_{k}}^{\tau_{k}+\theta}\| (n^\epsilon\nabla \phi)*\rho^\epsilon \|_{D(A)^*} \mathrm{d} t \lesssim_{\phi}  \theta ^{\frac{1}{2}} \mathbb{E} \int_{0}^{T} \|  n^\epsilon (t)\|_{L^2}^2  \mathrm{d} t.
\end{split}
\end{equation*}
By using Lemma \ref{lem9} and the Gagliardo-Nirenberg (GN) inequality (\cite[Lecture II]{nirenberg1959}), we have
\begin{equation*}
\begin{split}
 \|  n^\epsilon \|_{L^2}
  &\lesssim  \|  n^\epsilon \textbf{1}_{\{0<n^\epsilon<1\}}\|_{L^2} +\|\sqrt{n^\epsilon} \textbf{1}_{ \{n^\epsilon\geq1\}}\|_{L^4}^2 \\
& \lesssim\|  \sqrt{n^\epsilon }\textbf{1}_{\{0<n^\epsilon<1\}}\|_{L^2}+\|\sqrt{n^\epsilon} \textbf{1}_{ \{n^\epsilon\geq1\}}\|_{L^2}\|\nabla[\sqrt{n^\epsilon} \textbf{1}_{ \{n^\epsilon\geq1\}}]\|_{L^2}\\
& \lesssim_{n _0} 1 +  \left(\int_{\{n^\epsilon\geq 1\}} \frac{|\nabla (n^\epsilon+1) |^2}{n^\epsilon+1}\frac{n^\epsilon+1}{n^\epsilon}\mathrm{d}x\right)^{1/2}\\
& \lesssim_{n _0} 1 + \sqrt{2}\left(\int_{\{n^\epsilon\geq 1\}} \frac{|\nabla (n^\epsilon+1) |^2}{n^\epsilon+1} \mathrm{d}x\right)^{1/2}\\
& \lesssim_{n _0} 1 + \sqrt{2}\|\nabla  \sqrt{n^\epsilon+1}  \| _{L^2}.
\end{split}
\end{equation*}
It follows from the uniform bound \eqref{unf} in Lemma \ref{lem10} that
\begin{equation*}
\begin{split}
\mathbb{E}\|\mathcal {I}_3\|_{D(A)^*} \lesssim_{\phi,n_0}  \theta ^{\frac{1}{2}} \mathbb{E} \int_{0}^{T} (1 + \|\nabla  \sqrt{n^\epsilon+1}  \| _{L^2}^2)  \mathrm{d} t \lesssim _{n_0,c_0,u_0,\phi,\kappa,\chi ,p,T} \theta ^{\frac{1}{2}} .
\end{split}
\end{equation*}
For $\mathcal {I}_4$, we get by using the BDG inequality and \eqref{3.15} in Lemma \ref{cor} that
\begin{equation*}
\begin{split}
\mathbb{E}\|\mathcal {I}_4\|_{D(A)^*} 
& \lesssim \mathbb{E} \left(\int_{\tau_{k}}^{\tau_{k}+\theta} \sum_{j\geq 1} \| f_j(t,u^\epsilon) \|_{L_2(U;L^2)}^2  \mathrm{d} t\right)^{1/2}\\
& \lesssim \theta ^{\frac{1}{2}}\mathbb{E}   \sup_{t \in [\tau_{k},\tau_{k} +\theta]}   (1+ \|u^\epsilon (t)\|_{ L^2 }^2) \lesssim _{n_0,c_0,u_0,\phi,\kappa,\chi ,p,T} \theta ^{\frac{1}{2}}.
\end{split}
\end{equation*}
Plugging the estimates for $\mathcal {I}_1\sim\mathcal {I}_4$ into \eqref{3.25}, we get the desired inequality \eqref{3.24}. The proof of Lemma \ref{lem11} is completed.
\end{proof}

\begin{remark} \label{rem3.5}
By \eqref{3.16} in Lemma \ref{cor}, we have
\begin{equation*}
\begin{split}
 \mathbb{P}\{\widetilde{\Omega}_N^\epsilon\} \geq 1- \mathbb{P}\left\{\sup_{t\in[0,T]} \|u^\epsilon   (t) \|_{L^2}^2 >N\right\}-\mathbb{P}\left\{\int_0^T \| \nabla u^\epsilon  (t) \|_{L^2}^2  \mathrm{d} t>N\right\} \geq 1-\frac{2C_{n_0,c_0,u_0,\phi,\kappa,\chi ,p,T}}{N}.
\end{split}
\end{equation*}
\end{remark}

\subsection{Global solutions for KS-SNS system}
Based on the above established results, we are now in a position to prove the first main theorem.

\vspace{3mm}\noindent
\textbf{Proof of (a$_1$) in Theorem \ref{th1}.}  The proof is divided into two steps.

\textsf{Step 1 (Compactness)}.   Let $\{(n^\epsilon,c^\epsilon,u^\epsilon)\}_{\epsilon>0}$ be the sequence of global pathwise solutions to the modified system \eqref{Mod-1} constructed in Lemma \ref{lem8}.

$\bullet$ We denote by $\mathcal {L}[n^\epsilon]$ the law of $n^\epsilon$ on the phase space $\chi_{n^\epsilon}:= L^2(0,T;L^2_{\textrm{loc}}(\mathbb{R}^2))$. In view of the uniform bounds \eqref{3.23} in Lemma \ref{lem11} and \eqref{3.14} in Lemma \ref{cor}, it follows that for any $B_m \subset \mathbb{R}^2$ with radius $m\in \mathbb{N}$, there exists a constant $a_m>0$ such that the bound
$$
 \| n^\epsilon  \|_{L^2(0,T;H^{1}(B_m))}  + \| \partial_t n^\epsilon  \|_{L^2(0,T;H^{-3}(B_m))} \lesssim a_m,\quad \textrm{for all}~ \omega\in  \Omega _R^\epsilon \bigcap \widetilde{\Omega}_N^\epsilon
$$
holds  uniformly in $\epsilon$. Moreover, by  Theorem 2.1 in Chapter III of \cite{temam2001navier}, for any sequence of balls $\{B_m\}_{m\in\mathbb{N}}$, the space
$$
\textbf{Y}^m_n:=\{f\in \chi_{n^\epsilon};~ f\in L^2(0,T;H^1(B_m)) ,~f\in W^{1,2}(0,T;L^2(B_m)) \}
$$
is relatively compact in $L^2(0,T;L^2(B_m))$. Then we get from Remark \ref{rem3.3} and Remark \ref{rem3.5}  that
\begin{equation*}
\begin{split}
 \mathcal {L}[n^\epsilon]\{\|n^\epsilon\|_{\textbf{Y}_n^m} \leq a_m\} &\geq \mathbb{P}\{\Omega _R^\epsilon \bigcap \widetilde{\Omega}_N^\epsilon\}\\
 &=1-\mathbb{P}\{(\Omega _R^\epsilon)^C\}-\mathbb{P}\{  (\widetilde{\Omega}_N^\epsilon)^C\}\\
 &\geq 1-\frac{2C_{n_0,c_0,u_0,\phi,\kappa,\chi ,p,T}}{R}-\frac{2C_{n_0,c_0,u_0,\phi,\kappa,\chi ,p,T}}{N}.
\end{split}
\end{equation*}
Hence one can choose $R>N>0$ as large as one wish, which implies that the family $\{\mathcal {L}[n^\epsilon]:\epsilon \in (0,1)\}$ is tight on $\chi_{n^\epsilon}:= L^2(0,T;L^2_{\textrm{loc}}(\mathbb{R}^2))$.

$\bullet$  We denote by $\mathcal {L}[c^\epsilon]$ the law of $c^\epsilon$ on the space $\chi_{c^\epsilon}:= L^2(0,T;H^1_{\textrm{loc}}(\mathbb{R}^2))$. It follows from the estimate \eqref{3.16} in Lemma \ref{cor} that, for any ball $B_m$ as above, there exists a constant $a_m>0$ such that
$$
 \mathbb{E} \| c^\epsilon \|_{L^2(0,T;H^2(B_m))}^p  \lesssim a_m.
$$
Moreover, by \eqref{3.255} in Lemma \ref{lem11}, one can also assume that
$$
 \left\| \frac{\mathrm{d}  c^\epsilon }{ \mathrm{d} t}\right\|_{L^2(0,T;L^2(B_m))} \lesssim  \frac{a_m}{4},~~~~\textrm{for all}~ \omega\in  \Omega _R^\epsilon \bigcap \widetilde{\Omega}_N^\epsilon.
$$
Note that the following inclusion
$$
\textbf{Y}^m_c:=\left\{f\in \chi_{n^\epsilon};~ f\in L^2(0,T;H^2(B_m)) ,~\frac{\mathrm{d}  f}{ \mathrm{d} t}\in L^2(0,T;L^2(B_m)) \right\}\subset L^2(0,T;H^1(B_m))
$$
is relatively compact. It follows from the last two estimates that for $p>2$
\begin{equation*}
\begin{split}
 \mathcal {L}[c^\epsilon]\{\|c^\epsilon\|_{\textbf{Y}^m_c} > a_m\} &\leq \mathbb{P}\left\{\| c^\epsilon \|_{L^2(0,T;H^2(B_m))}>\frac{a_m}{2}\right\}+\mathbb{P}\left\{\| \frac{\mathrm{d}  c^\epsilon }{ \mathrm{d} t}\|_{L^2(0,T;L^2(B_m))}>\frac{a_m}{2}\right\}\\
 &\leq \frac{2^p}{a_m^p} \mathbb{E} \| c^\epsilon \|_{L^2(0,T;H^2(B_m))}^p+1-\mathbb{P}\left\{\| \frac{\mathrm{d}  c^\epsilon }{ \mathrm{d} t}\|_{L^2(0,T;L^2(B_m))}\leq\frac{a_m}{4}\right\}\\
 &\leq \frac{2^p}{a_m^{p-1}} +1-\mathbb{P}\left\{\Omega _R^\epsilon \bigcap \widetilde{\Omega}_N^\epsilon\right\}\\
 &\leq \frac{2^p}{a_m^{p-1}} + \frac{2C_{n_0,c_0,u_0,\phi,\kappa,\chi ,p,T}}{R}+\frac{2C_{n_0,c_0,u_0,\phi,\kappa,\chi ,p,T}}{N}.
\end{split}
\end{equation*}
Therefore, by taking $a_m>0$ and $R>N>0$ large enough, we conclude from the last inequality that the collection $\{\mathcal {L}[c^\epsilon]:\epsilon \in (0,1)\}$ is tight on $\chi_{c^\epsilon}:= L^2(0,T;H^1_{\textrm{loc}}(\mathbb{R}^2))$.

$\bullet$ We show that the probability measures $\mathcal {L}[u^\epsilon]$ associated to $u^\epsilon$ is tight on the space
$$
\chi_{u^\epsilon}:= \mathcal {C}([0, T]; (D(A))^*) \bigcap (L^2(0,T;H^1 (\mathbb{R}^2)),weak)\bigcap L^2([0, T]; L^2_{\textrm{loc}}(\mathbb{R}^2)).
$$
Indeed, the estimate \eqref{3.24} in Lemma \ref{lem11} informs us that the sequence $\{u^\epsilon\}_{\epsilon \in (0,1)}$ satisfies the Aldous condition with $\alpha=1$ and $\gamma=\frac{1}{2}$ (cf. \cites{aldous1978stopping,aldous1989stopping}, \cite{metivier1988stochastic} (Theorem 3.2, page 29)).
%
%
%
%
%
Hence,  the laws $\mathcal {L}[u^\epsilon]$ form a tight sequence on $\mathcal {D}([0, T]; (D(A))^*)$.  Moreover, since $ u^\epsilon $ takes values in $\mathcal {C}([0,T];L^2(\mathbb{R}^2))$ (by Lemma \ref{lem10} and the Aubin-Lions Lemma in \cite{temam2001navier}), it follows from the Proposition 1.6 in \cite{jakubowski1986skorokhod} that $\mathcal {L}[u^\epsilon]$ is also tight on $\mathcal {C}([0, T]; (D(A))^*)$ equipped with the uniform topology.

Now, by using the uniform bound \eqref{3.15} in Lemma \ref{cor} and the fact that any bounded ball $\mathbb{B}$ equipped with the weak topology in $L^2(0,T;H^1 (\mathbb{R}^2))$ is relatively compact, one can conclude that the collection of $\mathcal {L}[u^\epsilon]$ is tight on $(L^2(0,T;H^1 (\mathbb{R}^2)),weak)$. Furthermore, for any ball $B_m \in \mathbb{R}^2$, we get by \eqref{3.15} that there is a constant $a_m>0$ such that
\begin{equation}\label{232}
\begin{split}
\mathbb{E} \| u^\epsilon\|_{L^2(0,T;H^1(B_m))}^p \leq a_m,\quad p\geq2.
 \end{split}
\end{equation}
By \eqref{Mod-1}$_3$ and a similar method in Theorem 3.1 of \cite{flandoli1995martingale}, one can show that  $
\mathbb{E} \| u^\epsilon\|_{W^{\alpha,2}(0,T;(H^1(\mathbb{R}^2))^*)}^2 \leq C$, for some $C>0$ independent of $\epsilon$, and hence
\begin{equation}\label{233}
\begin{split}
\mathbb{E} \| u^\epsilon\|_{W^{\alpha,2}(0,T;(H^1(B_m))^*)}^2 \leq a_m.
 \end{split}
\end{equation}
Note that the following embedding
$$ L^2(0,T;H^1(B_m))\bigcap W^{\alpha,2}(0,T;(H^1(B_m))^*)\subset  L^2(0,T;L^2(B_m))$$ is compact (cf. \cite{temam2001navier}), we deduce from the boundedness of \eqref{232}-\eqref{233} that the family of measures $\mathcal {L}[u^\epsilon]$ is tight on $L^2(0,T;L^2(\mathbb{R}^2))$.

$\bullet$ Since $\mathcal {L}[W^\epsilon]$ is a single Radon measure on the Polish space $\chi_W:= \mathcal {C}([0,T];U_0)$, it is tight.

As a consequence, the joint laws $\mathcal {L}[n^\epsilon,c^\epsilon,u^\epsilon,W^\epsilon]=\mathbb{P}\circ(n^\epsilon,c^\epsilon,u^\epsilon,W)^{-1}$ form a tight sequence on
\begin{equation}
\begin{split}
\chi&:= \chi_{n^\epsilon}\times\chi_{c^\epsilon}\times\chi_{u^\epsilon}\times\chi_W\\
&=L^2(0,T;L^2_{\textrm{loc}}(\mathbb{R}^2))\times L^2(0,T;H^1_{\textrm{loc}}(\mathbb{R}^2)) \times \mathcal {C}([0, T]; (D(A))^*) \\
&\quad \cap (L^2(0,T;H^1 (\mathbb{R}^2)),weak)\cap L^2([0, T]; L^2_{\textrm{loc}}(\mathbb{R}^2))\times \mathcal {C}([0,T];U_0).
 \end{split}
\end{equation}
Therefore, we get from the Prohorov's theorem (cf. Theorem 5.1 in \cite{billingsley2013convergence}) that there exists a measure $\widetilde{\mathcal {L}}$ defined on $\chi$ such that, up to a subsequence $\{\epsilon_j\}_{j\geq 1}$, there holds
$$
\mathcal {L}[n^{\epsilon_j},c ^{\epsilon_j},u ^{\epsilon_j},W^{\epsilon_j}] \rightarrow  \widetilde{\mathcal {L}}~~ \textrm{weaky as} ~~ j\rightarrow\infty.
 $$
One can conclude from the Jakubowski-Skorokhod Representation Theorem (cf. \cite[Theorem 2]{jakubowski1998almost}) that, there exists a probability space $(\widetilde{\Omega},\widetilde{\mathcal {F}},\widetilde{\mathbb{P}})$, on this space defined a $\chi$-valued random variables $
(\widetilde{n},\widetilde{c},\widetilde{u},\widetilde{W})$  and a family of processes $(\widetilde{n}^{\epsilon_j},\widetilde{c} ^{\epsilon_j},\widetilde{u} ^{\epsilon_j},\widetilde{W}^{\epsilon_j})$
such that

\vspace{2mm}
\textsf{(d$_1$)} for any $\epsilon_j \in (0,1)$, the joint laws $
\mathcal {L}[\widetilde{n}^{\epsilon_j},\widetilde{c} ^{\epsilon_j},\widetilde{u} ^{\epsilon_j},\widetilde{W}^{\epsilon_j}]$ and $\mathcal {L}[n^{\epsilon_j},c ^{\epsilon_j},u ^{\epsilon_j},W^{\epsilon_j}]$ coincides on $\chi$;

\textsf{(d$_2$)}  the law $\mathcal {L}[\widetilde{n},\widetilde{c},\widetilde{u},\widetilde{W}]=\widetilde{\mathcal {L}}$ is a Radon measure on $\chi$, and as $j\rightarrow\infty$, we have $ \widetilde{\mathbb{P}}$-a.s.
\begin{subequations}
\begin{align}
\widetilde{n}^{\epsilon_j}\rightarrow\widetilde{n} \quad&\textrm{in} \quad  \chi_{n^\epsilon},\label{3.29a}\\
 \widetilde{c}^{\epsilon_j}\rightarrow\widetilde{c} \quad&\textrm{in} \quad \chi_{c^\epsilon},\label{3.29b}\\
 \widetilde{u}^{\epsilon_j}\rightarrow\widetilde{u} \quad&\textrm{in} \quad \chi_{u^\epsilon},\label{3.29c}\\
 \widetilde{W}^{\epsilon_j}\rightarrow\widetilde{W} \quad&\textrm{in}\quad \chi_{W};\label{3.29d}
 \end{align}
 \end{subequations}

\textsf{(d$_3$)} the quadruple $(\widetilde{n}^{\epsilon_j},\widetilde{c}^{\epsilon_j},\widetilde{u}^{\epsilon_j},\widetilde{W}^{\epsilon_j})$ satisfies the equalities in \eqref{1.2}, $\widetilde{\mathbb{P}}$-a.s.

\vspace{2mm}
\textsf{Step 2 (Identification of the limits)}.  On the probability space $(\widetilde{\Omega},\widetilde{\mathcal {F}},\widetilde{\mathbb{P}})$, we deduce from  \cite[Theorem 2.9.1]{breit2018stochastically} that $\widetilde{W}^{\epsilon_j}=\sum_{k\geq 1}e_k \widetilde{W}^{\epsilon_j,k}$ is a cylindrical Wiener process with respect to the filtration
$$
\widetilde{\mathcal {F}}_t^{\epsilon_j}:=\sigma\left\{\sigma_t[\widetilde{n}^{\epsilon_j}],\sigma_t[\widetilde{c}^{\epsilon_j}],\sigma_t[\widetilde{u}^{\epsilon_j}],
\cup_{k\geq 1} \sigma_t[\widetilde{W}^{\epsilon_j,k}]\right\},~~t \in [0,T],
$$
for all $\epsilon_j \in (0,1)$, and so $\widetilde{\mathcal {F}}_t$ is nonanticipative with respect to $\widetilde{W}^{\epsilon_j}$. By \eqref{3.29d}  and \cite[Lemma 2.9.3]{breit2018stochastically}, one can pass to the limit as $j \rightarrow\infty$ to conclude that
$$
\widetilde{\mathcal {F}}_t:=\sigma \left\{\sigma_t[\widetilde{n} ],\sigma_t[\widetilde{c} ],\sigma_t[\widetilde{u} ],
\cup_{k\geq 1} \sigma_t[\widetilde{W}^{ k}]\right\},~~t \in [0,T],
$$
is nonanticipative with respect to $\widetilde{W} $, which implies that $\widetilde{W}$ is a $\widetilde{\mathcal {F}}_t$-cylindrical Wiener process (cf. Lemma 2.1.35 and Corollary 2.1.36 in \cite{breit2018stochastically}).

Now we have all in hand to proceed with the verification of the identities in \eqref{1.2}. We will only show in detail how to deal with the stochastic equation with respect to $u$ since the identification of the first two equations in \eqref{KS-SNS} is standard and similar to \cites{duan2010global,chae2014global,zhang2022global,zhangliu202202}. Indeed, one can use \eqref{3.29c} to conclude that for all $t\in [0,T]$ and $\psi\in C^\infty_ \textrm{div}  (\mathbb{R}^2) $,
\begin{equation}\label{1}
\begin{split}
 &(\widetilde{u}^{\epsilon_j}(t),\psi)_{L^2}- (\widetilde{u}^{\epsilon_j}_0,\psi)_{L^2} + \int_0^t (\nabla \widetilde{u}^{\epsilon_j}, \nabla \psi)_{L^2} \mathrm{d} t -  \int_0^t ( (\widetilde{n}^{\epsilon_j}\nabla \phi)*\rho^\epsilon,\psi)_{L^2} \mathrm{d} t \\
 &\quad \rightarrow (\widetilde{u}(t),\psi)_{L^2}- (\widetilde{u}_0,\psi)_{L^2} + \int_0^t (\nabla \widetilde{u}, \nabla \psi)_{L^2} \mathrm{d} t -  \int_0^t ( (\widetilde{n}\nabla \phi)*\rho^\epsilon,\psi)_{L^2} \mathrm{d} t,~~\textrm{as} ~j\rightarrow\infty,~\widetilde{\mathbb{P}}\textrm{-a.s.}
\end{split}
\end{equation}
Note that
$$
(  \widetilde{u}^{\epsilon_j}\otimes \widetilde{u}^{\epsilon_j},\nabla\psi)_{L^2}-(\widetilde{u} \otimes \widetilde{u} ,\nabla\psi)_{L^2}=(( \widetilde{u}^{\epsilon_j}-\widetilde{u})\otimes\widetilde{u}^{\epsilon_j},\nabla\psi)_{L^2}+(\widetilde{u} \otimes ( \widetilde{u}^{\epsilon_j}-\widetilde{u}) ,\nabla\psi)_{L^2}.
$$
From the conclusion \eqref{3.29c}, we have as $j\rightarrow\infty$
$$
\widetilde{u}^{\epsilon_j}-\widetilde{u}\rightarrow 0 ~~ \textrm{weakly in} ~~ L^2(0,T;L^2 (\mathbb{R}^2)) ~~\textrm{and}~~\widetilde{u}^{\epsilon_j}- \widetilde{u}\rightarrow 0 ~~\textrm{a.e. on}  ~~[0,T] \times \mathbb{R}^2 , ~~ \widetilde{\mathbb{P}}\textrm{-a.s.,}
$$
and it implies
$$
( \widetilde{u}^{\epsilon_j}-\widetilde{u})\otimes\widetilde{u}^{\epsilon_j}\rightarrow 0 ~~\textrm{and} ~~ \widetilde{u} \otimes ( \widetilde{u}^{\epsilon_j}-\widetilde{u})\rightarrow 0 ~~\textrm{weakly in} ~~ L^1(0,T;L^1_{\textrm{loc}} (\mathbb{R}^2))~~\textrm{as}~ j\rightarrow\infty,~~ \widetilde{\mathbb{P}}\textrm{-a.s.}
$$
Hence, there holds
\begin{equation}\label{2}
\begin{split}
 \int_0^t(  \widetilde{u}^{\epsilon_j}\otimes \widetilde{u}^{\epsilon_j},\nabla\psi)_{L^2} \mathrm{d} t\rightarrow \int_0^t(\widetilde{u} \otimes \widetilde{u} ,\nabla\psi)_{L^2} \mathrm{d} t ~~\textrm{as} ~~j\rightarrow\infty,~\widetilde{\mathbb{P}}\textrm{-a.s.}
\end{split}
\end{equation}
To deal with the stochastic integration, we first conclude from \eqref{3.29c} that
$$
\widetilde{u}^{\epsilon_j} \rightarrow \widetilde{u}~~  \textrm{strongly in} ~~ \mathcal {C}([0,T];(D(A))^*)\subseteq L^2(0,T;H^{-l}(\mathbb{R}^2)) ~~ \textrm{as} ~~ j\rightarrow\infty,
$$
for some $l>0$ large enough, $\widetilde{\mathbb{P}}$-a.s., and then condition \textrm{(A$_3$)} implies that
\begin{equation}\label{3.30}
\begin{split}
f (t,\widetilde{u}^{\epsilon_j})\rightarrow f (t,\widetilde{u} ) ~~\textrm{in} ~~ L^2(0,T;L_2(U;H^{-l}(\mathbb{R}^2))) ~~as~~ j\rightarrow\infty,~~\widetilde{\mathbb{P}}\textrm{-a.s.}
\end{split}
\end{equation}
By \eqref{3.30} and \eqref{3.29d}, it follows from the Lemma 2.6.6 in \cite{breit2018stochastically} (see also Lemma 2.4.35 in \cite{mensah2019stochastic} and Lemma 5.3 in \cite{donatelli2020combined}) that
\begin{equation}\label{3}
\begin{split}
 \int_0^T   f (t,\widetilde{u}^{\epsilon_j}) \mathrm{d}  W^{\epsilon_j}\rightarrow\int_0^T   f (t,\widetilde{u} ) \mathrm{d}  W ~~~ \textrm{in probability},~ as ~  j\rightarrow\infty .
\end{split}
\end{equation}
According to the convergence \eqref{1}, \eqref{2} and \eqref{3}, one can now take the limit as $j\rightarrow\infty$ in the $u^{\epsilon_j}$-equation of \eqref{Mod-1} to obtain that, for each $\varphi \in C^\infty_{0, \textrm{div} }(\mathbb{R}^2;\mathbb{R}^2)$
\begin{equation*}
\begin{split}
(\widetilde{u}(t),\varphi)_{L^2}=& (\widetilde{u}_0,\varphi)_{L^2} +\int_0^t  ( \widetilde{u}\otimes \widetilde{u},\nabla\varphi)_{L^2} \mathrm{d}r -\int_0^t ( \nabla\widetilde{u},\nabla\varphi)_{L^2}\mathrm{d}r \\
 &+ \int_0^t(\widetilde{n}\nabla \phi,\varphi)_{L^2}\mathrm{d}r+ \sum_{k\geq 1}\int_0^t (f_k(s,\widetilde{u}),\varphi)_{L^2} \mathrm{d} \widetilde{W}^k,\quad \widetilde{\mathbb{P}}\textrm{-a.s.,}
\end{split}
\end{equation*}
for all $t \in [0,T]$.  By applying the Fatou Lemma and Lemma \ref{unf}, we infer that  $$\widetilde{u}\in \mathcal {C}([0,T];L^2(\mathbb{R}^2))\bigcap L^2(0,T;H^1(\mathbb{R}^2)),~~~ \widetilde{\mathbb{P}}\textrm{-a.s.}$$ In a similar manner, one can also verify that $(\widetilde{n},\widetilde{c})$ satisfies the $n$ and $c$-equations in \eqref{KS-SNS} in the sense of distribution. As a result, the quadruple $(\widetilde{n},\widetilde{c},\widetilde{u},\widetilde{W})$ defined on the stochastic basis $(\widetilde{\Omega},\widetilde{\mathcal {F}},\{\widetilde{\mathcal {F}}_t\}_{t>0},\widetilde{\mathbb{P}})$ is a global martingale weak solution to the KS-SNS system \eqref{KS-SNS}. \qed

\vspace{2mm}\noindent
\textbf{Proof of (a$_2$) in Theorem \ref{th1}.} According to the well-known Yamada-Wantanabe Theorem \cites{yamada1971uniqueness,watanabe1971uniqueness,da2014stochastic} and the Gy\"{o}ngy-Krylov characterization of the convergence in probability \cite{gyongy1980stochastic}, one can prove the existence and uniqueness of pathwise solutions (strong in probability), provided the existence of martingale solutions and the pathwise uniqueness result.

Based on the result (a$_1$) of Theorem \ref{th1}, it suffices to prove the pathwise uniquness. Precisely, for any $T>0$, let $(n_1,c_2,u_3)$ and $(n_1,c_2,u_3)$ be two global martingale solutions to the KS-SNS system \eqref{KS-SNS} with the same initial data $(n_0,c_0,u_0)$ under the stochastic basis $(\widetilde{\Omega},\widetilde{\mathcal {F}},\{\widetilde{\mathcal {F}}_t\}_{t>0},\widetilde{\mathbb{P}})$, we shall prove that
$$
\widetilde{\mathbb{P}}\{(n_1,c_1,u_1)(t)=(n_2,c_2,u_2)(t),~\forall t\in [0,T]  \}=1.
$$
Without loss of generality, we assume that $\chi(\cdot)\equiv 1$.  For simplicity, we set
$$n^*=n_1-n_2, ~~c^*=c_1-c_2,~~ u^*=u_1-u_2, ~~ P^*=P_1-P_2.$$
By the chain rule and Young inequality, we get from the $n$-equation that
\begin{equation}\label{3.35}
\begin{split}
 & \frac{1}{2}\|n^*(t)\|^2_{L^2} +\int_0^t \|\nabla n^*(t)\|^2_{L^2} \mathrm{d}r = \int_0^t (u^*   n_2,\nabla n^*)_{L^2} \mathrm{d}r\\
  &\quad +\int_0^t (n^* \nabla c_2,\nabla n^*)_{L^2} \mathrm{d}r -\int_0^t (n_1 \nabla c^*,\nabla n^*)_{L^2} \mathrm{d}r := A_1+ A_2+ A_3 .
\end{split}
\end{equation}
By using the Young inequality and the GN inequality that
\begin{equation}\label{3.36}
\begin{split}
 |A_1 |
 &\leq \epsilon \int_0^t\|\nabla n^*\|_{L^2}^2 \mathrm{d}r +C_\epsilon\int_0^t\|u^*\|_{L^2} \|\nabla u^*\|_{L^2}  \|n_2\|_{L^2} \|\nabla n_2\|_{L^2}  \mathrm{d}r\\
 &\leq \epsilon \int_0^t(\|\nabla n^*\|_{L^2}^2+\|\nabla u^*\|_{L^2}^2 ) \mathrm{d}r+C_\epsilon\int_0^t \|u^*\|_{L^2} ^2 \|n_2\|_{L^2}^2 \|\nabla n_2\|_{L^2}^2 \mathrm{d}r.
\end{split}
\end{equation}
For any $\epsilon>0$, we get by Young inequality that
\begin{equation}\label{3.37}
\begin{split}
 |A_2 |
 &\leq  \epsilon \int_0^t\|\nabla n^*\|_{L^2}^2 \mathrm{d}r +C_{\epsilon}\int_0^t  \|n ^*\|_{L^2}\|\nabla n ^*\|_{L^2}  \|\nabla c_2\|_{L^2}\|  c_2\|_{H^2}  \mathrm{d}r\\
 &\leq  2\epsilon \int_0^t\|\nabla n^*\|_{L^2}^2 \mathrm{d}r +C_{\epsilon}\int_0^t  \|n ^*\|_{L^2} ^2  \|\nabla c_2\|_{L^2}^2\|  c_2\|_{H^2}^2    \mathrm{d}r.
\end{split}
\end{equation}
By GN inequality,  we have
\begin{equation}\label{3.38}
\begin{split}
 |A_3 |
 &\leq    \epsilon \int_0^t\|\nabla n^*\|_{L^2}^2 \mathrm{d}r +C_{\epsilon}\int_0^t \|\nabla c^*\|_{L^2}\| c^*\|_{H^2}  \|n_1\|_{L^2}\|\nabla n_1\|_{L^2}  \mathrm{d}r\\
 &\leq    \epsilon \int_0^t(\|\nabla n^*\|_{L^2}^2+\| c^*\|_{H^2}^2) \mathrm{d}r +C_\epsilon  \int_0^t \|\nabla c^*\|_{L^2}^2 \|n_1\|_{L^2}^2\|\nabla n_1\|_{L^2}^2  \mathrm{d}r.
\end{split}
\end{equation}
Plugging the estimates \eqref{3.36}-\eqref{3.38} into \eqref{3.35}, we infer that
\begin{equation}\label{3.39}
\begin{split}
 & \frac{1}{2}\|n^*(t)\|^2_{L^2} +\int_0^t \|\nabla n^*(r)\|^2_{L^2} \mathrm{d}r \leq 4\epsilon \int_0^t(\|\nabla n^*\|_{L^2}^2+\|\nabla u^*\|_{L^2}^2+ \| c^*\|_{H^2}^2) \mathrm{d}r\\
 &\quad +C_\epsilon\int_0^t (\|u^*\|_{L^2} ^2 \|n_2\|_{L^2}^2 \|\nabla n_2\|_{L^2}^2+\|n ^*\|_{L^2} ^2 \|\nabla c_2\|_{L^2}^2\|  c_2\|_{H^2}^2 + \|\nabla c^*\|_{L^2}^2 \|n_1\|_{L^2}^2\|\nabla n_1\|_{L^2}^2 )\mathrm{d}r.
\end{split}
\end{equation}
To close the estimate \eqref{3.39}, we have to derive estimates for $\| c^*\|_{H^1}$ and $\|u^*(t)\| _{L^2}$, respectively. Indeed, by applying the chain rule to $\mathrm{d} \| c^*(t)\|_{L^2}^2$ and integrating by parts, we find
\begin{equation}\label{3.40}
\begin{split}
 &  \frac{1}{2} \|c^*(t)\|^2_{L^2} + \int_0^t \|\nabla c^*(r)\|^2_{L^2}\mathrm{d}r=    \int_0^t (\nabla c^*,u^* c_2 )_{L^2}\mathrm{d}r-\int_0^t (c^*,n^* \kappa(c_2) )_{L^2}\mathrm{d}r\\
  &\quad +  \int_0^t (c^*,  n_1 (\kappa(c_2)-\kappa(c_1)))_{L^2}\mathrm{d}r:= B_1+B_2+B_3,
\end{split}
\end{equation}
and
\begin{equation}\label{3.40-1}
\begin{split}
 &  \frac{1}{2} \|\nabla c^*(t)\|^2_{L^2} + \int_0^t \|\Delta c^*(r)\|^2_{L^2}\mathrm{d}r\\
 &\quad=   \int_0^t (\Delta c^*,  u^* \cdot \nabla c_1 )_{L^2}\mathrm{d}r- \int_0^t (\Delta c^*, u_2 \cdot\nabla c^* )_{L^2}\mathrm{d}r - \int_0^t (\Delta c^*, n^* \kappa(c_1) )_{L^2}\mathrm{d}r\\
 &\quad\quad- \int_0^t (\Delta c^*, n_2 (\kappa(c_1)-\kappa(c_2)) )_{L^2}\mathrm{d}r\\
 & \quad\leq \epsilon \int_0^t \|\Delta c^*(r)\|^2_{L^2}\mathrm{d}r+ \int_0^t (\|u^* \cdot \nabla c_1\|_{L^2}^2+\|u_2 \cdot\nabla c^*\|_{L^2}^2+\|n^* \kappa(c_1) \|_{L^2}^2\\
 &\quad \quad+\|n_2 (\kappa(c_1)-\kappa(c_2)) \|_{L^2}^2) \mathrm{d}r        := \epsilon \int_0^t \|\Delta c^*(r)\|^2_{L^2}\mathrm{d}r+D_1+D_2+D_3+D_4.
\end{split}
\end{equation}
For any $\epsilon>0$, it follows from the GN inequality that
\begin{equation} \label{3.41}
\begin{split}
 |B_1| &\leq \epsilon \int_0^t\|\nabla c^*\|_{L^2}^2 \mathrm{d}r + C_\epsilon \int_0^t\|u^*\|_{L^2}\|\nabla u^*\|_{L^2} \|c_2\|_{L^2}\|\nabla c_2\|_{L^2} \mathrm{d}r\\
 &\leq \epsilon \int_0^t(\|\nabla c^*\|_{L^2}^2+\|\nabla u^*\|_{L^2}^2) \mathrm{d}r + C_\epsilon \int_0^t\|u^*\|_{L^2} ^2\|c_2\|_{L^2}^2\|\nabla c_2\|_{L^2}^2 \mathrm{d}r.
\end{split}
\end{equation}
Since $\kappa(0)=0$, we get by the continuity of $\kappa(\cdot)$ and Mean Value Theorem that
\begin{equation}\label{3.42}
\begin{split}
 |B_2| &\leq  \sup_{r\in [0,\|c_2\|_{L^\infty}]}|\kappa'(r)|\int_0^t \|c^*\|_{L^2}\|n^*   \| _{L^2}\|c_2\|_{L^\infty}\mathrm{d}r \lesssim _{c_0,\kappa} \int_0^t (\|c^*\|_{L^2}^2 +  \|n^* \| _{L^2}^2)\mathrm{d}r.
\end{split}
\end{equation}
Similarly, we have
\begin{equation}\label{3.43}
\begin{split}
 |B_3| &\leq \sup_{r\in [0,\|c_1\|_{L^\infty}+\|c_2\|_{L^\infty}]}|\kappa'(r)|\int_0^t \|c^*\|_{L^4}^2 \| n_1 \|_{L^2}\mathrm{d}r \\
 &\leq C_{\kappa,c_0}\int_0^t \|c^*\|_{L^2} \|\nabla c^*\|_{L^2}  \| n_1 \|_{L^2}\mathrm{d}r\\
  &\leq \epsilon\int_0^t \|\nabla c^*\|_{L^2} ^2\mathrm{d}r+ C_{\kappa,c_0,\epsilon }\int_0^t \|c^*\|_{L^2} ^2 \| n_1 \|_{L^2}^2\mathrm{d}r.
\end{split}
\end{equation}
For $D_1$, we use the GN inequality to obtain
\begin{equation}\label{3.44}
\begin{split}
 |D_1|&\leq  \int_0^t \|u^*\|_{L^4}^2\| \nabla c_1\|_{L^4}^2 \mathrm{d}r\lesssim \int_0^t \|u^*\|_{L^2}\|\nabla u^*\|_{L^2} \| c_1\|_{L^2}\| \nabla c_1\|_{L^2}  \mathrm{d}r\\
 & \leq \epsilon \int_0^t  \|\nabla u^*\|_{L^2}^2 \mathrm{d}r+ C_{\epsilon} \int_0^t \|u^*\|_{L^2}^2 \| c_1\|_{L^2}^2\| \nabla c_1\|_{L^2}^2  \mathrm{d}r.
\end{split}
\end{equation}
For $D_2$, there holds
\begin{equation}\label{3.45}
\begin{split}
 |D_2|&\leq  \int_0^t \|u_2\|_{L^2} \|\nabla u_2\|_{L^2}\| \nabla c^*\|_{L^2}\|  c^*\|_{H^2}  \mathrm{d}r\\
 & \leq \epsilon \int_0^t\|  c^*\|_{H^2}^2  \mathrm{d}r+C_{\epsilon}\int_0^t  \|u_2\|_{L^2}^2 \|\nabla u_2\|_{L^2}^2\| \nabla c^*\|_{L^2}^2   \mathrm{d}r .
\end{split}
\end{equation}
The terms $D_3$ and $D_4$ can be estimated as $B_2$ and $B_3$, and we have
\begin{equation}\label{3.46}
\begin{split}
 |D_3|+|D_4|&\leq  \epsilon \int_0^t(\| \nabla n^*\|_{L^2}^2 +\| \nabla c^*\|_{L^2}^2) \mathrm{d}r \\
 &+C_\epsilon \int_0^t(\| n^*\|_{L^2}^2\|c_1\|_{L^2}^2 \|\nabla c_1\|_{L^2}^2+\| n_2\|_{L^2}^2\| \nabla n_2\|_{L^2}^2\| c^*\|_{L^2}^2) \mathrm{d}r.
\end{split}
\end{equation}
Combining the estimates \eqref{3.41}-\eqref{3.46}, we deduce from \eqref{3.40} and \eqref{3.40-1} that
\begin{equation} \label{3.47}
\begin{split}
 &  \frac{1}{2} (\|c^*(t)\|^2_{L^2}+ \|\nabla c^*(t)\|^2_{L^2}) + \int_0^t (\|\nabla c^*(r)\|^2_{L^2}+\|\Delta c^*(r)\|^2_{L^2}) \mathrm{d}r \\
  & \quad \leq  3\epsilon \int_0^t(\| \nabla n^*\|_{L^2}^2 +\|  c^*\|_{H^2}^2+\| \nabla u^*\|_{L^2}^2) \\
&\quad\quad+C_{\kappa,c_0,\epsilon } \int_0^t \Big(\| \nabla c^*\|_{L^2}^2 \|u_2\|_{L^2}^2 \|\nabla u_2\|_{L^2}^2+    \|c^*\|_{L^2} ^2 (\| n_1 \|_{L^2}^2+\| n_2\|_{L^2}^2\| \nabla n_2\|_{L^2}^2+1)  \\
  & \quad\quad+  (\|u^*\|_{L^2}^2+\| n^*\|_{L^2}^2) (\| c_1\|_{L^2}^2\| \nabla c_1\|_{L^2}^2+\|c_2\|_{L^2}^2\|\nabla c_2\|_{L^2}^2+1) \Big) \mathrm{d}r.
\end{split}
\end{equation}
To estimate $\|u^* \|_{L^2}$, we now apply It\^{o}'s formula to $ \|u^* (t)\|_{L^2}^2$, after integrating by parts and using the GN inequality, we infer that
\begin{equation}\label{3.48}
\begin{split}
&\|u^* (t)\|_{L^2}^2 +   \int_0^t \|\nabla u^* (r)\|_{L^2}^2  \mathrm{d}r\\
&= \int_0^t  (\nabla u^*, u^*\otimes u_1 )_{L^2}  \mathrm{d}r+ \int_0^t  (u^*,n^*\nabla \phi)_{L^2}  \mathrm{d}r+   \int_0^t  \|f(t,u_1)-f(t,u_2)\|^2_{L_2(U;L^2)} \mathrm{d}r \\
&\quad + \sum_{j\geq 1}\int_0^t  (u^*,f_j(t,u_1)-f_j(t,u_2))_{L^2} \mathrm{d}  W^j\\
&\leq \epsilon \int_0^t \|\nabla u^*\|_{L^2}^2 \mathrm{d}r+ C_\phi \int_0^t \Big(\| u^*\|_{L^2} ^2(\| u_1\|_{L^2}^2\|\nabla  u_1\|_{L^2}^2+1)+ \|n^*\|_{L^2}^2\Big)\mathrm{d}r\\
&\quad+ 2\sum_{j\geq 1}\int_0^t  (u^*,f_j(t,u_1)-f_j(t,u_2))_{L^2} \mathrm{d}  W^j.
\end{split}
\end{equation}
Recall the following inequality (cf.  \cite[Proposition 7.2]{taylor1996partial}):
$$
\| h \|_{H^2} ^2 \lesssim \| h \|_{L^2} ^2+\| \nabla h \|_{L^2} ^2+\|\Delta h \|_{L^2} ^2,~~\forall h\in H^2(\mathbb{R}^2).
$$
Putting the estimates \eqref{3.38}, \eqref{3.47} and \eqref{3.48} together, and  choosing $\epsilon>0$ small enough, the term $\mathscr{D}(t)$ can be estimated as
\begin{equation*}
\begin{split}
 \mathscr{D}(t)\leq C_{\phi,\chi,c_0}\int_0^t  \mathscr{D}(r) \mathscr{F}(r)  \mathrm{d}r+ 2\sum_{j\geq 1}\int_0^t  (u^*,f_j(t,u_1)-f_j(t,u_2))_{L^2} \mathrm{d}  W^j,~~ \forall t\in [0,T],
\end{split}
\end{equation*}
where
\begin{equation*}
\begin{split}
 \mathscr{D}(t)&:=\|n^*(t)\|^2_{L^2}+\|c^*(t)\|^2_{L^2}+ \|\nabla c^*(t)\|^2_{L^2}+\|u^* (t)\|_{L^2}^2,\\
\mathscr{F}(t)&:=   \|n_2\|_{L^2}^2 \|\nabla n_2\|_{L^2}^2+  \|\nabla c_2\|_{L^2}^2\|  c_2\|_{H^2}^2 +   \|n_1\|_{L^2}^2\|\nabla n_1\|_{L^2}^2+   \|u_2\|_{L^2}^2 \|\nabla u_2\|_{L^2}^2\\
 & \quad +  \| c_1\|_{L^2}^2\| \nabla c_1\|_{L^2}^2 +  \| u_1\|_{L^2}^2\|\nabla  u_1\|_{L^2}^2+  \| n_1 \|_{L^2}^2+1.
\end{split}
\end{equation*}
By using the Gronwall Lemma, we arrive at
\begin{equation}\label{3.481}
\begin{split}
 \mathscr{D}(t)\leq  2 e^{C_{\phi,\kappa, c_0}\int_0^t  \mathscr{F}(r)  \mathrm{d}r}\sup_{r\in [0,t]}\left|\sum_{j\geq 1}\int_0^r  (u^*,f_j(r,u_1)-f_j(r,u_2))_{L^2} \mathrm{d}  W^j\right|.
\end{split}
\end{equation}
In order to apply the BDG inequality to \eqref{3.481}, we define $\textbf{t}^R =\textbf{t}^R_1\bigwedge\textbf{t}^R_2\bigwedge T$ with
\begin{equation*}
\begin{split}
\textbf{t}^R_i:=&\inf \left\{t>0;~\sup _{t\in[0, T]} \|n_i(t) \|_{L^2}^2 \bigvee \int_0^T \|\nabla n_i(t) \|_{L^2}^2 \mathrm{d}r \bigvee \sup _{t\in[0, T]} \|c_i(t) \|_{H^1}^2\right.\\
 &\left.\bigvee \int_0^T \|\nabla c_i(t) \|_{H^2}^2\mathrm{d}r \bigvee \sup _{t\in[0, T]} \|u_i(t) \|_{L^2}^2 \bigvee \int_0^T \|\nabla u_i(t) \|_{L^2}^2\mathrm{d}r \geq R\right\},~~i=1,2.
\end{split}
\end{equation*}
Since $ (n_i, c_i, u_i )$, $i=1,2$ are solutions to the system  \eqref{Mod-1}, it is clear that $
\textbf{t}^R \nearrow T$  as $R\rightarrow\infty$, $\widetilde{\mathbb{P}}$-a.s.
Therefore, one can derive from \eqref{3.481} that
\begin{equation}
\begin{split}
 \mathscr{D}(t\wedge \textbf{t}^R)&\leq  2 e^{C_{\phi,\kappa, c_0}\int_0^{t\wedge \textbf{t}^R}  \mathscr{F}(r)  \mathrm{d}r}\sup_{r\in [0,t\wedge \textbf{t}^R]}\left|\sum_{j\geq 1}\int_0^r  (u^*,f_j(r,u_1)-f_j(r,u_2))_{L^2} \mathrm{d}  W^j\right|\\
 &\lesssim _{\phi,\kappa, c_0,R}\sup_{r\in [0,t\wedge \textbf{t}^R]}\left|\sum_{j\geq 1}\int_0^r  (u^*,f_j(r,u_1)-f_j(r,u_2))_{L^2} \mathrm{d}  W^j\right|,
\end{split}
\end{equation}
which together with  the BDG inequality and the definition of $\mathscr{D}(t)$ yield that
\begin{equation*}
\begin{split}
\mathbb{E} \sup_{r\in [0,t\wedge \textbf{t}^R]}\mathscr{D}(r)&\lesssim _{\phi,\kappa, c_0,R}\mathbb{E} \left(\int_0^{t\wedge \textbf{t}^R} \sum_{j\geq 1} \|u^*\|_{L^2}^2\|f_j(r,u_1)-f_j(r,u_2)\|_{L^2} ^2 \mathrm{d}r\right)^{1/2}\\
&\leq \frac{1}{2}\mathbb{E}\sup_{r\in [0,t\wedge \textbf{t}^R]}\|u^*(r)\|_{L^2}^2 + C_{\phi,\kappa, c_0,R} \mathbb{E}\int_0^{t\wedge \textbf{t}^R}  \|u^*(r)\|_{L^2} ^2 \mathrm{d}r \\
&\leq \frac{1}{2}\mathbb{E}\sup_{r\in [0,t\wedge \textbf{t}^R]}\mathscr{D}(r) + C_{\phi,\kappa, c_0,R} \int_0^{t } \mathbb{E}\sup_{\tau\in [0,r\wedge \textbf{t}^R]} \mathscr{D}(\tau) \mathrm{d}r,~~\forall t >0.
\end{split}
\end{equation*}
Using the Gronwall Lemma to the last inequality yield that
$$
\mathbb{E} \sup_{t\in [0,T\wedge \textbf{t}^R]}\mathscr{D}(t)=0, ~~\textrm{for any}~ T>0.
$$
Taking the limit as $R\rightarrow\infty$ leads to $\mathbb{E} \sup_{t\in [0,T ]}\mathscr{D}(t)=0,$ which implies the uniqueness result. The proof of Theorem \ref{th1} is now completed.
\qed

\section{Proof of Theorem \ref{th2}}\label{sec4}
The scheme for proving Theorem \ref{th2} is similar to  Theorem \ref{th1}, which is based on the a new uniform entropy functional inequality for the approximate solutions $(n^\epsilon,u^\epsilon,u^\epsilon)$ constructed in Lemma \ref{lem8}.  Precisely, we have the following result:

\begin{lemma} \label{lem4.1}
Let $\epsilon >0$ and $T>0$, and assume that  $(n^\epsilon,c^\epsilon,u^\epsilon)$ is the unique global pathwise solution in $L^2(\Omega;\mathcal {C}([0,T];\textbf{H}^{s}(\mathbb{R}^2)))\bigcap L^2(\Omega;L^2([0,T];\textbf{H}^{s+1}(\mathbb{R}^2)))$ in Lemma \ref{lem8}.  If the diffusion coefficients $d_i$, $i=1,2,3$ satisfy the assumption {(\textsf{B$_1$})}, then  for any $p\geq 1$, the estimate
\begin{equation}\label{entropy}
\begin{split}
 &\mathbb{E} \sup_{t\in [0,T]} \left(\int_{\mathbb{R}^2}n^\epsilon (|x|+|\ln n^\epsilon| )\mathrm{d}x +\|\nabla c^\epsilon\|_{L^2}^2+ \frac{4\Lambda^4 \|  c_0\|_{L^\infty}^2}{ d_3 d_2} \|u^\epsilon \|_{L^2}^2 \right)^p \\
 &\quad +  \mathbb{E}\left(\int_0^T \bigg(  \frac{d_1}{2}\|  \nabla\sqrt{n ^\epsilon}\|_{L^2}^2 + \frac{d_2}{2} \|\Delta c^\epsilon\|_{L^2}^2 +\frac{2\Lambda^4 \|  c_0\|_{L^\infty}^2}{ d_2} \|\nabla u^{\epsilon}\|_{L^2}^2\bigg) \mathrm{d} t \right)^p \\
 &\quad\lesssim_{d_1,d_2,d_3,n_0,c_0,u_0,p,T,\kappa,\chi,\varrho, \Lambda,\tilde{\Lambda}_p } 1
\end{split}
\end{equation}
holds uniformly in $\epsilon\in (0,1)$, where $\Lambda$ is provided  in condition {(\textsf{B$_2$})} and $\tilde{\Lambda}_p$ is a general constant in the BDG inequality depending only on $p$.

\end{lemma}

\begin{proof}[\emph{\textbf{Proof}}]
Recalling Lemma \ref{lem9}, it follows from the condition  {(\textsf{B$_1$})} that
$$
M_\kappa:=   \max_{t\in[0,\|c_0\|_{L^\infty}]} |\kappa (t)|,\quad M_\chi:=   \max_{t\in[0,\|c_0\|_{L^\infty}]} |\chi (t)|.
$$
Applying the chain rule to $\mathrm{d} (n^\epsilon \ln n^\epsilon)$ and integrating by parts over $\mathbb{R}^2$, we get from \eqref{Mod-1}$_1$ and the divergence-free condition $ \textrm{div}  u^\epsilon =0$ that
\begin{equation}\label{4.1}
\begin{split}
 \int_{\mathbb{R}^2}n^\epsilon \ln n^\epsilon \mathrm{d}x +d_1\int_0^t\int_{\mathbb{R}^2} \frac{|\nabla n ^\epsilon|^2}{n^\epsilon }\mathrm{d}x \mathrm{d}r = \int_{\mathbb{R}^2}n^\epsilon_0 \ln n^\epsilon_0 \mathrm{d}x+ \int_0^t\int_{\mathbb{R}^2} \nabla n^\epsilon \cdot[(\chi(c^\epsilon)\nabla c^\epsilon)*\rho^\epsilon] \mathrm{d}x \mathrm{d}r.
\end{split}
\end{equation}
For the last term in \eqref{4.1}, we get by Young inequality that
\begin{equation*}
\begin{split}
&\int_0^t\int_{\mathbb{R}^2} \nabla n^\epsilon \cdot[(\chi(c^\epsilon)\nabla c^\epsilon)*\rho^\epsilon] \mathrm{d}x \mathrm{d}r\\
& \quad \leq \frac{d_1}{2}\int_0^t\int_{\mathbb{R}^2}\frac{|\nabla n ^\epsilon|^2}{n^\epsilon } \mathrm{d}x \mathrm{d}r + \frac{\epsilon}{4}\int_0^t\int_{\mathbb{R}^2}  | \nabla c^\epsilon|^4  \mathrm{d}x \mathrm{d}r+  \frac{M_\chi^4}{4\epsilon d_1^2} \int_0^t\int_{\mathbb{R}^2}  |n^\epsilon |^2  \mathrm{d}x \mathrm{d}r.
\end{split}
\end{equation*}
We get from \eqref{4.1} that
\begin{equation} \label{4.2}
\begin{split}
 &\int_{\mathbb{R}^2}n^\epsilon \ln n^\epsilon \mathrm{d}x +\frac{d_1}{2}\int_0^t\int_{\mathbb{R}^2} \frac{|\nabla n ^\epsilon|^2}{n^\epsilon }\mathrm{d}x \mathrm{d}r\\
 &\quad\leq \int_{\mathbb{R}^2}n^\epsilon_0 \ln n^\epsilon_0 \mathrm{d}x+ \frac{  d_2}{ 2\Lambda^4 \|  c_0\|_{L^\infty}^2} \int_0^t\| \nabla c^\epsilon\|^4_{L^4} \mathrm{d}r+   \frac{\Lambda^4 \|  c_0\|_{L^\infty}^2 M_\chi^4}{8 d_2  d_1^2} \int_0^t\int_{\mathbb{R}^2}  |n^\epsilon |^2  \mathrm{d}x\\
 &\quad\leq \int_{\mathbb{R}^2}n^\epsilon_0 \ln n^\epsilon_0 \mathrm{d}x+ \frac{d_2}{2}\int_0^t\|\Delta c^{\epsilon}\|_{L^2}^2 \mathrm{d}r+  \frac{\Lambda^4 \|  c_0\|_{L^\infty}^2 M_\chi^4}{8 d_2  d_1^2} \int_0^t\|n^\epsilon \|_{L^2}^2  \mathrm{d}x,
\end{split}
\end{equation}
where the last inequality used Lemma \ref{lem9} and the GN inequality
\begin{equation} \label{gn}
\begin{split}
\|\nabla c^{\epsilon}\|_{L^4} \leq \Lambda \|  c^{\epsilon}\|_{L^\infty}^\frac{1}{2}\|\Delta c^{\epsilon}\|_{L^2}^\frac{1}{2}\leq \Lambda \|  c_0\|_{L^\infty}^\frac{1}{2}\|\Delta c^{\epsilon}\|_{L^2}^\frac{1}{2},\quad \textrm{for some}~ \Lambda>0.
\end{split}
\end{equation}
To close the estimate \eqref{4.2}, let us first apply the operator $\nabla $ to \eqref{Mod-1}$_1$ to obtain
$$
\mathrm{d}  \nabla c^\epsilon+ \nabla(u^\epsilon\cdot \nabla c^\epsilon)   \mathrm{d} t = d_2\Delta \nabla c^\epsilon \mathrm{d} t-\nabla[\kappa(c^\epsilon)(n^\epsilon*\rho^\epsilon)]  \mathrm{d} t.
$$
Taking the chain rule to $\mathrm{d}  \|\nabla c^\epsilon\|_{L^2}^2$ and integrating by parts over $\mathbb{R}^2$, we deduce from last equality that
\begin{equation} \label{4.3}
\begin{split}
 &  \|\nabla c^\epsilon\|_{L^2}^2+2d_2\int_0^t\|\Delta c^\epsilon\|_{L^2}^2\mathrm{d}r=\|\nabla c^\epsilon_0\|_{L^2}^2+2\int_0^t(\Delta c^\epsilon,  u^\epsilon\cdot \nabla c^\epsilon )_{L^2} \mathrm{d}r\\
 &\quad+2\int_0^t(\Delta c^\epsilon, \kappa(c^\epsilon)(n^\epsilon*\rho^\epsilon) )_{L^2}\mathrm{d}r.
\end{split}
\end{equation}
Using \eqref{gn} and the fact of $ \textrm{div}  u^\epsilon =0$, the second terms on the R.H.S of \eqref{4.3} can be estimated as
\begin{equation} \label{4.4}
\begin{split}
 2\int_0^t(\Delta c^\epsilon,  u^\epsilon\cdot \nabla c^\epsilon )_{L^2} \mathrm{d}r&=-2\int_0^t\int_{\mathbb{R}^2} \partial_ic^{\epsilon} \partial_i u^{\epsilon,j}\partial_jc^{\epsilon}\mathrm{d}x\mathrm{d}r-2\int_0^t\int_{\mathbb{R}^2} u^{\epsilon,j}\partial_ic^{\epsilon}\partial_i\partial_jc^{\epsilon}\mathrm{d}x\mathrm{d}r\\
 &=-2\int_0^t\int_{\mathbb{R}^2} \partial_ic^{\epsilon} \partial_i u^{\epsilon,j}\partial_jc^{\epsilon}\mathrm{d}x\mathrm{d}r \\
 &\leq  \frac{  d_2}{2\Lambda^4 \|  c_0\|_{L^\infty}^2} \int_0^t\|\nabla c^{\epsilon}\|_{L^4}^4\mathrm{d}r+ \frac{2\Lambda^4 \|  c_0\|_{L^\infty}^2}{   d_2 }\int_0^t\|\nabla u^{\epsilon}\|_{L^2}^2\mathrm{d}r\\
 &\leq \frac{ d_2 }{2} \int_0^t\|\Delta c^{\epsilon}\|_{L^2}^2\mathrm{d}r+ \frac{2\Lambda^4 \|  c_0\|_{L^\infty}^2}{   d_2 }\int_0^t\|\nabla u^{\epsilon}\|_{L^2}^2\mathrm{d}r.
\end{split}
\end{equation}
Moreover, we have
\begin{equation} \label{4.5}
\begin{split}
2\int_0^t(\Delta c^\epsilon, \kappa(c^\epsilon)(n^\epsilon*\rho^\epsilon) )_{L^2}\mathrm{d}r&\leq \frac{d_2}{2}  \int_0^t\|\Delta c^{\epsilon}\|_{L^2}^2\mathrm{d}r + \frac{ 2}{d_2}\int_0^t\|\kappa(c^\epsilon)(n^\epsilon*\rho^\epsilon) \|_{L^2}^2\mathrm{d}r\\
&\leq \frac{d_2}{2}  \int_0^t\|\Delta c^{\epsilon}\|_{L^2}^2\mathrm{d}r + \frac{2M_\kappa^2}{d_2}\int_0^t\| n^\epsilon\|_{L^2}^2\mathrm{d}r.
\end{split}
\end{equation}
Putting the estimates \eqref{4.4}-\eqref{4.5} into \eqref{4.3} leads to
\begin{equation} \label{4.6}
\begin{split}
 &  \|\nabla c^\epsilon\|_{L^2}^2+  d_2 \int_0^t\|\Delta c^\epsilon\|_{L^2}^2\mathrm{d}r\leq\|\nabla c^\epsilon_0\|_{L^2}^2+\frac{2\Lambda^4 \|  c_0\|_{L^\infty}^2}{ d_2} \int_0^t\|\nabla u^{\epsilon}\|_{L^2}^2\mathrm{d}r + \frac{M_\kappa^2}{d_2}\int_0^t\| n^\epsilon\|_{L^2}^2\mathrm{d}r.
\end{split}
\end{equation}
Putting the estimates \eqref{4.2} and \eqref{4.6} together, we arrive at
\begin{equation} \label{4.7}
\begin{split}
 &  \int_{\mathbb{R}^2}n^\epsilon \ln n^\epsilon \mathrm{d}x +\|\nabla c^\epsilon\|_{L^2}^2+2d_1 \int_0^t \|  \nabla\sqrt{n ^\epsilon}\|_{L^2}^2  \mathrm{d}r+ \frac{d_2}{2}\int_0^t\|\Delta c^\epsilon\|_{L^2}^2\mathrm{d}r\\
 &\quad\leq \int_{\mathbb{R}^2}n^\epsilon_0 \ln n^\epsilon_0 \mathrm{d}x+\|\nabla c^\epsilon_0\|_{L^2}^2+\frac{2\Lambda^4 \|  c_0\|_{L^\infty}^2}{ d_2} \int_0^t\|\nabla u^{\epsilon}\|_{L^2}^2\mathrm{d}r\\
  &\quad+ \left( \frac{\Lambda^2 \|n_0 \|_{L^1}M_\kappa^2}{d_2}+  \frac{\Lambda^6 \|n_0 \|_{L^1} \|  c_0\|_{L^\infty}^2 M_\chi^4}{8 d_2  d_1^2}\right) \int_0^t  \|\nabla\sqrt{n^\epsilon}\|_{L^2}^2 \mathrm{d}x,
\end{split}
\end{equation}
where we used the GN inequality to obtain
\begin{equation} \label{cnn}
\begin{split}
\|n^\epsilon \|_{L^2} \leq \Lambda \|\sqrt{n^\epsilon} \|_{L^2} \|\nabla\sqrt{n^\epsilon}\|_{L^2}  \leq \Lambda \|n_0 \|_{L^1}^{\frac{1}{2}} \|\nabla\sqrt{n^\epsilon}\|_{L^2}.
\end{split}
\end{equation}
To estimate the last term on the R.H.S. of \eqref{4.7}, we apply It\^{o}'s formula to $ \|u^\epsilon\|_{L^2}^2$, after integrating by parts with respect to $x $ and using the fact of $(u^\epsilon,(u^\epsilon\cdot\nabla)u^\epsilon)_{L^2}=0$, we obtain
\begin{equation} \label{4.9}
\begin{split}
 \|u^\epsilon(t)\|_{L^2}^2 + d_3\int_0^t\|\nabla u^\epsilon\|_{L^2}^2\mathrm{d}r&=\|u^\epsilon_0\|_{L^2}^2 + 2\int_0^t(u^\epsilon,\textbf{P}[ (n^\epsilon\nabla \phi)*\rho^\epsilon])_{L^2}\mathrm{d}r\\
 &+ \int_0^t\|\textbf{P} f(t,u^\epsilon) \|_{L_2(U;L^2)}^2\mathrm{d}r + 2\sum_{j\geq 1}\int_0^t(u^\epsilon,\textbf{P} f_j(t,u^\epsilon) )_{L^2}\mathrm{d}  W^j\\
 &\leq\|u^\epsilon_0\|_{L^2}^2 +(\| \phi\|_{W^{1,\infty}}^2+\varrho+1)\int_0^t (1+\|u^\epsilon\|_{L^2}^2)\mathrm{d}r\\
 &+\Lambda^2 \|n_0 \|_{L^1}\int_0^t \|\nabla\sqrt{n^\epsilon}\|_{L^2}^2\mathrm{d}r + 2\left|\sum_{j\geq 1}\int_0^t(u^\epsilon,\textbf{P} f_j(t,u^\epsilon) )_{L^2}\mathrm{d}  W^j\right|,
\end{split}
\end{equation}
where we used the condition {(\textsf{A$_3$})} and the property $(Au^\epsilon,u^\epsilon)_{L^2}=\|\nabla u^\epsilon\|_{L^2}$. Multiplying both sides of \eqref{4.9} by $  \frac{4\Lambda^4 \|  c_0\|_{L^\infty}^2}{ d_3 d_2} $, it then follows from \eqref{4.7} and \eqref{4.9} that
\begin{equation} \label{4.10}
\begin{split}
 &  \int_{\mathbb{R}^2}n^\epsilon \ln n^\epsilon \mathrm{d}x +\|\nabla c^\epsilon\|_{L^2}^2+ \frac{4\Lambda^4 \|  c_0\|_{L^\infty}^2}{ d_3 d_2}\|u^\epsilon(t)\|_{L^2}^2 +2d_1 \int_0^t \|  \nabla\sqrt{n ^\epsilon}\|_{L^2}^2  \mathrm{d}r\\
 &\quad + \frac{d_2}{2}\int_0^t\|\Delta c^\epsilon\|_{L^2}^2\mathrm{d}r+\frac{2\Lambda^4 \|  c_0\|_{L^\infty}^2}{ d_2} \int_0^t\|\nabla u^{\epsilon}\|_{L^2}^2\mathrm{d}r\\
 &\quad\leq \int_{\mathbb{R}^2}n^\epsilon_0 \ln n^\epsilon_0 \mathrm{d}x+\|\nabla c^\epsilon_0\|_{L^2}^2+\frac{4\Lambda^4 \|  c_0\|_{L^\infty}^2}{ d_3 d_2}\|u^\epsilon_0\|_{L^2}^2 \\
  &\quad+ \left( \frac{\Lambda^2 \|n_0 \|_{L^1}M_\kappa^2}{d_2}+  \frac{\Lambda^6 \|n_0 \|_{L^1} \|  c_0\|_{L^\infty}^2 M_\chi^4}{8 d_2  d_1^2}+\frac{4\Lambda^6 \|n_0 \|_{L^1}\|  c_0\|_{L^\infty}^2}{ d_3 d_2}\right ) \int_0^t  \|\nabla\sqrt{n^\epsilon}\|_{L^2}^2 \mathrm{d}x \\
  &\quad+\frac{4\Lambda^4 \|  c_0\|_{L^\infty}^2}{ d_3 d_2}(\| \phi\|_{W^{1,\infty}}^2+\varrho+1)\int_0^t (1+\|u^\epsilon\|_{L^2}^2)\mathrm{d}r\\
 &\quad + \frac{8\Lambda^4 \|  c_0\|_{L^\infty}^2}{ d_3 d_2}\left|\sum_{j\geq 1}\int_0^t(u^\epsilon,\textbf{P} f_j(t,u^\epsilon) )_{L^2}\mathrm{d}  W^j\right|.
\end{split}
\end{equation}
In order to bound $\int_{\mathbb{R}^2}n^\epsilon \ln n^\epsilon \mathrm{d}x$ from below, let us multiply \eqref{Mod-1}$_1$ by the smooth function $\langle x\rangle=(1+|x|^2)^{1/2}$ and integrating by parts over $\mathbb{R}^2$, it follows that
\begin{equation} \label{4.11}
\begin{split}
2\int_{\mathbb{R}^2} \langle x\rangle n^\epsilon \mathrm{d}x&=2\int_{\mathbb{R}^2} \langle x\rangle n^\epsilon_0 \mathrm{d}x+2\int_0^t\int_{\mathbb{R}^2}n^\epsilon u^\epsilon \cdot\nabla \langle x\rangle \mathrm{d}x\mathrm{d}r\\
& + 2d_1\int_0^t\int_{\mathbb{R}^2} n ^\epsilon \Delta \langle x\rangle \mathrm{d}x\mathrm{d}r+ 2 \int_0^t\int_{\mathbb{R}^2} \nabla \langle x\rangle \cdot\left(n^\epsilon[(\chi(c^\epsilon)\nabla c^\epsilon)*\rho^\epsilon]\right)\mathrm{d}x\mathrm{d}r.
\end{split}
\end{equation}
Notice that $\|\nabla\langle x\rangle\|_{L^\infty},\|\Delta \langle x\rangle\|_{L^\infty}\leq 1$, one can estimate the terms on the R.H.S. of \eqref{4.11} as
\begin{equation*}
\begin{split}
2\int_0^t\int_{\mathbb{R}^2}n^\epsilon u^\epsilon \cdot\nabla \langle x\rangle \mathrm{d}x\mathrm{d}r
&\leq2\Lambda \|n_0 \|_{L^1}^{\frac{1}{2}} \int_0^t \|\nabla\sqrt{n^\epsilon}\|_{L^2}\|u^\epsilon\|_{L^2}\mathrm{d}r\\
&\leq \frac{d_1}{2} \int_0^t \|\nabla\sqrt{n^\epsilon}\|_{L^2}^2\mathrm{d}r + \frac{2\Lambda^2 \|n_0 \|_{L^1} }{  d_1}\int_0^t \|u^\epsilon \|_{L^2}^2\mathrm{d}r,\\
2d_1\int_0^t\int_{\mathbb{R}^2} n ^\epsilon \Delta \langle x\rangle  \mathrm{d}x\mathrm{d}r &\leq 2d_1\int_0^t\|n ^\epsilon \|_{L^1}\mathrm{d}r= 2d_1 t\|n ^\epsilon _0\|_{L^1}\leq  2d_1 t\|n _0\|_{L^1} ,\\
2 \int_0^t\int_{\mathbb{R}^2} \nabla \langle x\rangle \cdot\left(n^\epsilon[(\chi(c^\epsilon)\nabla c^\epsilon)*\rho^\epsilon]\right)\mathrm{d}x\mathrm{d}r&\leq 2 \int_0^t\|n^\epsilon\|_{L^2}\|(\chi(c^\epsilon)\nabla c^\epsilon)*\rho^\epsilon\|_{L^2} \mathrm{d}r\\
&\leq 2\Lambda \|n_0 \|_{L^1}^{\frac{1}{2}}M_\chi\int_0^t \|\nabla\sqrt{n^\epsilon}\|_{L^2}\| \nabla c^\epsilon \|_{L^2} \mathrm{d}r\\
&\leq  \frac{d_1}{2}\int_0^t \|\nabla\sqrt{n^\epsilon}\|_{L^2}^2\mathrm{d}r+ \frac{2\Lambda^2 \|n_0 \|_{L^1} M_\chi^2}{ d_1} \int_0^t \| \nabla c^\epsilon \|_{L^2}^2 \mathrm{d}r,
\end{split}
\end{equation*}
where we used \eqref{cnn} in the first and third estimates. Putting the last three inequalities into \eqref{4.11}, we get
\begin{equation} \label{4.12}
\begin{split}
\int_{\mathbb{R}^2} 2 \langle x\rangle  n^\epsilon \mathrm{d}x&\leq 2\int_{\mathbb{R}^2}\langle x\rangle n^\epsilon_0 \mathrm{d}x+2 d_1 t\|n _0\|_{L^1}+d_1 \int_0^t \|\nabla\sqrt{n^\epsilon}\|_{L^2}^2\mathrm{d}r \\
& +\frac{2\Lambda^2 \|n_0 \|_{L^1} M_\chi^2}{d_1} \int_0^t \| \nabla c^\epsilon \|_{L^2}^2 \mathrm{d}r+ \frac{2\Lambda^2 \|n_0 \|_{L^1} }{ d_1}\int_0^t \|u^\epsilon\|_{L^2}^2\mathrm{d}r.
\end{split}
\end{equation}
Being inspired by the decomposition in (cf. \cite[p.649]{liu2011coupled}), let us consider
\begin{equation*}
\begin{split}
 \int_{\mathbb{R}^2} n^\epsilon |\ln n^\epsilon|\mathrm{d}x= \int_{\mathbb{R}^2} n^\epsilon  \ln n^\epsilon \mathrm{d}x +2 \int_{\mathbb{R}^2} n^\epsilon  \ln \frac{1}{n^\epsilon} \textbf{1}_{\{0< n^\epsilon\leq e^{-|x|}\} }\mathrm{d}x+2 \int_{\mathbb{R}^2} n^\epsilon  \ln \frac{1}{n^\epsilon} \textbf{1}_{ \{e^{-|x|}< n^\epsilon\leq 1\}}\mathrm{d}x.
\end{split}
\end{equation*}
Hence by \eqref{4.12} and the fact of $x\ln \frac{1}{x}\leq \sqrt{x}$ for all $x>0$, there holds
\begin{equation}\label{cccc}
\begin{split}
 0\leq\int_{\mathbb{R}^2} n^\epsilon |\ln n^\epsilon|\mathrm{d}x &\leq \int_{\mathbb{R}^2} n^\epsilon  \ln n^\epsilon \mathrm{d}x +2 \int_{\mathbb{R}^2}e^{-|x|}\mathrm{d}x+2 \int_{\mathbb{R}^2} |x| n^\epsilon \mathrm{d}x\\
 &\leq \int_{\mathbb{R}^2}n^\epsilon (\ln n^\epsilon+  2|x|) \mathrm{d}x +4,
\end{split}
\end{equation}
which implies that $\int_{\mathbb{R}^2}n^\epsilon (\ln n^\epsilon+  2|x|) \mathrm{d}x$ is bounded from below. Combining \eqref{4.10}, \eqref{4.12}  and \eqref{cccc}, it follows form the condition \textbf{(B$_3$)} that for all $t\in (0,T]$
\begin{equation} \label{4.15}
\begin{split}
\mathscr{F}(t) + \int_0^t \mathscr{G}(t)\mathrm{d}r&\leq \mathscr{F}(0) +2d_1 t\|n _0\|_{L^1} +\Theta_1 \int_0^t \| \nabla c^\epsilon \|_{L^2}^2 \mathrm{d}r+ \Theta_2\int_0^t (1+\|u^\epsilon\|_{L^2}^2)\mathrm{d}r\\
  &  + \Theta_3\left|\sum_{j\geq 1}\int_0^t(u^\epsilon,\textbf{P} f_j(t,u^\epsilon) )_{L^2}\mathrm{d}  W^j\right|,
\end{split}
\end{equation}
where
\begin{equation*}
\begin{split}
&\mathscr{F}(t) := \int_{\mathbb{R}^2}n^\epsilon (\ln n^\epsilon+  2\xi) \mathrm{d}x+4 +\|\nabla c^\epsilon\|_{L^2}^2+ \frac{4\Lambda^4 \|  c_0\|_{L^\infty}^2}{ d_3 d_2}(1+\|u^\epsilon \|_{L^2}^2) ,\\
& \mathscr{G}(t) :=\frac{d_1}{2}\|  \nabla\sqrt{n ^\epsilon}\|_{L^2}^2 + \frac{d_2}{2} \|\Delta c^\epsilon\|_{L^2}^2 +\frac{2\Lambda^4 \|  c_0\|_{L^\infty}^2}{ d_2} \|\nabla u^{\epsilon}\|_{L^2}^2,\\
&\Theta_1 := \frac{2\Lambda^2 \|n_0 \|_{L^1} M_\chi^2}{d_1},~~
\Theta_2 :=\frac{2\Lambda^2 \|n_0 \|_{L^1} }{  d_1} +\frac{4\Lambda^4 \|  c_0\|_{L^\infty}^2(\| \phi\|_{W^{1,\infty}}^2+\varrho+1)}{ d_3 d_2},~~
\Theta_3 :=  \frac{8\Lambda^4 \|  c_0\|_{L^\infty}^2}{ d_3 d_2}.
\end{split}
\end{equation*}
To deal with the last stochastic integral in \eqref{4.15}, for any $p>1$ and $T>0$, one can apply the BDG inequality and utilizing the condition (\textsf{B$_2$}) and Young inequality $ab\leq\epsilon a^2+\frac{1}{4\epsilon}b^2$ to derive that
\begin{equation}\label{4.16}
\begin{split}
 & \mathbb{E}\sup_{t\in [0,T]}\left|\Theta_3\sum_{j\geq 1}\int_0^t(u^\epsilon,\textbf{P} f_j(t,u^\epsilon) )_{L^2}\mathrm{d}  W^j\right|^p\\
 &\quad \leq \tilde{\Lambda}_p\mathbb{E} \left[\sup_{t\in [0,T]}\|u^\epsilon(t)\|_{L^2}^p\left(\Theta_3^2\int_0^T\sum_{j\geq 1}\|\textbf{P} f_j(t,u^\epsilon) \|_{ L^2 }^2 \mathrm{d} t\right)^{p/2}\right]\\
 &\quad \leq \frac{1}{2}(\frac{\Theta_3}{2})^p \mathbb{E} \sup_{t\in [0,T]}\|u^\epsilon(t)\|_{L^2}^{2p}+ 2^{p-1}  \tilde{\Lambda}_p^2\Theta_3^{p}\varrho ^p \mathbb{E}\left(\int_0^T 1+\|u^\epsilon\|_{H^s}^2 \mathrm{d} t\right)^p\\
 &\quad \leq \frac{1}{2} \mathbb{E} \sup_{t\in [0,T]}\mathscr{F}(t)^{ p}+ 2^{p-1}  \tilde{\Lambda}_p^2\Theta_3^{p}\varrho ^p \mathbb{E}\left(\int_0^T 1+\|u^\epsilon\|_{H^s}^2  \mathrm{d} t\right)^p,
\end{split}
\end{equation}
where $\tilde{\Lambda}_p$ denotes the general constant in the BDG inequality (cf. \cite{applebaum2009levy}). Thanks to the lower-bound in \eqref{cccc}, one can raise the $p$-th ($p\geq1$) power on both sides of \eqref{4.15}, after taking the supremum over $[0,T]$ and then the expected values, we get from \eqref{4.16} that
\begin{equation} \label{ddd}
\begin{split}
 &\mathbb{E}\sup_{t\in [0,T]}\mathscr{F}(t)^p + \mathbb{E}\left(\int_0^T \mathscr{G}(t)\mathrm{d}r\right)^p\\
 &\quad \leq  \mathscr{F}(0)^p +(2d_1\|n _0\|_{L^1})^p T^p + (\Theta_1^p+\Theta_2^p)\mathbb{E}\left( \int_0^T \| \nabla c^\epsilon \|_{L^2}^2 \mathrm{d}r+ \int_0^t (1+\|u^\epsilon\|_{L^2}^2)\mathrm{d}r\right)^p \\
  &\quad +\frac{1}{2} \mathbb{E} \sup_{t\in [0,T]}\mathscr{F}(t)^{ p}+ 2^{p-1}  \tilde{\Lambda}_p^2\Theta_3^{p}\varrho ^p \mathbb{E}\left(\int_0^T 1+\|u^\epsilon\|_{H^s}^2  \mathrm{d} t\right)^p.
\end{split}
\end{equation}
After some direct calculations, it follows from the inequality \eqref{ddd} that for any $\tau\in [0,T]$
\begin{equation}\label{4.17}
\begin{split}
\mathbb{E}\sup_{t\in [0,\tau]}\mathscr{F}(t)^p
  &\leq 2\mathscr{F}(0)^p +2^{p+1}d_1^p\|n _0\|_{L^1}^p T^p  \\
  &+\left(2 (\Theta_1^p+\Theta_2^p)(1+\frac{2}{\Theta_3})^p +4^{p}  \tilde{\Lambda}_p^2 \varrho ^p\right)  \int_0^\tau\mathbb{E}\sup_{r\in [0,t]}\mathscr{F}(r)^p  \mathrm{d} t.
\end{split}
\end{equation}
Applying the Gronwall Lemma to \eqref{4.17} leads to
\begin{equation*}
\begin{split}
\mathbb{E}\sup_{t\in [0,T]}\mathscr{F}(t)^p \leq (2\mathscr{F}(0)^p +2^{p+1}d_1^p\|n _0\|_{L^1}^p T^p)e^{(2 (\Theta_1^p+\Theta_2^p)(1+\frac{2}{\Theta_3})^p +4^{p}  \tilde{\Lambda}_p^2 \varrho ^p)T},
\end{split}
\end{equation*}
which combined with \eqref{ddd} gives the  bound for $\mathbb{E} (\int_0^T \mathscr{G}(t)\mathrm{d}r )^p$ uniformly in $\epsilon$. This completes the proof of Lemma \ref{lem4.1}.
\end{proof}

Based on Lemma \ref{lem10} and the entropy functional inequality in Lemma \ref{lem4.1}, we could now give the proof of Theorem \ref{th2}.

\vspace{3mm}\noindent
\textbf{Proof of Theorem \ref{th2}.} The method is very similar to the corresponding one for Theorem \ref{th1} and we only give a sketch for completeness.

\textsf{Step 1.} The Lemma \ref{lem9} for modified system \eqref{Mod-1} remains to be true, that is, for any $T>0$
\begin{align}
&n^{\epsilon}(x,t)\geq 0,~~ c^{\epsilon}(x,t)\geq 0,~~ ~ (t,x)\in [0,T]\times\mathbb{R}^2,\label{4.20}\\
&\| n^\epsilon(\cdot,t)\|_{L^1}\leq\| n _0\|_{L^1},~~ \|c^\epsilon(\cdot,t)\|_{L^1\cap L^\infty}\leq\|c_0\|_{L^1\cap L^\infty},~~ ~ t\in[0,T].\label{4.21}
\end{align}
For any $p\geq2$, it follows from Lemma \ref{lem4.1} that $ \{\sqrt{n^\epsilon}\}_{\epsilon >0}\in L^p(\Omega; L^2(0,T;H^1(\mathbb{R}^2)))$ is uniformly bounded, which together with \eqref{4.20} and the GN inequality yield that $n^\epsilon\in L^p(\Omega; L^2(0,T;H^1(\mathbb{R}^2)))$ is bounded.

Moreover,  $\{c^\epsilon\}_{\epsilon >0}$ is uniformly bounded in $L^p(\Omega; L^\infty(0,T;H^1(\mathbb{R}^2))\cap L^2(0,T;H^2(\mathbb{R}^2)))$ and $\{u^\epsilon\}_{\epsilon >0}$ is uniformly bounded in $L^p(\Omega; L^\infty(0,T;L^2(\mathbb{R}^2))\cap L^2(0,T;H^1(\mathbb{R}^2)))$.

\textsf{Step 2.} Similar to the argument in Section 3.2 (cf. Lemma \ref{lem11}) and the \textsf{Step 1} in the proof of Theorem \ref{th1}, one can prove that the joint laws $\mathcal {L}[n^\epsilon,c^\epsilon,u^\epsilon,W^\epsilon]$ of $(n^\epsilon,c^\epsilon,u^\epsilon,W)$ is tight on
\begin{equation}
\begin{split}
\mathcal {X}&:=L^2(0,T;L^2_{\textrm{loc}}(\mathbb{R}^2))\times L^2(0,T;H^1_{\textrm{loc}}(\mathbb{R}^2)) \times C([0, T]; (D(A))^*) \\
&\quad \cap (L^2(0,T;H^1 (\mathbb{R}^2)),weak)\cap L^2([0, T]; L^2_{\textrm{loc}}(\mathbb{R}^2))\times C([0,T];U_0).
 \end{split}
\end{equation}
Therefore, by using the stochastic compactness method (similar to \textsf{Step 2} in the proof of Theorem \ref{th1}), the Prohorov  theorem, one can prove that  there exists a new probability space $(\widetilde{\Omega},\widetilde{\mathcal {F}},\widetilde{\mathbb{P}})$ with $\mathscr{X}$-valued random variables $
(\widetilde{n},\widetilde{c},\widetilde{u},\widetilde{W})$  and $(\widetilde{n}^{\epsilon_j},\widetilde{c} ^{\epsilon_j},\widetilde{u} ^{\epsilon_j},\widetilde{W}^{\epsilon_j})$
such that

\begin{itemize}
\item [ {(e$_1$)}]  $\mathcal {L}[\widetilde{n}^{\epsilon_j},\widetilde{c} ^{\epsilon_j},\widetilde{u} ^{\epsilon_j},\widetilde{W}^{\epsilon_j}]=\mathcal {L}[n^{\epsilon_j},c ^{\epsilon_j},u ^{\epsilon_j},W^{\epsilon_j}]$ coincides on $\mathscr{X}$;

\item [ {(e$_2$)}]  the law $\mathcal {L}[\widetilde{n},\widetilde{c},\widetilde{u},\widetilde{W}]$ on $\mathscr{X}$ is a Radon measure, and

    \begin{equation}\label{eeee}
     \begin{split}
    (\widetilde{n}^{\epsilon_j},\widetilde{c} ^{\epsilon_j},\widetilde{u} ^{\epsilon_j},\widetilde{W}^{\epsilon_j})\rightarrow (\widetilde{n},\widetilde{c},\widetilde{u},\widetilde{W})~~  \textrm{as} ~~ j\rightarrow\infty , ~~ \widetilde{\mathbb{P}}\textrm{-a.s.;}
     \end{split}
\end{equation}

\item [ {(e$_3$)}] the quadruple $(\widetilde{n}^{\epsilon_j},\widetilde{c}^{\epsilon_j},\widetilde{u}^{\epsilon_j},\widetilde{W}^{\epsilon_j})$ satisfies the system \eqref{Mod-1} in the sense of distribution, that is, \eqref{Mod-1} admits a global martingale weak solution.
\end{itemize}

Thanks to the almost surely convergence  \eqref{eeee}, one can take the limit as $j\rightarrow\infty$ in the system \eqref{Mod-1} to show that the quadruple $(\widetilde{n},\widetilde{c},\widetilde{u},\widetilde{W})$ is actually a global martingale weak solution to the system \eqref{KS-SNS} under the new stochastic basis $(\widetilde{\Omega},\widetilde{\mathcal {F}},\{\widetilde{\mathcal {F}}_t\}_{t\geq 0},\widetilde{\mathbb{P}})$. By Fatou's Lemma, one can verify from the uniform bound \eqref{entropy} that the regularity of solution in \textbf{(b$_1$)}  hold true.

\textsf{Step 3.} When the sensitivity $\chi(\cdot)$ is equal to a positive constant, then the similar argument as  {(a$_2$)} of Theorem \ref{th1} implies that the pathwise uniqueness holds,  which together with the martingale theory in \textsf{Step 2} and the Yamada-Watanabe Theorem (cf. \cites{yamada1971uniqueness,watanabe1971uniqueness} imply the existence and uniqueness of global pathwise weak solutions to \eqref{KS-SNS}. The proof of Theorem \ref{th2} is completed. \qed

\section*{Data availability}

No data was used for the research described in the article.

\section*{Conflict of interest statement}

The authors declared that they have no conflicts of interest to this work.

\section*{Acknowledgements}

The authors would like to warmly thank the anonymous referee for his/her constructive comments and suggestions which helped to improve the quality of this article. This work was partially supported by the National Natural Science Foundation of China (Grant No. 12231008), and the National Key Research and Development Program of China (Grant No.  2023YFC2206100).

\bibliographystyle{amsrefs}
\bibliography{2dSCNS2}

\end{document}